\documentclass{article}

\usepackage{arxiv}

\usepackage[utf8]{inputenc} 
\usepackage[T1]{fontenc}    
\usepackage{hyperref}       
\usepackage{url}            
\usepackage{booktabs}       
\usepackage{amsfonts}       
\usepackage{amsmath}
\usepackage{amsthm}
\usepackage{amssymb}
\usepackage{nicefrac}       
\usepackage{float}          
\usepackage{microtype}      
\usepackage{lipsum}
\usepackage{graphicx}
\usepackage{subcaption}
\usepackage[ruled]{algorithm2e} 
\graphicspath{ {./images/} }

\newtheorem{theo}{Theorem}[section]

\newtheorem{lemma}{Lemma}[section]
\newtheorem{cor}{Corollary}[section]
\newtheorem{defin}{Definition}[section]
\newtheorem{rem}{Remark}[section]

\renewcommand{\u}{\mathbf{u}}
\renewcommand{\P}[1]{\mathbb{P}_k(#1)}

\newcommand{\K}{\mathbf{K}_h}
\newcommand{\M}{\hat{\mathbf{M}}_h}
\newcommand{\scalmassproj}{\Pi^{0}_{k,E}}
\newcommand{\scalenergyproj}{\Pi^{\nabla}_{k,E}}
\newcommand{\localvespace}{V^E_{h,k}}

\newcommand{\integral}[1]{\int \limits_{#1}}
\newcommand{\numdofs}{N^E_{\text{dof}}}
\newcommand{\lebesguespace}[1]{L^2(#1)}

\title{Mass-Lumped Virtual Element Method with Strong Stability-Preserving Runge-Kutta Time Stepping for Two-Dimensional Parabolic Problems}

\author{
 Paulo Akira F. Enabe \\
    Escola Politénica\\
    University of São Paulo\\
    Department of Structural and Geotechnical Engineering\\
  \texttt{paulo.enabe@usp.br} \\
   \And
 Rodrigo Provasi \\
    Escola Politénica\\
    University of São Paulo\\
    Department of Structural and Geotechnical Engineering\\
  \texttt{provasi@usp.br} \\
}

\begin{document}
\maketitle
\begin{abstract}
    This paper presents a mass-lumped Virtual Element Method (VEM) with explicit Strong Stability-Preserving Runge--Kutta (SSP-RK) time integration for two-dimensional parabolic problems on general polygonal meshes. A diagonal mass matrix is constructed via row-sum operations combined with flooring to ensure uniform positivity. Stabilization terms vanish identically under row summation, so the lumped weights derive solely from the $L^2$ projector and are computable through a small polynomial system at cost $\mathcal{O}(N_k^3)$ per element. The resulting lumped bilinear form satisfies $L^2$-equivalence with constants independent of the number of element edges, yielding a symmetric positive definite discrete inner product. A mesh-robust spectral estimate is established, showing that the largest eigenvalue of the discrete diffusion operator scales like $h^{-2}$, with constants depending only on the space dimension, polynomial degree, and mesh regularity. This yields the classical diffusion-type CFL condition $\Delta t=\mathcal{O}(h^2)$ for forward Euler stability and extends to higher-order SSP-RK schemes, ensuring the preservation of stability properties inherited from the forward Euler step. Numerical experiments on distorted quadrilateral, serendipity, and Voronoi meshes validate the theoretical predictions: for $k=1$, the lumped VEM attains optimal convergence rates, namely $\mathcal{O}(h)$ in the $H^1$-seminorm and $\mathcal{O}(h^2)$ in the $L^2$-norm, without degradation due to mesh distortion or diagonal mass approximation, while the SSP-RK methods remain stable under the predicted $\Delta t\propto h^2$ scaling. Additional tests on accuracy versus efficiency and on heterogeneous anisotropic diffusion further illustrate the practical competitiveness and robustness of the proposed formulation.
\end{abstract}


\section{Introduction}
\label{sec:introduction}

The Virtual Element Method (VEM) extends the finite element methodology to general polygonal/polyhedral meshes by defining local spaces that contain $\mathbb{P}_k(E)$ together with “virtual” functions, and by constructing computable bilinear forms from degrees of freedom (DOFs) and polynomial projections rather than explicit shape functions. On the scalar Laplace model \cite{beirao2013vem}, the extended space $\overline{V}^E_{h,k}$ (with polynomial traces and $\Delta v \in \mathbb{P}_{k-2}(E)$) and the virtual space $V^E_{h,k}$ are specified, the DOFs are prescribed, and the discrete stiffness is built as a consistency term based on the energy projector $\scalenergyproj$ plus a stabilization acting on $\ker(\scalenergyproj)$, yielding mesh--independent stability and optimal convergence. The “Hitchhiker’s Guide” \cite{veiga2014hitchhiker} details how to compute $\scalenergyproj$ from DOFs and assemble the consistency--stability split, and—crucially for transient or reaction terms—how to compute the full $L^2$ projector $\scalmassproj$ using only DOFs, enabling accurate mass matrices and load terms on arbitrary polygons/polyhedra.

VEM has seen broad use across mechanics and multiphysics. In solid mechanics, applications include linear and nonlinear elasticity, finite deformations, and elasto-plasticity \cite{beirao2015elastic},\cite{artioli2017arbitrary},\cite{hudobivnik2019plasticity}, \cite{wriggers2017efficient}, \cite{vanhuyssteen2020isotropic}. In fluid and transport problems, studies address topology optimization of non-Newtonian flows, two-phase flow in porous media, coupled Poisson-Nernst-Planck/Navier-Stokes systems, and hemivariational formulations for stationary Navier-Stokes \cite{suarez2022topology},\cite{berrone2022twophase},\cite{dehghan2023poisson},\cite{han2021hemivariational}. Recent developments include stabilization-free and serendipity variants that remove or redesign the stabilization term while retaining accuracy \cite{chen2023stabilization},\cite{chen2023serendipity}.

The work in \cite{vacca2015parabolic} is the foundational reference that brings conforming VEM to time-dependent diffusion on polygonal meshes. It sets up the augmented/enhanced local spaces and shows how to compute, purely from DOFs, the energy and $L^2$ projectors that make the stiffness and mass forms both $k$-consistent and stably assembled on arbitrary polygons. The paper develops the semidiscrete parabolic formulation, proves optimal convergence under standard star-shaped regularity, derives fully discrete error bounds for implicit stepping, and documents behavior through comprehensive tests (including a reduced-mass variant and conservation for Neumann problems). For our purposes, this work provides the canonical projector-based mass and stiffness construction and the error/stability framework we build upon, serving as the baseline we extend with mass lumping and SSP Runge--Kutta explicit time stepping together with mesh-robust CFL bounds. The extension to the nonconforming setting is addressed in \cite{zhao2019nonconforming}, where semi-discrete and fully discrete formulations are analyzed and optimal error estimates are established in the $L^2$ norm and broken $H^1$ seminorm. The authors in \cite{adak2020nonlocal} further extends the parabolic VEM framework to nonlinear nonlocal diffusion problems on general polygonal and polyhedral meshes, establishing well-posed semi-discrete and fully discrete schemes, deriving \textit{a priori} error estimates in the $L^2$ and $H^1$ norms, and introducing linearized and reformulated nonlinear solvers to reduce the computational cost associated with the nonlocal term.

The article \cite{gomez2022serendipity} considers semilinear parabolic problems on polygonal meshes from the perspective of the high-order interpolatory Serendipity Virtual Element Method. There, the emphasis is on reducing the number of degrees of freedom and on approximating the nonlinear term through computable interpolatory or quasi-interpolatory constructions in the serendipity space, leading to improved efficiency while retaining optimal semi-discrete error estimates in the $L^2$ norm. When combined with a second-order Strang operator splitting time discretization, the resulting fully discrete scheme is shown to be particularly effective and naturally suited to parallel implementations.

The article \cite{arrutselvi2022nonlocal} further broadens the time-dependent VEM literature by considering a coupled nonlocal parabolic system on polygonal meshes. In that setting, the virtual element framework is used to discretize both the diffusion operator and the nonlocal terms through computable projection-based forms, and the analysis covers semi-discrete, fully discrete, and linearized formulations. The paper derives error estimates in the $L^2$ and $H^1$ norms and introduces an equivalent reformulation together with a linearized scheme in order to retain a simpler Jacobian structure and reduce the computational complexity associated with the coupled nonlocal coefficients. 

More recently, \cite{wang2025bdf2} studies strongly nonlinear parabolic problems on general polygonal meshes within the VEM framework, using a linearized second-order backward differentiation formula for the time discretization. The analysis combines virtual element projections with a temporal--spatial error splitting technique to establish optimal convergence and unconditional $L^\infty$ boundedness of the numerical solution, thereby obtaining error estimates without any time-step restriction. In a related but distinct direction, \cite{dar2024timefractional} considers time-fractional parabolic equations over distorted polygonal meshes, further illustrating the flexibility of VEM for transient diffusion models beyond the classical integer-order setting. Taken together, these contributions confirm that recent advances in time-dependent VEM have mainly developed along the direction of implicit, linearized, or fractional formulations for increasingly complex models, rather than explicit mass-lumped schemes.

The $L^2$ projectors were originally introduced in \cite{ahmad2013projectors}. The authors introduced the “equivalent projectors” framework for VEM, showing that with a mild modification of the local space the full $L^2$ projector onto $\mathbb{P}_k$ becomes computable solely from the standard degrees of freedom together with the energy projector. The construction enforces matching of higher-order moments between a virtual function and its energy projection, thereby revealing the required polynomial information without explicit shape functions. As a result, mass terms and load integrals can be assembled in a DOF-only, polynomially consistent manner on arbitrary polygons/polyhedra—an essential ingredient for parabolic problems—yielding accurate, stable element mass matrices and straightforward treatment of reaction and source terms, with a formulation that extends naturally to three dimensions and higher orders.

Explicit time stepping is attractive in the VEM setting because it avoids global linear solves at every step—often the dominant cost and the main algorithmic bottleneck on general polygonal meshes. Updates are purely local once the spatial operator is applied, which makes them easy to parallelize and well suited to accelerators (GPUs). They are also simple to implement for variable coefficients and nonlinearities (no Jacobians or changing preconditioners), and they interact cleanly with limiters and strong-stability notions (via SSP-RK) so that discrete maximum principles, positivity, or energy decay can be enforced under a verifiable stepsize restriction. For many diffusion-type problems—especially short-to-moderate time horizons, repeated solves, or as a building block in operator splitting—these properties yield lower time-to-solution and a more predictable performance profile than implicit solves on challenging polytopal meshes. See for example \cite{alexiades1996super} and \cite{onate2024explicit}.

A related development, particularly relevant from the viewpoint of qualitative stability, is the positivity-preserving conservative VEM-based scheme proposed in \cite{yang2022vem} for radiation diffusion problems on generalized polyhedral meshes. In that work, the emphasis is not on explicit time integration, but on the construction of a discretization that preserves positivity and conservation properties on highly general three-dimensional meshes. This is closely connected to the present perspective, since it highlights that, in polygonal and polyhedral methods, the preservation of physically meaningful bounds is a central design issue rather than a secondary numerical detail. A related contribution in the broader context of qualitative stability on polygonal meshes is \cite{cao2025edge}, where edge-averaged virtual element schemes are developed for convection-diffusion and convection-dominated problems, with particular emphasis on monotonicity, M-matrix structure, and the suppression of nonphysical oscillations.

Mass lumping is the key to making those explicit schemes practical and robust. A consistent VEM mass matrix is sparse but not diagonal; inverting it at each step would reintroduce global solves. Replacing it with a carefully designed diagonal (lumped) matrix preserves polynomial consistency where needed, ensures positivity of the discrete norm, and dramatically reduces cost—each time step becomes a sequence of scaled vector operations. On general polygons, naïve row-sums may fail to be uniformly positive for higher orders; enforcing a small floor on the lumped weights remedies this and yields an $L^2$-equivalent discrete inner product with mesh-independent constants. The result is a symmetric positive definite, diagonally dominant mass operator that enables clean SSP arguments, delivers mesh-robust CFL bounds, and remains stable even on highly irregular polytopal meshes—exactly the regime where VEM is most compelling.

Regarding mass lump formulation and the Virtual Element Method, the work in \cite{park2020elastodynamic} develops an explicit-time VEM for 2D/3D elastodynamics. It builds the element stiffness as a consistency part (from an $L^2$ projector acting on gradients) plus a stabilization part (either diagonal or scalar), and studies how the stabilization choice affects accuracy on convex and nonconvex meshes. It also analyzes stability limits for explicit updates via global and local eigenvalue estimates and via an “effective element length,” and it compares two lumped-mass constructions—diagonal scaling and a row-sum technique—showing, among other points, that the row-sum approach does not guarantee positive nodal masses while diagonal scaling can yield larger critical time steps on nonconvex meshes without degrading accuracy. The study includes a B-bar formulation for nearly incompressible solids and documents convergence on a variety of polygonal/polyhedral meshes. 

This authors in \cite{park2020elastodynamic} do provide a VEM lumped-mass formulation, but it is aimed at hyperbolic elastodynamic problems and central-difference-type explicit updates, not parabolic diffusion equations. By contrast, the present work targets parabolic operators and uses SSP Runge-Kutta time stepping with a mass-lumped VEM to obtain positivity/monotonicity and diffusion-type CFL bounds. The elastodynamic recipe’s stability analysis is wave-speed/eigenfrequency based and its lumping options are diagonal scaling vs. row-sum; our formulation instead focuses on diffusion spectra and SSP-compatible lumping tailored to parabolic dynamics.  

The combination of a projector-based VEM discretization and a diagonal (lumped) mass matrix sets the stage for provably stable explicit time integration. Once the $L^2$ inner product is replaced by its lumped counterpart, forward Euler becomes contractive in the discrete energy $|\cdot|_{\M}$ under a diffusion-type CFL bound tied to the spectral radius of $(\M)^{-1}\K$. Strong-stability-preserving Runge--Kutta (SSP-RK) schemes in the Shu--Osher framework then lift this first-order monotonicity to higher order: any stability (e.g., energy decay, positivity) guaranteed by forward Euler is inherited by the multi-stage integrator with a computable step-size scaling. Section~\ref{sec:estimates} formalizes this mechanism for the VEM semi-discretization, specifies practical SSP-RK choices, and couples them with mesh-robust spectral estimates for $(\M)^{-1}\K$ so that the resulting CFL restriction remains of order $h^2$ and independent of the number of element edges; see also \cite{shu1988eno}, \cite{gottlieb1998ssp}, \cite{califano2018ssp}, \cite{izzo2022ssp}.

Section~\ref{sec:vem} recalls the VEM ingredients used here: augmented/enhanced spaces, degrees of freedom, and the computable energy and $L^2$ projectors, together with the consistency--stability split for stiffness and mass forms. Section~\ref{sec:lumped_mass} introduces the lumped-mass variant, shows that stabilization terms vanish under row summation, proves positivity and $L^2$-equivalence for the resulting diagonal inner product with constants independent of the number of element edges, and provides a degree-of-freedom-based algorithm for computing the lumped weights on each polygon. Section~\ref{sec:estimates} develops the explicit SSP-RK time discretization for the lumped VEM system, deriving forward Euler and SSP stability estimates together with a mesh-robust spectral bound leading to the classical diffusion CFL condition. Section~\ref{sec:numerical} presents numerical experiments assessing convergence, stability, accuracy versus efficiency, and robustness in heterogeneous anisotropic diffusion on both structured and highly irregular polygonal meshes. Finally, Section~\ref{sec:conclusion} summarizes the main findings and discusses possible extensions, including variable-coefficient problems, source terms, and coupled multiphysics models.

The main contributions of this paper are:
\begin{itemize}
    \item \textbf{Explicit VEM framework:} Energy-stable explicit scheme for parabolic partial differential equation on polygonal meshes, with stability measured in the discrete energy norm induced by $\M$ and a stepsize restriction linked to the spectral quantity $\lambda_{\max}\!\big((\M)^{-1}\K\big)$.
    \item \textbf{Mass-lumping formulation:} Row-sum, diagonally scaled mass matrix that is $k$-consistent; stabilization vanishes in row-sums, and the use of floored weights (cf. \eqref{eq:floored_weights}) yields uniform positivity and $L^2$-equivalence, making the lumped bilinear form symmetric positive definite.
    \item \textbf{Spectral bound and CFL estimate:} Proof of a mesh-robust upper bound
    \begin{equation}
        \nonumber
        \lambda_{\max}\!\big((\M)^{-1}\K\big)\ \le\ \tfrac{C_{\mathrm{inv}}^2}{\hat{\beta}_*}\,h_{\min}^{-2},
    \end{equation}
    with constants depending only on $d$, $k$, and the chunkiness parameter, and independent of the number of faces/edges; this delivers a CFL rule $\Delta t=\mathcal{O}(h^2)$ for explicit diffusion with lumped $\M$.
    \item \textbf{SSP time integration for VEM:} Adoption of SSP-RK methods from \cite{izzo2022ssp} (third-order SSP-RK3 and fourth-order SSP-RK(5,4)), using the GLM/monotonicity framework to guarantee preservation of $\Phi(\u)=\|\u\|_{\M}$ under
    \begin{equation}
        \nonumber
        0<\Delta t \ \le\ C_{\mathrm{SSP}}\,\Delta t_{\mathrm{FE}},
    \end{equation}
    avoiding Shu-Osher implementation issues.
    \item \textbf{Computable and efficient construction:} Element-wise computation of lumped weights via a small polynomial system
    \begin{equation}
        \nonumber
        \mathbf{G}^E\,\mathbf{w}^E=\mathbf{c}^E,
    \end{equation}
    and DOF matrix contractions, with cost $\mathcal{O}(N_k^3)$ per element and $N_k \ll N_{\text{dof}}^E$, enabling straightforward incorporation in existing VEM codes.
    \item \textbf{Scope and generality:} Analysis presented for the scalar diffusion problem; by componentwise action of the diffusion operator and VEM projectors, the vector-valued case decouples into identical scalar blocks, so the spectral and stability results carry over componentwise.
\end{itemize}

\section{Virtual Element Method}
\label{sec:vem}

The problem considered in this paper is the time-dependent scalar diffusion equation on a Lipschitz domain $\Omega \in \mathbb{R}^2$:
\begin{equation}\label{eq:strong_formulation}
  \left\{ \begin{aligned}
    \frac{\partial u}{\partial t} - \Delta u &= f && \text{in } \Omega \times (0,T),\\
    u &= 0 && \text{on } \partial \Omega \times (0,T),\\
    u(\cdot, 0) &= u_0 && \text{in } \Omega,
  \end{aligned} \right. ,
\end{equation}
where $f$ is the source term, $u_0 \in L^2(\Omega)$ is the initial condition, and $u$ is the solution field. Here $\Omega\subset\mathbb{R}^2$ is a bounded Lipschitz polygonal domain with boundary $\partial\Omega$, and $T>0$. Assume $f\in L^2\big(0,T; L^2(\Omega)\big)$ and $u_0\in L^2(\Omega)$. 

The weak formulation seeks $u \in L^2(0,T;H^1_0(\Omega))$ with $\frac{\partial u}{\partial t} \in L^2(0,T;H^{-1}(\Omega))$ such that
\begin{equation}\label{eq:weak_formulation}
  \left\langle  \frac{\partial u}{\partial t}, v \right\rangle + a(u(t), v) = \langle f(t), v \rangle,
\end{equation}
for all $v \in H^1_0(\Omega)$ and for almost every $t\in (0,T)$, with $u(0)=u_0$, where $\langle \cdot, \cdot \rangle$ denotes the $H^{-1}(\Omega)$--$H^1_0(\Omega)$ duality pairing, and
\begin{equation}
    \nonumber
  a(u, v) = \int \limits_\Omega \nabla u \cdot \nabla v \, d\Omega.
\end{equation}
Under these assumptions, $u \in C\big([0,T]; L^2(\Omega)\big)$, so the initial condition $u(0)=u_0$ is meaningful.

Some assumptions regarding the domain and the mesh are necessary here. Let $\Omega \in \mathbb{R}^d$, with $d=2,3$, be a bounded polygonal/polyhedral Lipschitz domain. Consider a polygonal/polyhedral mesh $\mathcal{T}_h$ of $\Omega$, where each element $E \in \mathcal{T}_h$ has diameter $h_E$, with
\begin{equation}
    \nonumber
    h = \max \limits_{E \in \mathcal{T}_h} h_E, \; h_{\min} = \min \limits_{E \in \mathcal{T}_h} h_E. 
\end{equation}
Assume the mesh satisfies the following regularity conditions:
\begin{itemize}
    \item Each element $E \in \mathcal{T}_h$ is star-shaped with respect to a ball $\mathcal{B}_E$ of radius $\rho_E \geq C_\gamma h_E$, where $C_\gamma > 0$ is a constant independent of the mesh;
    \item The number of edges (or faces) per element is uniformly bounded, and edge (or face) length are comparable to $h_E$ (up to a constant independent of the mesh);
    \item The mesh is quasi-uniform in the sense that $h$ and $h_{\min}$ are bounded.
\end{itemize}

The virtual element formulation presented next is based on \cite{beirao2013vem}, \cite{veiga2014hitchhiker} and \cite{vacca2015parabolic}.

\subsection{Extended virtual element space}
\label{sec:extended_ve_space}

Consider the arbitrary decomposition $\mathcal{T}_h$ in simple polygons $E$ where $h$ is the discretization parameter (e.g., the maximum polygonal diameter of all elements). Also, let $k \in \mathbb{N}$, be the order of the method that is related to the degree of the polynomial projections explored later in this text. The augmented virtual element space is given by:
\begin{equation}\label{eq:augmented_virtual_space}
    \overline{V}^E_{h,k} = \left\{ v \in H^1(E) \cap C^0(E)  :v|_e \in \mathbb{P}_k(e), \forall e\in\partial E, \Delta v \in \mathbb{P}_k(E)  \right\},
\end{equation}
where $\{ e \}_{e \in \partial E}$ indicates the set of edges of the polygon $E$, and $\mathbb{P}_k(E)$ is the polynomial space of polynomials of degree at most $k$ with dimension given by:
\begin{equation}\label{eq:polynomial_space_dim}
    N_k = \frac{(k+1)(k+2)}{2}.
\end{equation}
In this space, $v \in H^1(E) \cap C^0(E)$, i.e. functions $v$ in this space must belong to $H^1(E)$, the Sobolev space of functions with square-integrable first derivative over $E$. They must also be continuous over $E$, ensuring that each $v$ has no discontinuities over $E$. This space contains functions satisfying boundary and Laplace conditions. However, it is too large for direct numerical computations because it includes functions that are defined implicitly and cannot be directly evaluated. It is the theoretical foundation to define the actual virtual element space $V^E_{h,k}$ which includes additional constraints and projections. 

The functions in both the extended and the virtual element space are computed implicitly using two essential elements in the VEM theory: the degrees of freedom (DOF) and the projection operators. The DOF are defined as follows.

\begin{defin} \label{def:degrees of freedom}
    The set of degrees of freedom functionals $\{ \chi_i^E \}^{N_{dof}^E}_{i=1}$, with $\chi_i^E: \overline{V}^{E}_{h,k}\!\longrightarrow\!\mathbb{R}$, is a collection of linear functionals that uniquely determine any function in $\overline{V}_{h,k}^E$, where $N_{dof}^E$ is the total number of degrees of freedom associated to the element $E$ given by:
    \begin{equation}\label{eq:ndofs}
        N^E_{\text{dof}} = N^E_V + N^E_V (k - 1) + \frac{(k-1)k}{2},
    \end{equation}
    where $N^E_V$ is the number of vertices of polygon $E$. Each $\chi_i^E$ is associated with a specific DOF:
    \begin{enumerate}
        \item (Vertex DOF functional) For all $k \geq 1$,
        \begin{equation}\label{eq:dof_vertex}
            \chi^E_{V_i}(u) = u(V_i),
        \end{equation}
        where $V_i$ is a vertex of $E$.
        \item (Edge moment DOF functional) For $k \geq 2$,
        \begin{equation}\label{eq:dof_edge}
            \chi^E_{e,j} (u) = \frac{1}{|e|} \int \limits_e u(s) \hat{L}_j (s) ds,
        \end{equation}
        with $j=0,2,...,k-2$, where $e \in \partial E$ is an edge, $|e|$ is the length of the edge, and $\hat{L}_j$ is the orthonormal Legendre polynomial of degree $j$ on $e$ mapped to $[-1,1]$.
        \item (Interior moment DOF functional) For $k \geq 2$,
        \begin{equation}\label{eq:dof_moment}
            \chi^E_{\alpha, i}(u) = \frac{1}{|E|} \int \limits_E u_i m_{\boldsymbol\alpha} dE
        \end{equation}
        with $i=1,2$ and $m_{\boldsymbol\alpha}$ the scaled monomial of total degree $|\boldsymbol\alpha|\le k-2$.
    \end{enumerate}
\end{defin}

With the degrees of freedom already specified, it is possible to introduce the local canonical basis functions $\{\psi_i\}_{i=1}^{N_E}\subset \overline{V_{h,k}^E}$. Each $\psi_i$ is uniquely associated with a DOF functional
$\chi_i : \overline{V}_{h,k}^E\!\longrightarrow\!\mathbb{R}$ and is characterized by the unisolvence condition
\begin{equation}
    \nonumber
    \chi_j(\psi_i)=\delta_{ij}
    \quad\text{for every } j=1,\dots,N^E_{\text{dof}}\;
\end{equation}
where $\delta_{ij}$ denotes the Kronecker delta. For clarity, the action of $\chi_j$ in the three DOF families can be written explicitly:
\begin{equation}
    \nonumber
    \psi_i(V_j)=\delta_{ij},\qquad
    \frac{1}{|e|}\int_{e}\psi_i\,\widehat L_q\,ds=\delta_{ij},\qquad
    \frac{1}{|E|}\int_{E}\psi_i\,m_{\boldsymbol\alpha}\,dE=\delta_{ij},
\end{equation}
These relations guarantee that the set $\{\psi_i\}$ forms a basis of $\overline{V_{h,k}^E}$ and that the DOF map $\,v_h\longmapsto\bigl(\chi_1(v_h),\dots,\chi_{N_E}(v_h)\bigr)$ is bijective.

Regarding the monomials mentioned above, they form a basis to the polynomial space $\mathbb{P}_k(E)$. In particular, the scaled monomial for two-dimensions is given by:
\begin{equation}
    \nonumber
    m_{\boldsymbol{\alpha}}(x,y) = \left( \frac{x-x_c}{h_E} \right)^{\alpha_1} \left( \frac{y-y_c}{h_E} \right)^{\alpha_2}, \quad \alpha_1, \alpha_2\in \mathbb{N}, \; \alpha_1 + \alpha_2 \leq k,
\end{equation}
where $(x_c, y_c)$ is the centroid and $h_E$ is the polygonal diameter. It is worth mentioning that, the monomial are not approximating polynomials in the sense of fitting data. They are the exact basis for $\mathbb{P}_k(E)$, used to express any polynomial of degree up to $k$.

\subsection{Energy projection}
\label{sec:energy_projection}

The energy projection operator $\Pi^{\nabla}_{k,E}\;:\;\overline{V}^{E}_{h,k}\;\longrightarrow\;\mathbb{P}_{k}(E)$ maps a function $v \in \overline{V}_{h,k}^E$ onto the polynomial space $\mathbb{P}_k(E)$, ensuring certain properties of $v$ are preserved. The projection $\Pi^{\nabla}_{k,E}$ is defined by the following two conditions:
\begin{enumerate}
    \item \textbf{Gradient orthogonality}: For all $q\in \mathbb{P}_k(E)$:
    \begin{equation}\label{eq:proj_grad_ortho}
        \int_{E}\nabla\bigl(\Pi^{\nabla}_{k,E} v\bigr)\cdot\nabla q\,dE=\int_{E}\nabla v\cdot\nabla q\,dE.
    \end{equation}
    Equation (\ref{eq:proj_grad_ortho}) can be written as:
    \begin{equation}\label{eq:proj_grad_ortho_v2}
        a^E(\Pi^{\nabla}_{k,E} v, q) = a^E(v, q),
    \end{equation}
    where $a^E(\cdot, \cdot)$ is the bilinear form given by:
    \begin{equation}\label{eq:bilinear_form}
        a^E(u, v) = \int \limits_E \nabla u \cdot \nabla v\,dE.
    \end{equation}
    This condition ensures that the gradient of $\Pi^{\nabla}_{k,E} v$ matches the gradient of $v$ when tested against polynomials in $\mathbb{P}_k(E)$.
    \item \textbf{Zero-mean condition}: The operator $\Pi^{\nabla}_{k,E}$ must also satisfy a consistency condition for its mean value:
    \begin{equation}\label{eq:zero_mean_condition}
        \hat{\Pi}_{k,E}^\nabla (\Pi^{\nabla}_{k,E} v) = \hat{\Pi}_{k,E}^\nabla (v),
    \end{equation}
    where $\hat{\Pi}_{k,E}^\nabla$ represents the mean value projection operator, which depends on the degree $k$. For $k = 1$,
    \begin{equation}
    \nonumber
        \hat{\Pi}_{k,E}^\nabla(v) = \frac{1}{N^E_v} \sum \limits^{N^E_v}_{i=1} v(V_i),
    \end{equation}
    where $N^E_v$ is the number of vertices of the polygonal element $E$. And, for $k > 1$,
    \begin{equation}
    \nonumber
        \hat{\Pi}_{k,E}^\nabla(v) = \frac{1}{|E|} \int \limits_E v\, dE.
    \end{equation}
    In particular, this condition ensures that the projection retains the same mean value as the original function $v$. 
\end{enumerate}

The projection operator holds two fundamental properties. If $v \in \mathbb{P}_k(E)$, then
\begin{equation}
    \nonumber
    \Pi^\nabla_{k,E} v = v.
\end{equation}
This ensures that polynomials of degree $\leq k$ are preserved exactly by the projection. The second property refers to the fact that the operator is stable in the sense that the norms of $v$ and $\Pi^\nabla_{k,E} v$ are computable:
\begin{equation}
    \nonumber
    \| \Pi^\nabla_{k,E} v \|_{L^2(\Omega)} \leq C \| \nabla v \|_{L^2(\Omega)},
\end{equation}
where $C$ is a constant independent of the mesh size $h_E$.

The accuracy of the energy projection operator $\Pi^\nabla_{k,E}$ is characterized by several approximation properties that are fundamental to the convergence analysis of the VEM. The most important property concerns the gradient approximation, which directly affects the energy norm and the overall accuracy of the method. For sufficiently smooth functions $v \in H^{k+1}(E)$, the energy projection satisfies:
\begin{equation}\label{eq:energy_proj_gradient_accuracy}
    \| \nabla(v - \Pi^\nabla_{k,E} v) \|_{L^2(E)} \leq C h_E^k |v|_{H^{k+1}(E)},
\end{equation}
where $C$ is a constant independent of the mesh size $h_E$ and $|v|_{H^{k+1}(E)}$ denotes the $(k+1)$-th order Sobolev seminorm. This property ensures optimal convergence in the energy norm, which is the natural norm for elliptic problems.

The $L^2$-approximation property of the energy projection is also crucial for the overall error analysis. Under the same regularity assumptions, the projection satisfies:
\begin{equation}\label{eq:energy_proj_l2_accuracy}
    \| v - \Pi^\nabla_{k,E} v \|_{L^2(E)} \leq C h_E^{k+1} |v|_{H^{k+1}(E)}.
\end{equation}
This $L^2$-approximation property is one order higher than the gradient approximation.

Also, the accuracy of the energy projection is intimately connected to the polynomial exactness property. Since $\Pi^\nabla_{k,E} p = p$ for all $p \in \mathbb{P}_k(E)$, the projection error vanishes exactly for polynomial functions of degree at most $k$. This property ensures that the method achieves optimal convergence rates when the exact solution has sufficient regularity.

Furthermore, the energy projection operator satisfies a crucial consistency property in the discrete bilinear form. For any $u, v \in \overline{V}_{h,k}^E$, the following identity holds:
\begin{equation}\label{eq:energy_proj_consistency}
    a^E(\Pi^\nabla_{k,E} u, \Pi^\nabla_{k,E} v) = a^E(u, \Pi^\nabla_{k,E} v) = a^E(\Pi^\nabla_{k,E} u, v).
\end{equation}
This consistency property is essential for proving the stability and convergence of the method, as it ensures that the discrete bilinear form inherits the coercivity and continuity properties of the continuous bilinear form.

The projection error can be bounded in terms of the local mesh geometry through the shape regularity assumption. For star-shaped polygons satisfying the usual VEM geometric assumptions, the constants in the approximation estimates (\ref{eq:energy_proj_gradient_accuracy}) and (\ref{eq:energy_proj_l2_accuracy}) depend only on the shape regularity parameter $\gamma^E = h_E/\rho_E$, where $\rho_E$ is the radius of the largest circle inscribed in $E$. This geometric dependence ensures that the projection accuracy is maintained under mesh refinement as long as the elements remain well-shaped.

The accuracy analysis extends to the treatment of boundary conditions through the degrees of freedom. The vertex and edge degrees of freedom provide sufficient information to accurately reconstruct the boundary values of $v$, while the interior moment degrees of freedom capture the interior behavior. The combination of these degrees of freedom ensures that the projection $\Pi^\nabla_{k,E} v$ approximates $v$ optimally in both the interior and on the boundary of the element.

The energy projection operator exhibits a superconvergence property for certain classes of problems. When the exact solution possesses additional regularity beyond $H^{k+1}(E)$, the projection error can achieve higher convergence rates in specific norms or at particular points. This superconvergence behavior is particularly relevant for post-processing techniques and error estimation procedures in adaptive mesh refinement algorithms.

\subsection{Virtual element space}
\label{sec:virtual_element_space}

With the projection operator properly defined, it is possible to introduce the virtual element space:
\begin{equation}\label{eq:ve_space}
    V_{h,k}^E = \left\{ v \in \overline{V}_{k,E} : \int \limits_E \left( \Pi^{\nabla, scal}_{k,E} v \right) q\;dE= \int \limits_E vq \; dE, \; \forall q \in \mathbb{P}_k/\mathbb{P}_{k-2}(E) \right\},
\end{equation}
where $\mathbb{P}_k(E)/\mathbb{P}_{k-2}(E)$ denotes the polynomials in $\mathbb{P}_k(E)$ that are orthogonal to all polynomials in $\mathbb{P}_{k-2}(E)$ with respect to the $L^2$-inner product such that
\begin{equation}
    \nonumber
    \int \limits_E pq dE = 0,
\end{equation}
for all $p \in \mathbb{P}_k/\mathbb{P}_{k-2}(E)$ and for all $q \in \mathbb{P}_k(E)$. Note that by the definition of virtual element space, it is known that $\mathbb{P}_k(E) \subseteq V^E_{k,h}$.

Intuitively, $\mathbb{P}_k/\mathbb{P}_{k-2}(E)$ represents the part of $\mathbb{P}_k(E)$ that is not already covered by $\mathbb{P}_{k-2}(E)$. In this way, using the direct sum, $\mathbb{P}_k(E)$ can be decomposed in two orthogonal subspaces:
\begin{equation}
    \nonumber
    \mathbb{P}_k(E) = \mathbb{P}_{k-2}(E) \oplus \mathbb{P}_k/\mathbb{P}_{k-2}(E).
\end{equation}
This decomposition allows the method to handle difference components of $w \in V^E_{h,k}$ separately: lower--degree components in $\mathbb{P}_{k-2}(E)$ that are typically associated with basic physical behavior (e.g., rigid body translations), and higher--degree components in $\mathbb{P}_k/\mathbb{P}_{k-2}(E)$ that captures additional fine-scale features (e.g., linear--varying strains).

\begin{rem}
    Using $\mathbb{P}_{k-1}(E)$ in $\mathbb{P}_k/\mathbb{P}_{k-1}(E)$ would lead to  neglect polynomial terms, leading to an information loss. The degree $k-1$ terms are critical for capturing the gradient behavior of $v$, especially in the stiffness matrix construction. By using $\mathbb{P}_k/\mathbb{P}_{k-2}(E)$, it is ensured that both the gradient-related terms (degree $k-1$) and the highest-order terms (degree $k$) are considered. For example, suppose $k=2$, then $\mathbb{P}_{2} = span \{ 1, x, y, x^2, xy, y^2 \}$. The space $\mathbb{P}_2/\mathbb{P}_0(E)$ includes the components orthogonal to the constants in $\mathbb{P}_0(E)$, such that $\mathbb{P}_2/\mathbb{P}_0(E)=span\{ x, y, x^2, xy,  y^2 \}$. On the other hand, space $\mathbb{P}_2/ \mathbb{P}_1 (E)$ includes only the degree $2$ components such that $\mathbb{P}_2/\mathbb{P}_1(E) = span \{ x^2, xy, y^2 \}$. The second choice would neglect the constant and linear terms.
\end{rem}

The constraint
\begin{equation}\label{eq:enhancement_condition}
    \int \limits_E \left( \Pi^{\nabla}_{k,E} v \right) q\;dE = \int \limits_E vq \; dE
\end{equation}
is called enhancement condition, and it is essential for the VEM formulation to ensure the projection $\Pi^\nabla_{k,E} v$ properly represents $w$ when tested against the higher degree polynomial components in $\mathbb{P}_{k}/\mathbb{P}_{k-1}(E)$. To illustrate its role in the method's theory, for simplicity, consider the scalar version of the bilinear form presented in (\ref{eq:bilinear_form}):
\begin{equation}\label{eq:continuous_bilinear_stiff}
    \nonumber
    a^E(u,v) = \int \limits_E \nabla u \nabla v dE.
\end{equation}
The gradient $\nabla u$ and $\nabla v$ are a first order derivative. So, if $u,v \in \mathbb{P}_k(E)$, then $\nabla u, \nabla v\in \mathbb{P}_{k-1}(E)$. As result, $a^E(u,v)$ depends only on the degree $k-1$ components of $u$ and $v$. However, this bilinear form does not directly involve the degree $k$ components of $u$ and $v$. Those components contribute to $u$ and $v$ themselves but do not affect their gradients. The missing degree $k$ components can be accounted by imposing the enhancement condition. This constraint properly tests the projection $\Pi^\nabla_{k,E} v$ against the higher degree components $q \in \mathbb{P}_k / \mathbb{P}_{k-2}(E)$, which are orthogonal to lower-degree components.

\subsection[L2-projection operator]{$L^2$-projection operator}
\label{sec:l2_projection}

The standard $L^2$-projection operator $\Pi^0_{k,E}: V_{h,k}^E \longrightarrow \mathbb{P}_k(E)$ maps $v \in V_{h,k}^E$ to a polynomial space while preserving the $L^2$-inner product structure. This projection operator is originally introduced in \cite{ahmad2013projectors}, and is defined as the unique polynomial $\Pi^0_{k,E} v \in \mathbb{P}_k(E)$ that satisfies the orthogonality condition:
\begin{equation}\label{eq:l2_projection_def}
    \int \limits_E \left( \Pi^0_{k,E} v \right) q \, dE = \int \limits_E v q \, dE, \quad \forall q \in \mathbb{P}_k(E).
\end{equation}
This definition ensures that $\Pi^0_{k,E} v$ approximates $v$ in the sense of the $L^2$-norm.

The $L^2$-projection operator possesses several fundamental properties that are essential for the VEM formulation. First, it satisfies the polynomial preservation property: if $v \in \mathbb{P}_k(E)$, then
\begin{equation}\label{eq:l2_poly_preservation}
    \Pi^0_{k,E} v = v.
\end{equation}
This ensures that polynomials of degree at most $k$ are preserved exactly by the projection. Second, the operator is stable in the $L^2$-norm:
\begin{equation}\label{eq:l2_stability}
    \| \Pi^0_{k,E} v \|_{L^2(E)} \leq \| v \|_{L^2(E)},
\end{equation}
where the inequality reflects the fact that the projection minimizes the $L^2$-distance to the polynomial space.

To compute $\Pi^0_{k,E} v$, the projection is expressed in terms of the monomial basis $\{m_{\boldsymbol{\alpha}}\}$ of $\mathbb{P}_k(E)$:
\begin{equation}\label{eq:l2_expansion}
    \Pi^0_{k,E} v = \sum_{|\boldsymbol{\alpha}| \leq k} c_{\boldsymbol{\alpha}} m_{\boldsymbol{\alpha}},
\end{equation}
where $c_{\boldsymbol{\alpha}} \in \mathbb{R}$ are the coefficients to be determined. Substituting this expansion into the orthogonality condition (\ref{eq:l2_projection_def}) yields the linear system:
\begin{equation}\label{eq:l2_linear_system}
    \sum_{|\boldsymbol{\beta}| \leq k} c_{\boldsymbol{\beta}} \int_E m_{\boldsymbol{\beta}} m_{\boldsymbol{\alpha}} \, dE = \int_E v\, m_{\boldsymbol{\alpha}} \, dE, \quad \forall |\boldsymbol{\alpha}| \leq k.
\end{equation}
The matrix $\mathbf{M}$ with entries $M_{\boldsymbol{\alpha}, \boldsymbol{\beta}} = \int_E m_{\boldsymbol{\alpha}} m_{\boldsymbol{\beta}} \, dE$ is the mass matrix for the monomial basis, which is symmetric and positive definite.

The right-hand side of system (\ref{eq:l2_linear_system}) must be computed using the available degrees of freedom, since functions in $V_{h,k}^E$ are not explicitly known. For the interior moment DOFs defined in equation (\ref{eq:dof_moment}), the computation is straightforward:
\begin{equation}\label{eq:l2_interior_moments}
    \int_E v\, m_{\boldsymbol{\alpha}} \, dE = |E| \chi^E_{\boldsymbol{\alpha}}(v), \quad |\boldsymbol{\alpha}| \leq k-2.
\end{equation}
For monomials of degree $k-1$ and $k$, the integrals must be computed using integration by parts and the boundary DOFs. Specifically, for $|\boldsymbol{\alpha}| = k-1$ or $|\boldsymbol{\alpha}| = k$, the divergence theorem provides:
\begin{equation}\label{eq:l2_boundary_contribution}
    \int_E v\, m_{\boldsymbol{\alpha}} \, dE = \int_E v\, \nabla \left( \frac{m_{\boldsymbol{\alpha}+\mathbf{e}_i}}{(\boldsymbol{\alpha}+\mathbf{e}_i) \cdot \mathbf{e}_i} \right) dE - \int_{\partial E} v\, \mathbf{n} \cdot \frac{\nabla m_{\boldsymbol{\alpha}+\mathbf{e}_i}}{(\boldsymbol{\alpha}+\mathbf{e}_i) \cdot \mathbf{e}_i} \, ds,
\end{equation}
where $\mathbf{e}_i$ denotes the $i$-th canonical basis vector and the boundary integral is evaluated using the vertex and edge DOFs.

The $L^2$-projection operator plays a crucial role in the VEM formulation through several mechanisms. First, it enables the construction of computable approximations to integrals involving functions in the virtual element space. Since functions in $V_{h,k}^E$ are not explicitly known, integrals of the form $\int_E v w \, dE$ cannot be computed directly. However, the approximation
\begin{equation}\label{eq:l2_integral_approximation}
    \int_E v w \, dE \approx \int_E \Pi^0_{k,E} v \; \Pi^0_{k,E} w \, dE
\end{equation}
provides a computable alternative that maintains the accuracy of the method.

Second, the $L^2$-projection operator is essential for the construction of the mass matrix in time-dependent problems. In the context of parabolic problems, the mass matrix entries are given by:
\begin{equation}\label{eq:mass_matrix_entries}
    M_{ij} = \int_E \psi_i \, \psi_j \, dE \approx \int_E \Pi^0_{k,E} \psi_i \, \Pi^0_{k,E} \psi_j \, dE,
\end{equation}
where $\{\psi_i\}$ are the basis functions of the virtual element space.

Third, the operator is fundamental for the treatment of body forces and distributed loads in the variational formulation. The load vector components are computed as:
\begin{equation}\label{eq:load_vector}
    F_i = \int_E f \, \psi_i \, dE \approx \int_E f \, \Pi^0_{k,E} \psi_i \, dE,
\end{equation}
where $f$ represents the body force field.

The relationship between the $L^2$-projection operator and the energy projection operator $\Pi^{\nabla}_{k,E}$ is particularly important. While $\Pi^{\nabla}_{k,E}$ preserves the gradient structure and is used for the stiffness matrix construction, $\Pi^0_{k,E}$ preserves the $L^2$-inner product and is used for mass matrix and load vector computations. These two operators are generally distinct, and their combined use ensures that the VEM formulation maintains both the energy structure and the $L^2$-approximation properties of the continuous problem.

The accuracy of the $L^2$-projection operator is characterized by the approximation property:
\begin{equation}\label{eq:l2_approximation_property}
    \| v - \Pi^0_{k,E} v \|_{L^2(E)} \leq C h_E^{k+1} |v|_{H^{k+1}(E)},
\end{equation}
where $C$ is a constant independent of the mesh size $h_E$ and $|v|_{H^{k+1}(E)}$ denotes the $(k+1)$-th order Sobolev seminorm. This property ensures that the projection error decreases optimally with mesh refinement, which is essential for the overall convergence of the VEM.

Finally, the $L^2$-projection operator satisfies a commuting property with the energy projection operator for polynomial functions:
\begin{equation}\label{eq:commuting_property}
    \Pi^0_{k,E} \Pi^{\nabla}_{k,E} p = \Pi^{\nabla}_{k,E} p = p, \quad \forall p \in \mathbb{P}_k(E).
\end{equation}
This property ensures consistency between the two projection operators when applied to polynomial functions, which is crucial for maintaining the exactness of the method on polynomial solutions.

\subsection{Discrete bilinear forms}
\label{sec:discrete_bilinear}

Two computable and discrete bilinear forms are defined. The first form is related to the gradient term, with $a^E_h: V^E_{h,k} \times V^E_{h,k}\longrightarrow \mathbb{R}$:
\begin{equation}
    \nonumber
    a^E_h(\cdot, \cdot) \sim  a^E(\cdot,  \cdot).
\end{equation}
It is constructed using the gradient projection operator $\Pi^\nabla_{k,E}$, ensuring it is computable using only the degrees of freedom. This bilinear is mostly associated with the computation of the stiffness matrix of the discrete system, and it governs the gradients of the virtual element basis functions locally on each element of $E$. Here, "$\sim$" means uniform equivalence (refers to stability bounds regarding coercivity and continuity with mesh independent constants).

The second form is associated with the computation of the mass matrix, with $m^E_h: V^E_{h,k} \times V^E_{h,k}\longrightarrow \mathbb{R}$:
\begin{equation}
    \nonumber
    m^E_h(\cdot, \cdot) \sim  \langle \cdot, \cdot\rangle_E,
\end{equation}
where $\langle \cdot, \cdot \rangle_E$ is the usual $L^2$-inner product. This discrete bilinear form is constructed using the $L^2$-projection operator $\Pi^0_{k,E}$, which ensure the mass matrix is computable using only the discrete system,  and is plays a crucial role in problems involving time-dependent equations.

Both discrete bilinear forms must satisfy \textit{k}-consistency and stability conditions. For all $E\in \mathcal{T}_h$, all $p \in \mathbb{P}_k(E)$ and for all $n \in V^E_{h,k}$, the \textit{k}-consistency guarantees that:
\begin{equation}
    \nonumber
    \begin{split}
        a^E_h(p,v) &= a^E(p,v), \\
        m^E_h(p,v) &= \langle p,v \rangle_E
    \end{split}
\end{equation}
The gradient approximation $a^E_h$ exactly reproduces the continuous bilinear form $a^E$ when one argument is as polynomial. It guarantees that the gradient-based energy projection operator $\Pi^\nabla_{k,E}$ acts consistently for polynomials of degree less or equal to $k$. Also, the mass approximation $m^E_h$ matches the exact $L^2$-inner product for polynomials, and guarantees that the $L^2$-projection operator preserves polynomial components. The consistency property ensures exactness on polynomials, and it is essential for the accuracy of the method. It ensures that the discrete formulation correctly approximates the continuous variational formulation at least for the polynomial space $\mathbb{P}_k(E)$.

The stability condition ensures that there exist four positive constants $\alpha_*$, $\alpha^*$, $\beta_*$, and $\beta^*$ that are mesh-independent, such that
\begin{equation}\label{eq:stability_condition}
    \begin{split}
        \alpha_* a^E(v,v) &\leq a^E_h(v,v) \leq  \alpha^* a^E(v,v), \\
        \beta_* \langle  v,v \rangle_E &\leq m^E_h(v,v) \leq \beta^* \langle v,v \rangle_E.
    \end{split}
\end{equation}
The discrete bilinear form $a^E_h$ must be bounded above and below by the continuous form $a_{E}$ for all $v \in V^E_{h,k}$ to ensure that the associated stiffness matrix is well-conditioned. Similarly, the mass bilinear form $m^E_h$ must be bounded by the $L^2$-inner product to ensure that the associated mass matrix is also well-conditioned and provides consistent time discretization. Stability is crucial for the well-posedness of the discrete  form, and it ensures that the discrete problem inherits the coercivity and continuity of the continuous problem.

The global forms are given by:
\begin{equation}\label{eq:global_energy_bilinear_form}
    a_h(\cdot, \cdot) = \sum \limits_{E\in\mathcal{T}_h}a^E_h (\cdot, \cdot)
\end{equation}\
and
\begin{equation}\label{eq:global_mass_bilinear_form}
    m_h (\cdot, \cdot) = \sum \limits_{E\in \mathcal{T}_h} m^E_h (\cdot, \cdot) = \langle \cdot, \cdot \rangle_h.
\end{equation}

The next lemma refers to the stability condition presented in (\ref{eq:stability_condition}) in terms of norm equivalence.
\begin{lemma}
    Consider the bilinear form in (\ref{eq:global_energy_bilinear_form}) and (\ref{eq:global_mass_bilinear_form}). There exist mesh independent constants $\alpha_*, \alpha^*, \beta_*, \beta^* > 0$ such that:
    \begin{equation}\label{eq:aux_ineq_energy_stab}
        \alpha_* \| \nabla v_h \|_{L^2(\Omega)}^2 \leq a_h (v_h, v_h) \leq \alpha^* \| \nabla v_h \|^2_{L^2(\Omega)}
    \end{equation}
    and
    \begin{equation}\label{eq:aux_ineq_mass_stab}
        \beta_* \| v_h \|_{L^2(\Omega)}^2 \leq \langle v_h, v_h \rangle_h \leq \beta^* \| v_h \|_{L^2{\Omega}}^2
    \end{equation}
    for all $v_h \in V_h$.
\end{lemma}

\begin{proof}
    Consider the decomposition for $v_h \in \localvespace$:
    \begin{equation}
    \nonumber
        v_h = \scalenergyproj v_h + w_h, \; \text{with} \; w_h = (I - \scalenergyproj)v_h \in \ker (\scalenergyproj).
    \end{equation}
    Considering the e local bilinear form in (\ref{eq:continuous_bilinear_stiff}), by orthogonality, it holds that:
    \begin{equation}
    \nonumber
        a^E(\scalenergyproj v_h, w_h) = 0.
    \end{equation}
    So, it is possible to write:
    \begin{equation}
    \nonumber
       \begin{split}
            \| \nabla v_h \|^2_{L^2(\Omega)} = a^E (\scalenergyproj v_h + w_h, \scalenergyproj v_h + w_h) = \\
            a^E(\scalenergyproj v_h, \scalenergyproj v_h) + 2 a^E(\scalenergyproj v_h, w_h) + a^E(w_h, w_h) = \\
            a^E(\scalenergyproj v_h, \scalenergyproj v_h) + a^E(w_h, w_h).
       \end{split}
    \end{equation}
    Thus,
    \begin{equation}
    \nonumber
        \| \nabla v_h \|^2_{L^2(\Omega)} = \| \nabla \scalenergyproj v_h \|^2_{L^2(\Omega)} + \| w_h \|^2_{L^2(\Omega)}.
    \end{equation}
    It holds true  that (this result is explored further in eq. (\ref{eq:stability_term_condition})):
    \begin{equation}
    \nonumber
        a^E_h(v_h, v_h) \geq \| \nabla \scalenergyproj v_h \|^2_{\lebesguespace{E}} + C_0 \| \nabla w_h \|^2_{\lebesguespace{E}} \geq \min (1, C_0) \| \nabla v_h \|^2_{\lebesguespace{E}},
    \end{equation}
    and
    \begin{equation}
    \nonumber
        a^E_h(v_h, v_h) \leq \| \nabla \scalenergyproj v_h \|^2_{\lebesguespace{E}} + C_1 \| \nabla w_h \|^2_{\lebesguespace{E}} \leq \max (1, C_1) \| \nabla v_h \|^2_{\lebesguespace{E}}.
    \end{equation}
    Summing over $E$ and using mesh regularity yields the global bounds with:
    \begin{equation}
    \nonumber
        \alpha_* = \min (1,C_0), \; \alpha^* = \max (1, C_1),
    \end{equation}
    resulting in (\ref{eq:aux_ineq_energy_stab}).

    Similar arguments and bounds hold for the mass form leading to (\ref{eq:aux_ineq_mass_stab}).
\end{proof}

\subsection{Stabilization term}
\label{sec:stabilization_term}

The projection $\Pi^\nabla_{k,E}$ is defined to preserve gradients against polynomials (see the \textit{Gradient orthogonality} and \textit{Zero-mean condition} in section \ref{sec:energy_projection}). It is a linear operator from a space from a space of dimension $2N_{dof}^E$ to one of dimension $2N_k$. By the Rank-nullity Theorem (see Appendix \ref{ap:results_maths}):
\begin{equation}
    \nonumber
    \dim ([V_{h,k}^E]^2)= 2N_{dof}^E - 2 N_k.
\end{equation}
The dimension grow with the number of vertices/edges once more vertices/edges mean more degrees of freedom enlarging the virtual element space without expanding the polynomial space, and with order $k$ once higher $k$ adds more edge and interior moments degrees of freedom. As can be seen the kernel functions $\mathbf{v}_h\in \ker (\Pi^\nabla_{k,E})$ are non-polynomial modes in the virtual element space, they satisfy all VEM constraints but have gradients orthogonal to polynomial gradients. In other words, they represent the extra flexibility VEM gains on polygons, but they are invisible to the projection. 

A naive choice would be:
\begin{equation}
    \nonumber
    a^E_h(u_h, v_h) = a^E (\Pi^\nabla_{k,E}u_h, \Pi^\nabla_{k,E}v_h).
\end{equation}
In this way, the discrete bilinear form would satisfy $k$-consistency (once it is exact on polynomials) but fails stability, specifically what refers to coercivity. For the discrete problem to be well-posed (unique solution via Lax-Milgram Theorem), the global bilinear form 
\begin{equation}
    \nonumber
    a_h(\cdot, \cdot ) = \sum \limits_{E \in \mathcal{T}_h}a^E_h(\cdot, \cdot),
\end{equation}
must be coercive, i.e., there exists $\alpha_* > 0$ that is independent of the mesh such that:
\begin{equation}
    \nonumber
    a_h(v_h, v_h) \geq \alpha_* \| \nabla v_h \|_{L^2(\Omega)}^2,
\end{equation}
for all $v_h \in V_h$. Locally, this requires that:
\begin{equation}
    \nonumber
    \alpha_* a^E(v_h, v_h) \leq a^E_h(v_h, v_h).
\end{equation}

The consistency term alone violates coercivity. Take any nonzero $v_h \in \ker (\Pi^\nabla_{k,E})$ (it exists if $N_{dof}^E > N_k$). Then, 
\begin{equation}\label{eq:aux_zero_implication}
    \Pi^\nabla_{k,E} v_h = 0 \Rightarrow a^E(\Pi^h_{k,E} v_h, \Pi^h_{k,E} v_h) = a^E(0, 0) = 0.
\end{equation}
The bilinear operator in (\ref{eq:bilinear_form})
\begin{equation}
    \nonumber
    a^E(v, v) = \int \limits_E |\nabla v |^2 dE
\end{equation}
is the $H^1$-seminorm. It is strictly positive for any nonzero $v$ that is not in the natural kernel of the operator (e.g., constants for the scalar Laplacian). Also, the projection operator is exact on polynomials of degree less or equal to $k$. Hence, any constant cannot lie in $\ker (\Pi^\nabla_{k,E})$ unless it is zero, because $\Pi^\nabla_{k,E}$ would return it unchanged. Therefore, if $v_h \in \ker (\Pi^\nabla_{k,E})$ and $v_h \neq 0$ it is not a constant. In this way, it holds that:
\begin{equation}\label{eq:aux_bilinear_estimation}
    a^E(v_h, v_h) = \int \limits_E |\nabla v_h |^2 dE = \| \nabla v_h \|^2_{L^2(\Omega)} > 0
\end{equation}
By (\ref{eq:aux_zero_implication}) and (\ref{eq:aux_bilinear_estimation}):
\begin{equation}
    \nonumber
    a^E_h(v_h, v_h) = 0 < \alpha_* a^E(v_h, v_h),
\end{equation}
for any $\alpha_* > 0$. Then, coercivity is violated because the discrete form doesn't account the energy in kernel modes.

Another argument regarding the naive choice of the bilinear term refers to the rank deficiency of the stiffness matrix. Represent $a^E$ as a matrix, and let $\{ \psi _i \}^{N^E_{\text{dof}}}_{i=1}$ be the canonical basis of $V^E_{h,k}$. The local stiffness matrix $\mathbf{K}^E$ has entries $K^E_{ij} = a^E_h(\psi _i, \psi _j)$. Using only consistency, the stiffness matrix entries are:
\begin{equation}
    \nonumber
    K^E_{ij} = a^E(\Pi^\nabla_{k,E} \psi_i, \Pi^\nabla_{k,E}  \psi_j).
\end{equation}\
The projection $\Pi^\nabla_{k,E}\psi _i$ exists in an $N_k$--dimensional space, so the matrix $\mathbf{K}^E$ (with size $N_{dof}^E \times N_{dof}^E$) has rank at most $N_k < N_{dof}^E$. It is rank deficient (null space dimension at least $N_{dof}^E - N_k$), leading to a singular system when assembled globally.

To address those issues, a symmetric positive semi-definite bilinear form $S^E: V^E_{h,k} \times V^E_{h,k} \longrightarrow \mathbb{R}$ is introduced such that:
\begin{equation}\label{eq:discrete_energy_form}
    a^E_h (u_h, v_h) = a^E (\Pi^\nabla_{k,E} u_h, \Pi^\nabla_{k,E} v_h) + S^E\left( (I-\Pi^\nabla_{k,E}) u_h, (I-\Pi^\nabla_{k,E}) v_h \right),
\end{equation}
where $I$ is the identity operator. It is worth mentioning that the stability term also satisfies the stability condition presented in (\ref{eq:stability_condition}):
\begin{equation}\label{eq:stability_term_condition}
    C_0 a^E(w, w) \leq S^E(w,w) \leq C_1 a^E(w,w), \; \forall w \in \ker(\scalenergyproj),
\end{equation}
with $C_0, C_1 >0$ mesh independent constants. All those arguments are also valid for the mass term. Although, as it is presented in Section \ref{sec:lumped_mass}, the stability term in the lumped mass formulation tends to disappear. 

In terms of matrix construction, the Gram matrix associated to the stiffness and mass matrices are, respectively:
\begin{equation}\label{eq:stiffness_gram}
    G^{E,\mathbf{K}}_{ij} = \int \limits_E \nabla \psi_i \nabla \psi_j dE,
\end{equation}
and
\begin{equation}\label{eq:mass_gram}
    G^{E,\mathbf{M}}_{ij} = \int \limits_E \psi_i \psi_j dE. 
\end{equation}
Thus, the stiffness and mass matrices are, respectively:
\begin{equation}
    \nonumber
    K^E_{ij} = G^{E,\mathbf{K}}_{ij} + |E| \gamma^{E,\mathbf{K}}  \langle (I-\Pi^\nabla_{k,E})\psi_i, (I-\Pi^\nabla_{k,E})\psi_j \rangle_E,
\end{equation}
and
\begin{equation}
    \nonumber
    M^E_{ij} = G^{E,\mathbf{M}}_{ij} + |E| \gamma^{E,\mathbf{M}}  \langle (I-\Pi^\nabla_{k,E})\psi_i, (I-\Pi^\nabla_{k,E})\psi_j \rangle_E,
\end{equation}
where $\gamma^{E,\mathbf{K}}, \gamma^{E,\mathbf{M}}>0$.

\section{Lumped mass formulation}
\label{sec:lumped_mass}

The semi-discrete VEM scheme approximates the weak formulation in (\ref{eq:weak_formulation}) as:
\begin{equation}\label{eq:discrete_form}
    m_h \left( \frac{\partial u_h}{\partial t}, v_h\right) + a_h(u_h, v_h) = \langle f_h, v_h \rangle
\end{equation}
for all $v_h \in V_{h,k}$, with almost every $t \in (0,T)$, with $u_{h,0} = u_h(0)$, where
\begin{equation}
    \nonumber
    f_h = \Pi^0_k f    
\end{equation}
is the $L^2$--projection source. In matrix form, this can be written as:
\begin{equation}\label{eq:matrix_form}
    \mathbf{M}_h \, \dot{\mathbf{u}}_h + \mathbf{K}_h \, \mathbf{u}_h = \mathbf{f}_h,
\end{equation}
where $\mathbf{M}_h$ and $\mathbf{K}_h$ are, respectively, the global mass and stiffness matrices. For explicit time-stepping, one needs to invert $\mathbf{M}_h$ repeatedly, which is costly if $\mathbf{M}_h$ is full. The mass lumping strategy replaces $\mathbf{M}_h$ with a diagonal $\hat{\mathbf{M}}_h$, enabling:
\begin{equation}
    \nonumber
    \mathbf{u}^{n+1}_h = \mathbf{u}^n_h - \Delta t \, (\hat{\mathbf{M}}_h)^{-1} \, \mathbf{K}_h \, \mathbf{u}_h^n + \Delta t \, (\hat{\mathbf{M}}_h)^{-1} \, \mathbf{f}_h^n.
\end{equation}

\subsection{The lumped mass matrix}
\label{sec:lumped_mass_matrix}

At local level, row-sum lumping computes a diagonal approximation of $\mathbf{M}^E$:
\begin{equation}
    \nonumber
    \hat{M}^E_{ii} \sum \limits^{N^E_{\text{dof}}}_{j=1} M^E_{ij}, \; \hat{M}^E_{ij} = 0, \; \forall i\neq j. 
\end{equation}
Let $\mathbf{s}^E \in \mathbb{R}^{N_k}$ be the local row-sum vector:
\begin{equation}
    \nonumber
    s^E_i = \sum \limits^{N^E_{\text{dof}}}_{j=1} M^E_{ij} = m^E_h \left( \sum \limits^{N^E_{\text{dof}}}_{j=1} \psi_j, \psi_i \right).
\end{equation}

By the partition of unity, the local basis satisfies
\begin{equation}
    \nonumber
    \sum \limits^E_{j=1}\psi_i = 1
\end{equation}
on $E$. Thus,
\begin{equation}\label{eq:local_row_sum}
    s^E_i = \int \limits_E \psi_i dE.
\end{equation}
So, the lumped mass matrix is given by:
\begin{equation}
    \nonumber
    \hat{\mathbf{M}}^E = \text{diag} \left( d_1, ..., d_{\numdofs} \right),
\end{equation}
where
\begin{equation}
    \nonumber
    d_i = \int \limits_E \psi_i dE.
\end{equation}

Note that in (\ref{eq:local_row_sum}) the stabilization term disappears. Because $\Pi^{0}_{k,E}$ is exact on $\mathbb{P}_k(E)$, it holds that:
\begin{equation}
    \nonumber
    (I - \Pi^{0}_{k,E})1 = 1 - \Pi^{0}_{k,E}1 = 0.
\end{equation}
The operator $(I-\Pi^{0}_{k,E})$ is the projection error operator, orthogonal to $\mathbb{P}_k(E)$ in the $L^2$--inner product. Since constants are in $\mathbb{P}_k(E)$, the error on constants is zero, ensuring stabilization (designed for non-polynomial modes) does not affect row-sums. This preserves mass conservation and positivity, as row-sums reduce to pure integrals without artificial scaling. So, it is possible to conclude that the stabilization term acts on the kernel of $\Pi^{0}_{k,E}$.  

\begin{rem}
    It is known that $\mathbb{P}_k(E)$ is the complement of $\ker (\Pi^{0}_{k,E})$. Recall that for any projector $\mathcal{P}$ on a vector space $V$ one has the canonical direct sum $V = \ker \mathcal{P} \oplus Im \, \mathcal{P}$, where $Im \, \mathcal{P}$ is the image of $\mathcal{P}$. For any $v_h \in V^E_{h,k}$, it is possible to write:
    \begin{equation}
    \nonumber
        v_h = \underbrace{(v_h - \Pi^{0}_{k,E} v_h)}_{\in \ker (\Pi^{0}_{k,E})} \oplus \underbrace{\Pi^{0}_{k,E} v_h}_{\in Im (\Pi^{0}_{k,E})}. 
    \end{equation}
    Also, $\ker (\Pi^{0}_{k,E}) \cap Im(\Pi^{0}_{k,E}) = \{ 0 \}$. Hence, the sum is direct.

    Considering the constraint (\ref{eq:enhancement_condition}), $\Pi^{0}_{k,E}v_h$ is the unique minimizer in $\mathbb{P}_k(E)$ of the $L^2$--distance to $v_h$. By definition $\scalmassproj v_h \in \P{E}$ for every $v_h \in \localvespace$, hence $Im (\scalmassproj)$. Now, taking $p \in \P{E}$, and plugging $v_h = p$ in (\ref{eq:enhancement_condition}):
    \begin{equation}
    \nonumber
        \int \limits_E (p - \scalmassproj p)\,q \, dE = 0, \: \forall q \in \P{E}.
    \end{equation}
    But $p - \scalmassproj p \in \P{E}$. The only element orthogonal to all elements of $\P{E}$ is $0$. Thus, 
    \begin{equation}
    \nonumber
        p - \scalmassproj p  = 0 \Rightarrow p = \scalmassproj p, \: \forall p \in \P{E}. 
    \end{equation}
    Therefore, every $p \in \P{E}$ lies in the image of $\scalmassproj$, and $Im (\scalmassproj) = \P{E}$. 
\end{rem}

In equation (\ref{eq:local_row_sum}), it seems like it is missing $\scalmassproj \psi_i$ in the integrand since the consistency term has projection on both arguments. However, this is not an omission but a simplification arising from the orthogonality of the projection error to constants which preserves the integral mean. Expanding the consistency term for $v_h = 1$:
\begin{equation}
    \nonumber
    \int \limits_E \scalmassproj 1 \scalmassproj \psi_i dE = \int \limits_E  \scalmassproj \psi_i dE. 
\end{equation}
By the projection definition, 
\begin{equation}
    \nonumber
    \int \limits_E (\scalmassproj \psi_i - \psi_i)q dE = 0, \; \forall q \in \P{E}.
\end{equation}
Considering $q = 1$, then:
\begin{equation}
    \nonumber
    \int \limits_E \scalmassproj \psi_i - \psi_i dE = 0 \Rightarrow \integral{E} \scalmassproj \psi_i dE = \integral{E} \psi_i dE.
\end{equation}

The next lemma shows that, as for the dense case, the lumped mass matrix continues to be symmetric positive definite.
\begin{lemma}[Positive definite]\label{lemma:spd}
    The lumped mass matrix $\M$ as defined so far is positive definite, i.e., $d_i > 0$ for all free degrees of freedom, with $i=1,...,N_{\text{dof}}$, and then
    \begin{equation}
        \nonumber
        \mathbf{v}^T\M \mathbf{v} > \mathbf{0}, \; \forall \mathbf{v} \neq \mathbf{0}
    \end{equation}
    on the free degrees of freedom.
\end{lemma}

\begin{proof}
    The strategy is to prove that $d_i > 0$ by showing that each local row-sum $s^E_i$, and the global sum is positive. The argument uses the partition of unity and the orthogonality of the projection error to constants.

    From the partition of unity on each $E$:
    \begin{equation}
        \nonumber
        \sum \limits^{\numdofs}_{i=1} \psi_i = 1.
    \end{equation}
    Integrating over $E$:
    \begin{equation}
        \nonumber
        \sum \limits^{\numdofs}_{i=1} \integral{E}\psi_i dE = \integral{E} 1 dE = |E| > 0 \Rightarrow \sum \limits^{\numdofs}_{i=1} s^E_i = |E|.
    \end{equation}
    Summing over all $E$:
    \begin{equation}
        \nonumber
        \sum \limits^{\numdofs}_{i=1} d_i = \sum \limits_{E \in \mathcal{T}_h} |E| = |\Omega| > 0.
    \end{equation}
    This shows that the total mass is positive.

    For $k = 1$, the local VEM basis $\psi_i$ on an element $E$ is harmonic, i.e., $\Delta \psi_i = 0$ in $\partial E$ and $\psi = 1$ at vertex $\mathcal{V}_i$, $\psi_i = 0$ at other vertices, and along two incident edges it varies linearly $1$ down to $0$. Hence,
    \begin{equation}
        \nonumber
        0 \leq \psi_i \leq 1 \; \text{on} \;  \partial E,
    \end{equation}
    and the boundary data is not identically constant. By the Maximum Principle (see Theorem \ref{theo:maximum_principle}), the harmonic extension satisfies:
    \begin{equation}
        \nonumber
        0 \leq \psi_i (x) \leq 1, \; \forall x \in E.
    \end{equation}
    Because the boundary data is not constant, $\psi_i$ is not constant. The Strong Maximum Principle (see Theorem \ref{theo:strong_maximum_principle}), then gives the strict interior bounds:
    \begin{equation}
        \nonumber
        0 < \psi_i (x) < 1, \; \forall x \in E,
    \end{equation}
    i.e., $\psi_i$ is strictly positive throughout the interior. Consequently,
    \begin{equation}
        \nonumber
        \int \limits_E \psi_i dE > 0.
    \end{equation}

    For $k \geq 2$, the canonical basis functions may change sign. The constant interior moment is included among the degrees of freedom:
    \begin{equation}
        \nonumber
        \chi^E_{\text{const}}(v) = \frac{1}{|E|} \int \limits_E v \, dE,
    \end{equation}
    and, by unisolvence of the dual basis,
    \begin{equation}
        \nonumber
        \chi^E_{\text{const}}(\psi_i) = \delta_{i,\text{const}}
        \;\Rightarrow\;
        \integral{E} \psi_i \, dE =
        \begin{cases}
            |E|, & i = \text{const interior moment}, \\
            0,   & \text{otherwise}.
        \end{cases}
    \end{equation}
    Therefore, with the row-sum weights $d_i$, the element-wise diagonal has exactly one positive entry and the others are zero when $k \geq 2$. It is positive semi-definite but not uniformly positive, so the lumped mass form cannot satisfy the stability condition given in Lemma \ref{lemma:lumped_mass_stability}. To enforce a uniformly positive diagonal (independent of the number of edges/vertices), the floored weight is adopted as in (\ref{eq:floored_weights}), which guarantees
    \begin{equation}
        \nonumber
        \tilde{d}_i > \delta \frac{|E|}{\numdofs} > 0
    \end{equation}
    for every local degree of freedom. Combined with Theorem \ref{theo:l2_equivalence}, this yields a mesh-independent $L^2$--equivalence for the lumped mass bilinear form. 

    For each global degree of freedom index $i$, there is at least one $i$ for which
    \begin{equation}
        \nonumber
        \integral{E} \psi_i dE > 0. 
    \end{equation}
    Thus, 
    \begin{equation}
        \nonumber
        d_i = \sum \limits_i \integral{E} \psi_i dE > 0.
    \end{equation}

    For $\mathbf{v} \neq \mathbf{0}$, let $v_j \neq 0$. Then, 
    \begin{equation}
        \nonumber
        \mathbf{v}^T \M \mathbf{v} = \sum \limits_i d_i v_i^2 \geq d_j d_j > 0.
    \end{equation}
\end{proof}

The lumped mass bilinear form also satisfies the stability condition.
\begin{lemma}\label{lemma:lumped_mass_stability}
    Let $\M$ be the lumped mass matrix assembled by the local row-sum lumping with floored weights
    \begin{equation}\label{eq:floored_weights}
        \tilde{d}_i = \max \left\{ d_i, \delta \frac{|E|}{\numdofs} \right\},
    \end{equation}
    where $\delta \in (0,1)$ is fixed. Assume the standard VEM mesh regularity as presented in Section \ref{sec:introduction} (star-shaped with chunkiness $\gamma$), and the number of edges per element is uniformly bounded by the maximum number of vertices in an element of the mesh $N_{\max}$. Then, there exist mesh-independent constants $\hat{\beta}_*, \hat{\beta}^*>0$ depending only on $d$, $k$, $\gamma$, $\delta$, and $N_{\max}$, such that for all $v_h \in V_h$:
    \begin{equation}
    \nonumber
        \hat{\beta_*} \| v_h \|^2_{\lebesguespace{\Omega}} \leq \langle v_h, v_h \rangle^{\text{lump}}_h \leq \hat{\beta}^* \| v_h \|^2_{\lebesguespace{\Omega}},
    \end{equation} 
    where
    \begin{equation}
    \nonumber
        \langle v_h, v_h \rangle^{\text{lump}}_h = \sum \limits_{E \in \tau_h} \sum \limits^{\numdofs}_{i=1}\tilde{d}_i \chi_i(v_h)^2.
    \end{equation}
\end{lemma}
\begin{proof}
    The proof is carried out locally on each element $E \in \mathcal{T}_h$, and the global bound follow by the summation over all elements, leveraging the mesh regularity to ensure the constants are uniform.

    The lumped form is positive due to $\tilde{d}_i > 0$, and the proof relies on Theorem \ref{theo:l2_equivalence}:
    \begin{equation}
        \nonumber
        C_*(d,k,\gamma) h_E^{-d} \| v_h \|_{\lebesguespace{E}}^2 \leq \sum \limits^{\numdofs}_{i=1} \chi (v_h)^2 \leq C^*(d,k,\gamma)h_E^{-d}\| v_h \|^2_{\lebesguespace{E}},
    \end{equation}
    where the constants $C_*, C^*>0$ are independent of the number of faces.

    Since $\tilde{d}_i \leq \delta |E|/\numdofs$, and
    \begin{equation}\label{eq:aux_upper_bound}
        c(\gamma) h_E^d \leq |E| \leq C(\gamma) h_E^d
    \end{equation}
    where $c, C > 0$ depends only on $\gamma$ due to the star-shaped constraints, it holds that:
    \begin{equation}
        \nonumber
        \langle v_h, v_h\rangle^{\text{lump}}_{h,E} \geq \min \limits_i \tilde{d}_i \sum \limits_i \geq \delta \frac{|E|}{\numdofs} C_*(d,k,\gamma) h_E^{-d} \| v_h \|^2_{\lebesguespace{E}} \geq \frac{\delta c(\gamma) C_* (d,k,\gamma)}{N_{\max}} \| v_h \|^2_{\lebesguespace{E}},
    \end{equation}
    yielding the lower bound with 
    \begin{equation}
        \nonumber
        \hat{\beta}_* = \frac{\delta c(\gamma) C_*(d,k,\gamma)}{N_{\max}}.
    \end{equation}

    Let
    \begin{equation}
        \nonumber
        S = \{ i: \; d_i < a \},
    \end{equation}
    with $S^C$ its complement and
    \begin{equation}
        \nonumber
        a = \delta \frac{|E|}{\numdofs},
    \end{equation}
    and let 
    \begin{equation}
        \nonumber
        \tilde{d}_i = \max \{ d_i, a \} = \begin{cases}
            d_i, \; i\notin S \ \\
            a, \; i \in S
        \end{cases}.
    \end{equation}
    Then, splitting the sum over $S$ and $S^C$:
    \begin{equation}
        \nonumber
        \sum \limits_i \tilde{d}_i = \sum \limits_{i \in S^C} \tilde{d}_i + \sum \limits_{i \in S} \tilde{d}_i = \sum \limits_{i \in S^C}d_i + \sum \limits_{i \in S}a.
    \end{equation}
    It holds that:
    \begin{equation}
        \nonumber
        \sum \limits_i d_i = \sum \limits_i - \sum \limits_{i \in S}d_i,
    \end{equation}
    so,
    \begin{equation}
        \nonumber
        \sum \limits_i \tilde{d}_i = \left( \sum \limits_i d_i - \sum \limits_{i \in S} d_i \right) + \sum \limits_{i \in S} a = \sum \limits_i d_i + \sum \limits_{i \in S} (a-d_i).
    \end{equation}
    The partition of unity gives that:
    \begin{equation}
        \nonumber
        \sum \limits_i d_i = |E|.
    \end{equation}
    Thus,
    \begin{equation}
        \nonumber
        \sum \limits_i \tilde{d}_i = |E| + |S|a - \sum \limits_{i \in S} d_i \Rightarrow \sum \limits_i \tilde{d}_i \leq |E| + |S|a + \sum \limits_{i \in S} |d_i|.
    \end{equation}
    With $\numdofs \leq N_{\max}$, it holds that $|S| < N_{\max}$, and:
    \begin{equation}
        \nonumber
        \sum \limits_i \tilde{d}_i \leq |E| + N_{\max} a + N_{\max} \max \limits_i |d_i|.
    \end{equation}

    By the definition of $d_i$:
    \begin{equation}
        \nonumber
        \begin{split}
            d_i = \integral{E} \psi dE = \langle \psi_i, 1 \rangle_E \Rightarrow |d_i| = \left|\langle \psi_i, 1 \rangle_E\right| \leq \| \psi_i \|_{\lebesguespace{E}} \| 1 \|_{\lebesguespace{E}} = \\
            \| \psi_i\|_{\lebesguespace{E}} \left( \integral{E} 1^2 dE \right)^{1/2} = \| \psi_i \|_{\lebesguespace{E}} |E|^{1/2}.
        \end{split}
    \end{equation}
    Therefore,
    \begin{equation}
        \nonumber
        |d_i|  \leq \| \psi_i \|_{\lebesguespace{E}}|E|^{1/2}.
    \end{equation}
    By Theorem \ref{theo:l2_equivalence} applied to $v_h = \psi_i$:
    \begin{equation}
        \nonumber
        C_*(d,k,\gamma) h_E^{-d}\| \psi \|_{\lebesguespace{E}} \leq \sum \limits_j \chi_j(\psi_i)^2 = 1 \Rightarrow \| \psi_i \|_{\lebesguespace{E}} \leq C_*(d,k,\gamma)^{-1/2} h_E^{d}. 
    \end{equation}
    Since $|E| \leq C(\gamma)h_E^d$, it follows that:
    \begin{equation}
        \nonumber
        \| \psi_i \|_{\lebesguespace{E}} \leq C_*(d,k,\gamma)^{-1/2}(C(\gamma))|E|^{1/2}.
    \end{equation}
    Thus,
    \begin{equation}
        \nonumber
        \max \limits_i |d_i| \leq C_d(k,\gamma) |E|,
    \end{equation}
    with
    \begin{equation}
        \nonumber
        C_d(k, \gamma) = C_*(d,k,\gamma)^{-1/2} C(\gamma)^{1/2}.
    \end{equation}
    Hence, 
    \begin{equation}
        \nonumber
        \sum \limits_i \tilde{d}_i \leq |E| + N_{\max} \delta \frac{|E|}{\numdofs} \numdofs + N_{\max} C_d(k, \gamma) |E| \leq(1+ \delta + N_{\max}C_d).
    \end{equation}
    So, it is true that:
    \begin{equation}
        \nonumber
        \begin{split}
            \langle v_h, v_h \rangle^{\text{lump}}_{h,E} = \sum \limits_i \tilde{d}_i \chi_i (v_h)^2 &\leq \\
            \left( \sum \limits_i \tilde{d}_i \right) \left( \sum \limits_i \chi_i (v_h)^2 \right) &\leq (1+ \delta + N_{\max}C_d) C^*(d,k,\gamma) h_E^{-d} \|v_h\|_{\lebesguespace{E}}.
        \end{split}
    \end{equation}
    Using the upper bound of (\ref{eq:aux_upper_bound}):
    \begin{equation}
        \nonumber
        \langle v_h, v_h \rangle^{\text{lump}}_{h,E} \leq (1+ \delta + N_{\max}C_d) C(\gamma) C^*(d,k,\gamma) \| v_h \|^2_{\lebesguespace{E}},
    \end{equation}
    yielding the upper bound:
    \begin{equation}
        \nonumber
        \hat{\beta}^* = (1+ \delta + N_{\max}C_d) C(\gamma) C^*(d,k,\gamma).
    \end{equation}
    Note that the constants $C_*, C^*$ of Theorem \ref{theo:l2_equivalence} depends only on $d$, $k$ and $\gamma$, and the volume-diameter constants $c(\gamma), C(\gamma)$ ensure uniformity. 
\end{proof}

\subsection{Compute the lumped mass matrix}
\label{sec:compute_lumped_mass_matrix}
The computation of the lumped local mass matrix $\hat{\mathbf{M}}^E$ without the explicit knowledge of the basis functions leverages the degrees of freedom of the virtual element space, the $L^2$--projection properties, and a small linear system solvable per element. Let $\{ p_\alpha \}^{N_k}_{\alpha=1}$ be a basis for $\P{E}$. The chosen form for each $p_\alpha$ are the scaled monomials. For the two-dimensional case those monomials are given by:
\begin{equation}
    \nonumber
    p_\alpha (x,y) = \left( \frac{x - x_c}{h_E} \right)^{\alpha_1} \left( \frac{y - y_c}{h_E} \right)^{\alpha_2},
\end{equation}
with $\alpha_1, \alpha_2 \in \mathbb{N}$, $\alpha_1 + \alpha_2 \leq k$, and where $(x_c, y_c)$ is the centroid of $E$, and $h_E$ is the polygonal diameter of element $E$.

It is necessary to compute the DOF matrix $\mathbf{D}^E \in \mathbb{R}^{N_k \times \numdofs}$ such that
\begin{equation}
    \nonumber
    D^E_{\alpha i} = \chi^E_i(p_\alpha).
\end{equation}
Concretely, by DOF type:
\begin{enumerate}
    \item Vertex DOF: for each vertex $V_i$:
        \begin{equation}
    \nonumber
            D^E_{\alpha V_i} = \chi^E_{V_i}(p_\alpha) = p_\alpha (V_i), \nonumber    
        \end{equation}
    \item Edge moment DOF: when $k \geq 2$, for $(e_i, j)$, where $e_i$ regards to the $i$--th edge of $E$:
        \begin{equation}
    \nonumber
            D^E_{\alpha (e_i, j)} = \chi^E_{e_i, j}(p_\alpha) = \frac{1}{|e_i|} \int \limits_{e_i} p_\alpha \hat{L}_j^{(e_i)}ds = \frac{1}{2} \int \limits^{1}_{-1} p_\alpha (\gamma_{e_i}(t))\hat{L}_j(t)dt, \nonumber
        \end{equation}
    \item Interior moment DOF: for multi-index $\boldsymbol{\beta}$ with $\boldsymbol{\beta} \leq k - 2$ when $k \geq 2$:
        \begin{equation}
    \nonumber
            D^E_{\boldsymbol{\alpha}, \boldsymbol{\beta}} = \chi^E_{\boldsymbol{\beta}}(p_{\boldsymbol{\alpha}}) = \frac{1}{|E|} \int \limits_E p_{\boldsymbol{\alpha}} m_{\boldsymbol{\beta}} \, dE.
        \end{equation}
\end{enumerate}
All these entries are exactly computable for polynomials $p_\alpha$: vertex values are point evaluations, edge entries are one-dimensional integrals of polynomials (exact with a finite number of Gauss points, or via orthonormality of $\hat{L}_j$), and interior entries are polygonal moments (obtainable by Green's Theorem or exact polygonal quadrature). 

As in the standard formulation for the full mass matrix (see, for example, \cite{vacca2015parabolic}), it is necessary to build the polynomial part of the mass term as in (\ref{eq:mass_gram}):
\begin{equation}\label{eq:lumped_mass_gram}
    G^{E,\mathbf{M}}_{\alpha \beta} = \integral{E} p_\alpha p_\beta dE.
\end{equation}
The constant vector $\mathbf{c}^E \in \mathbb{R}^{N_k}$ is built as:
\begin{equation}\label{eq:lumped_constant}
    c^E_\alpha = \integral{E} p_\alpha dE.
\end{equation}
Both (\ref{eq:lumped_mass_gram}) and (\ref{eq:lumped_constant}) are polynomial integration and can be computed exactly. The computation cost to solve the system
\begin{equation}
    \nonumber
    \mathbf{G}^E \mathbf{w}^E = \mathbf{c}^E
\end{equation}
is $\mathcal{O}(N_k^3)$ and becomes cheap as $N_k << \numdofs$. Finally, the row-sum vector $\mathbf{s}^E$ is computed as:
\begin{equation}
    \nonumber
    \mathbf{s}^E = \left( \mathbf{D}^E \right)^T \mathbf{w}^E.
\end{equation}

\section{Time discretization}
\label{sec:estimates}

This section develops explicit time integration for the semi-discrete VEM system with lumped mass. The focus is on strong stability-preserving Runge-Kutta (SSP-RK) schemes that inherit the forward Euler nonexpansiveness in a convex functional, here specialized to the discrete energy norm induced by $\M$. The classical Shu-Osher representation is regarded as a theoretical device; the practical framework adopted follows the GLM/monotonicity characterization in \cite{izzo2022ssp}, which directly yields implementable Butcher coefficients and a stepsize restriction $\Delta t \le C_{\mathrm{SSP}}\Delta t_{\mathrm{FE}}$ once $\Delta t_{\mathrm{FE}}$ is known.

The presentation proceeds as follows. First, the SSP notion is recalled for the semi-discrete system $\dot{\u}=\mathbf{G}(\u)$ and the FE-based stepsize bound is linked to the spectral quantity $\lambda_{\max}\!\big((\M)^{-1}\K\big)$. Next, the specific integrators employed—third-order SSP-RK3 and fourth-order SSP-RK(5,4) from \cite{izzo2022ssp}—are indicated (coefficients and implementation details are available in that reference). Finally, energy stability estimates are derived for homogeneous and nonhomogeneous problems, using $L^2$-DOF equivalence and inverse inequalities together with a mesh-robust spectral upper bound for $(\hat{\mathbf{M}}_h)^{-1}\mathbf{K}_h$, which yields a CFL restriction of order $h^2$ with constants independent of the number of faces/edges.

\subsection{Strong stability preserving Runge-Kutta methods}
\label{sec:runge_kutta}

In this section, $\mathbf{G}$ denotes the semi-discrete right-hand side induced by the VEM discretization with lumped mass. For the general (nonhomogeneous) case:
\begin{equation}
    \nonumber
  \mathbf{G}(\u) \;=\; -\,\M^{-1}\K\,\u \;+\; \M^{-1}\,\mathbf{f}_h,
\end{equation}
where $\u$ is the vector of VEM coefficients, $\M$ is the (diagonal) lumped mass matrix, $\K$ is the global stiffness matrix, and $\mathbf{f}_h$ is the load vector (assembled using the $L^2$-projection).

In the homogeneous case ($\mathbf{f}_h=\mathbf{0}$):
\begin{equation}
    \nonumber
  \mathbf{G}(\u) \;=\; -\,\M^{-1}\K\,\u.
\end{equation}

Strong stability-preserving Runge-Kutta (SSP-RK) methods are high-order time integrators designed to retain “strong stability” properties (e.g., monotonicity, contractivity, boundedness in a chosen norm or convex functional) that forward Euler guarantees for semi-discrete systems
\begin{equation}
    \nonumber
    \dot{\u} \;=\; \mathbf{G}(\u),
\end{equation}
under an appropriate time-step restriction. Let $\Phi(\cdot)$ denote a convex functional (for instance, total variation, $\|\cdot\|_\infty$, or the discrete energy norm $\|\cdot\|_{\M}$). If forward Euler (FE)
\begin{equation}
    \nonumber
    \u^{n+1}_{\mathrm{FE}} \;=\; \u^{n} + \Delta t\,\mathbf{G}(\u^{n})
\end{equation}
is nonexpansive in $\Phi$, i.e.,
\begin{equation}
    \nonumber
    \Phi\!\big(\u^{n} + \Delta t\,\mathbf{G}(\u^{n})\big)\ \le\ \Phi(\u^{n})
    \quad \text{for all } 0<\Delta t\le \Delta t_{\mathrm{FE}},
\end{equation}
then an SSP-RK method, which admits a Shu-Osher (convex-combination-of-FE) representation introduced in \cite{shu1988eno}, inherits this nonexpansiveness by convexity. Concretely, there exists a method-dependent SSP coefficient $C_{\mathrm{SSP}}>0$ such that the SSP-RK update is nonexpansive for
\begin{equation}
    \nonumber
    0<\Delta t\ \le\ C_{\mathrm{SSP}}\,\Delta t_{\mathrm{FE}}.
\end{equation}
This provides a pathway to construct high-order time integrators that preserve the strong-stability features of FE while maintaining high accuracy in smooth regimes. In the next section, this framework is specialized to the VEM semi-discretization with lumped mass, using $\Phi(\u)=\|\u\|_{\M}$ and the underlying spectral bounds.

The classical Shu-Osher representation expresses SSP-RK schemes as convex combinations of forward Euler (FE) steps and explains nonexpansiveness in a convex functional $\Phi(\cdot)$ by convexity. While theoretically insightful, that representation is not implementation-ready in general: the convex coefficients are nonunique, higher-order or many-stage schemes can force negative or modified weights and even require down winding operators $\tilde{L}$, and the construction does not directly yield a unique Butcher tableau to implement. A modern alternative that is directly implementable is the GLM/monotonicity framework reviewed in \cite{izzo2022ssp}, which recasts Runge-Kutta methods as a special case of general linear methods and characterizes SSP via a matrix positivity condition. In that setting, with suitable $S$ and $T$ built from the Runge-Kutta coefficients, the SSP property holds if and only if there exists $\gamma>0$ such that
\begin{equation}
    \nonumber
    (\mathbf{I}+\gamma \mathbf{T})^{-1}\,[\,\mathbf{S} \;\mid\; \gamma \mathbf{T}\,]\ \ge\ 0 \; \text{(componentwise)},
\end{equation}
and the associated SSP coefficient is defined by
\begin{equation}
    \nonumber
    C_{\mathrm{SSP}}\;=\;\sup\big\{\gamma>0:\ (\mathbf{I}+\gamma \mathbf{T})^{-1}\,[\,\mathbf{S} \;\mid\; \gamma \mathbf{T}\,]\ \ge\ 0\big\},
\end{equation}
where $\mathbf{I}$ is the identity operator. This yields a clean constrained optimization (positivity plus Runge-Kutta order conditions) whose solution provides directly the Butcher coefficients $(\mathbf{A},\mathbf{b},\mathbf{c})$ with guaranteed SSP radius, avoiding ad-hoc convex decomposition.

In this work, two explicit SSP-RK schemes from \cite{izzo2022ssp} are adopted: the optimal third-order three-stage SSP-RK3 and a fourth-order five-stage SSP-RK(5,4). Their coefficients and detailed constructions are reported in \cite{izzo2022ssp}; implementation follows those tables. The associated SSP timestep restriction reads
\begin{equation}
    \nonumber
    0<\Delta t\ \le\ C_{\mathrm{SSP}}\,\Delta t_{\mathrm{FE}},
    \qquad
    \Delta t_{\mathrm{FE}} \;=\; \frac{2}{\lambda_{\max}\!\big((\M)^{-1}\K\big)},
\end{equation}
so, with $\Phi(\u)=\|\u\|_{\M}$ and the spectral bounds proved below, the resulting explicit integrators preserve the discrete strong-stability features inherited from FE for the VEM semi-discretization with lumped mass.

\subsection{Energy stability estimates for the explicit time-step schemes with VEM space-discretization with lumped mass matrix}
\label{sec:stability_estimates}

All results in this section are stated for the scalar problem. Since the diffusion operator and the VEM projectors act componentwise, the vector-valued case decouples into identical scalar problems (the global mass and stiffness matrices are block diagonal with repeated scalar blocks), and the same spectral bounds (with the same constants) hold componentwise. Let
\begin{equation}
    \nonumber
    \lambda_{\max} = \max_{v_h \in V_h \setminus \{0\}} \frac{a_h(v_h, v_h)}{\langle v_h, v_h \rangle_h}.
\end{equation}

The next theorem gives a verifiable $\Delta t$ condition that makes the explicit VEM scheme with lumped mass formulation energy stable. Also, it aligns the discrete behavior with the PDE diffusion physics, and yields standard $\Delta t \approx \mathcal{O}(h^2)$ rule used in practice.

\begin{theo}[Homogeneous case]\label{theo:homogeneous_case}
    Consider the forward Euler discretization of the homogeneous semi-discrete VEM system:
    \begin{equation}
    \nonumber
        \u^{n+1} = \u^n - \Delta t \M^{-1}\K \u^n.
    \end{equation}
    Then, the discrete energy satisfies:
    \begin{equation}
    \nonumber
        \| \u^{n+1} \|_{\M} \leq \| \u^n \|_{\M}, \; \forall n,
    \end{equation}
    with the discrete norm given by:
    \begin{equation}
    \nonumber
        \| \mathbf{v} \|_{\M} = \mathbf{v}^T \M \mathbf{v}= \sum \limits^{N_{\text{dof}}}_{i=1} d_i v_i^2, \; \mathbf{v} \neq \mathbf{0},
    \end{equation}
    and
    \begin{equation}
    \nonumber
        \M = \text{diag} (d_1, ..., d_{N_{\text{dof}}}),
    \end{equation}
    provided
    \begin{equation}
    \nonumber
        \Delta t \leq \frac{2}{\lambda_{\max}\left(\M^{-1}\K\right)},
    \end{equation}
    where $N_{\text{dof}}$ is the total number of degrees of freedom.
\end{theo}

\begin{proof}
    See Appendix \ref{sec:proof_homogeneous}.
\end{proof}

The next theorem provides a discrete energy estimate for the nonhomogeneous semi-discrete VEM system advanced by forward Euler. Under the diffusion-type CFL condition $\Delta t \le 2/\lambda_{\max}\!\big(\M^{-1}\K\big)$, the $\M$-energy at step $n+1$ is bounded by $(1+\Delta t)$ times the previous energy plus load-dependent terms measured in the dual norm $\|\cdot\|_{\M^{-1}}$. The inequality shows that the dissipative action of $\K$ (controlled through the spectral bound) counteracts growth, while the forcing contributes linearly and quadratically in $\Delta t$. In the homogeneous limit ($\mathbf{f}_h^n=\mathbf{0}$), the estimate reduces to non-increase of the discrete energy, matching the unforced result.

\begin{theo}[Nonhomogeneous case]\label{theo:non_homogeneous_case}
    Consider the forward Euler discretization of the non-homogenoeus semi-discrete lumped VEM system:
    \begin{equation}
    \nonumber
        \u^{n+1}  = \u^n - \Delta t \M^{-1}\K \u^n + \Delta t \M^{-1} \mathbf{f}_h^n.
    \end{equation}
    Then, the discrete energy satisfies
    \begin{equation}
    \nonumber
        \| \u^{n+1} \|^2_{\M} \leq (1 + \Delta t) \| \u^n \|^2_{\M} + (\Delta t + \Delta t^2)\|\mathbf{f}_h^n\|_{\M^{-1}}
    \end{equation}
    provided
    \begin{equation}
    \nonumber
        \Delta t \leq \frac{2}{\lambda_{\max}\left(\M^{-1}\K\right)}.
    \end{equation}
\end{theo}

\begin{proof}
    See Appendix \ref{sec:proof_nonhomogeneous}.
\end{proof}

The corollary applies a discrete Grönwall argument to the recursion when the load is uniformly bounded in $\|\cdot\|_{\M^{-1}}$. Writing $t_n=n\,\Delta t$, the bound exhibits at most exponential growth in time with constants depending on the forcing magnitude and $\Delta t$, but not on mesh topology (thanks to the mesh-robust CFL). This delivers a global-in-time stability estimate that quantifies the effect of a bounded source on the discrete energy.

\begin{cor}\label{cor:gronwall_consequence}
    If $\| \mathbf{f}_h^n\|_{\M^{-1}} \leq C_F$ uniformly for all $n$ and $C_F > 0$, then for $t_n = \Delta t n$, then
    \begin{equation}
        \nonumber
        \| \u^n \|^2_{\M} \leq \exp (t_n) \| \u^0 \|_{\M} + (1 + \Delta t)(\exp(t_n)-1)C_F^2.
    \end{equation}
\end{cor}

\begin{proof}
    See Appendix \ref{sec:proof_gronwall}.
\end{proof}

To the best of our knowledge, the next result is the first proof, in the VEM setting for parabolic problems, of an explicit-time stability/CFL condition based on a lumped mass matrix, including a bound $\lambda_{\max}\!\le C_{\text{VEM}}/h^2$ under standard VEM mesh regularity and stabilization, with constants independent of the number of faces/edges. The analysis establishes an energy stability condition for forward Euler and, via standard SSP theory, for Shu--Osher/SSP Runge--Kutta methods. The central object is the generalized eigenvalue $\lambda_{\max}\big((\hat{\mathbf{M}}_h)^{-1}\mathbf{K}_h\big)$; the diffusion-type estimate above holds with $C_{\mathrm{VEM}}$ depending only on the polynomial degree and the mesh chunkiness parameter, and independent of the number of faces/edges per element. This yields a CFL restriction $\Delta t \le c\, h^2$, rendering mass lumping a provably stable and mesh-robust choice for explicit diffusion solvers on general polytopal meshes.

\begin{theo}[Spectral upper bound for the lumped mass operator]\label{theo:lumped_spectral_bound}
Let $V_{h,k}\subset H^1_0(\Omega)$ be the global conforming VEM space with basis $\{\psi_i\}_{i=1}^N$ and let $a_h(\cdot,\cdot)$ be the global stiffness form. Let $\hat{\mathbf{M}}_h$ be the global lumped mass matrix assembled from the local row-sum construction, and let $\hat m_h(\cdot,\cdot)$ denote the bilinear form induced by $\hat{\mathbf{M}}_h$, i.e., $\hat m_h(u_h,v_h)=\mathbf{u}^T\hat{\mathbf{M}}_h\mathbf{v}$ for coefficient vectors $\mathbf{u},\mathbf{v}$. Assume the mesh is shape regular follows the regularity conditions so that the following mesh-independent bounds hold:
\begin{enumerate}
    \item[$(i)$] ($L^2$-equivalence for the lumped mass) There exists $\hat{\beta}_*>0$ such that
    \begin{equation}
    \nonumber
        \hat m_h(v_h,v_h) \;\ge\; \hat{\beta}_*\, \| v_h \|^2_{L^2(\Omega)} \qquad \forall\, v_h\in V_{h,k}.
    \end{equation}
    \item[$(ii)$] (Inverse inequality) There exists $C_{\mathrm{inv}}>0$ such that
    \begin{equation}
    \nonumber
        | v_h |_{H^1(\Omega)} \;\le\; C_{\mathrm{inv}}\, h^{-1}\, \| v_h \|_{L^2(\Omega)} \qquad \forall\, v_h\in V_{h,k}.
    \end{equation}
\end{enumerate}
Then the largest generalized eigenvalue of the pair $(\mathbf{K}_h,\hat{\mathbf{M}}_h)$, equivalently
\begin{equation}
    \nonumber
    \lambda_{\max}\!\big( (\hat{\mathbf{M}}_h)^{-1}\mathbf{K}_h \big)
    \,=\, \max_{v_h\in V_{h,k}\setminus\{0\}} \frac{a_h(v_h,v_h)}{\hat m_h(v_h,v_h)}
\end{equation}
satisfies the diffusion-type bound
\begin{equation}
    \nonumber
    \lambda_{\max}\!\big( (\hat{\mathbf{M}}_h)^{-1}\mathbf{K}_h \big) \;\le\; \frac{C_{\mathrm{inv}}^2}{\hat{\beta}_*}\, \frac{1}{h^{2}}\,.
\end{equation}
\end{theo}

\begin{proof}
    The maximum eigenvalue is given by the Rayleigh quotient:
    \begin{equation} 
        \nonumber
        \lambda_{\max} \left( (\M)^{-1} \K \right) = \max \limits_{\mathbf{v} \neq \mathbf{0}} \frac{\mathbf{v}^T \K \mathbf{v}}{\mathbf{v}^T \M \mathbf{v}} = \max \limits_{v_h \in V_{h,k} \setminus \{0\}} \frac{a_h(v_h, v_h)}{\langle v_h, v_h\rangle^{\text{lump}}_h},
    \end{equation}
    where
    \begin{equation}
        \nonumber
        v_h = \sum \limits^{N_{\text{dof}}}_{i=1} v_i \psi_i,
    \end{equation}
    and
    \begin{equation}
        \nonumber
        a_h(v_h, v_h) = \| \nabla v_h \|^2_{\lebesguespace{\Omega}}.
    \end{equation}

    From Lemma \ref{lemma:lumped_mass_stability}, it is known that:
    \begin{equation}
        \nonumber
        \langle v_h, v_h \rangle^{\text{lump}}_h \geq \hat{\beta}_* \| v_h \|^2_{\lebesguespace{\Omega}}.
    \end{equation}
    Thus,
    \begin{equation}
        \nonumber
        \lambda_{\max} \leq \frac{1}{\hat{\beta}_*} \max \limits_{v_h \neq 0}\frac{\| \nabla v_h \|_{\lebesguespace{\Omega}}}{\| v_h \|_{\lebesguespace{\Omega}}}.
    \end{equation}
    The maximum quotient is bounded by the inverse inequality presented in \cite{beirao2013vem}:
    \begin{equation}
        \nonumber
        \| \nabla v_h \|^2_{\lebesguespace{\Omega}} \leq C_{\text{inv}}^2 h^{-2} \| v_h \|^2_{\lebesguespace{\Omega}}.
    \end{equation}
    Therefore,
    \begin{equation}
        \nonumber
        \lambda_{\max} \leq \frac{C^2_{\text{inv}}}{\hat{\beta}_*}h^{-2} = \frac{C}{h^2},
    \end{equation}
    with
    \begin{equation}
        \nonumber
        C = \frac{C^2_{\text{inv}}}{\hat{\beta}_*} > 0
    \end{equation}
    depending only on $d$, $k$ and $\gamma$.
\end{proof}

\section{Numerical experiments}
\label{sec:numerical}

The present section illustrates the performance of the proposed lumped virtual element discretization through two complementary numerical studies. The first one is a convergence test based on a manufactured smooth solution, designed to verify the theoretical spatial accuracy of the method and to compare the behavior of the explicit SSP--Runge--Kutta integrators introduced in Section~\ref{sec:runge_kutta}. The second one addresses a variable-coefficient diffusion problem with discontinuous and anisotropic tensors, with the aim of assessing the robustness of the lumped explicit formulation in a more demanding setting where no analytical reference solution is available. Taken together, these experiments provide both a verification of the theoretical estimates and an indication of the practical behavior of the method in heterogeneous diffusion regimes.

All experiments are carried out on the same three mesh families already considered throughout the paper, namely distorted quadrilateral meshes, serendipity-type quadrilateral meshes, and general Voronoi polygonal meshes. Representative examples of these discretizations are shown in Figures~\ref{fig:mesh-distorted}, \ref{fig:mesh-serendipity}, and \ref{fig:mesh-voronoi}. For each family, a sequence of successively refined meshes is employed, with characteristic mesh sizes $h=0.2500$, $0.1250$, and $0.0625$ (or the corresponding values obtained from the actual polygonal tessellations). This common mesh setting allows a direct comparison between the manufactured-solution study and the heterogeneous diffusion test, while also highlighting the behavior of the proposed method on both structured and genuinely polygonal grids.

The spatial discretization is the conforming scalar VEM of order $k=1$ combined with the lumped mass construction developed in the previous sections. Time integration is performed by explicit SSP--RK schemes, with the time step chosen according to a diffusion-type scaling $\Delta t=\theta h^2$, where $\theta$ is selected below the stability threshold predicted by the spectral estimates for $(\M)^{-1}\K$; cf.~Section~\ref{sec:estimates}. In the first study this choice ensures that temporal errors remain negligible with respect to the spatial discretization error, whereas in the second study it provides the appropriate framework for investigating the CFL restriction in the presence of heterogeneous and anisotropic diffusion coefficients.

\begin{figure}[H]
    \centering
    \includegraphics[width=0.85\textwidth]{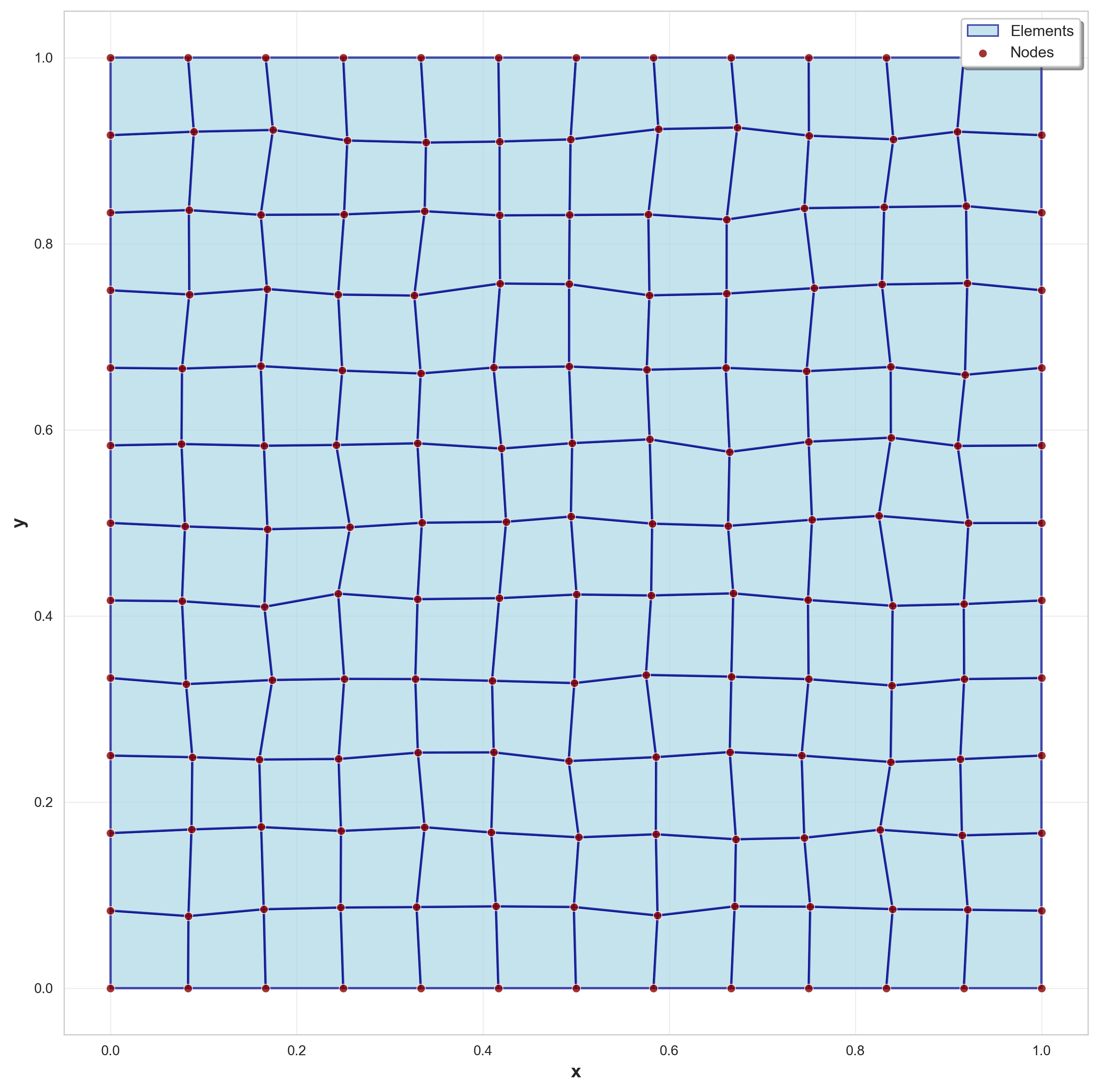}
    \caption{Distorted Q4 mesh (12$\times$12), $h_{\max}=0.125$.}
    \label{fig:mesh-distorted}
\end{figure}

\begin{figure}[H]
    \centering
    \includegraphics[width=0.85\textwidth]{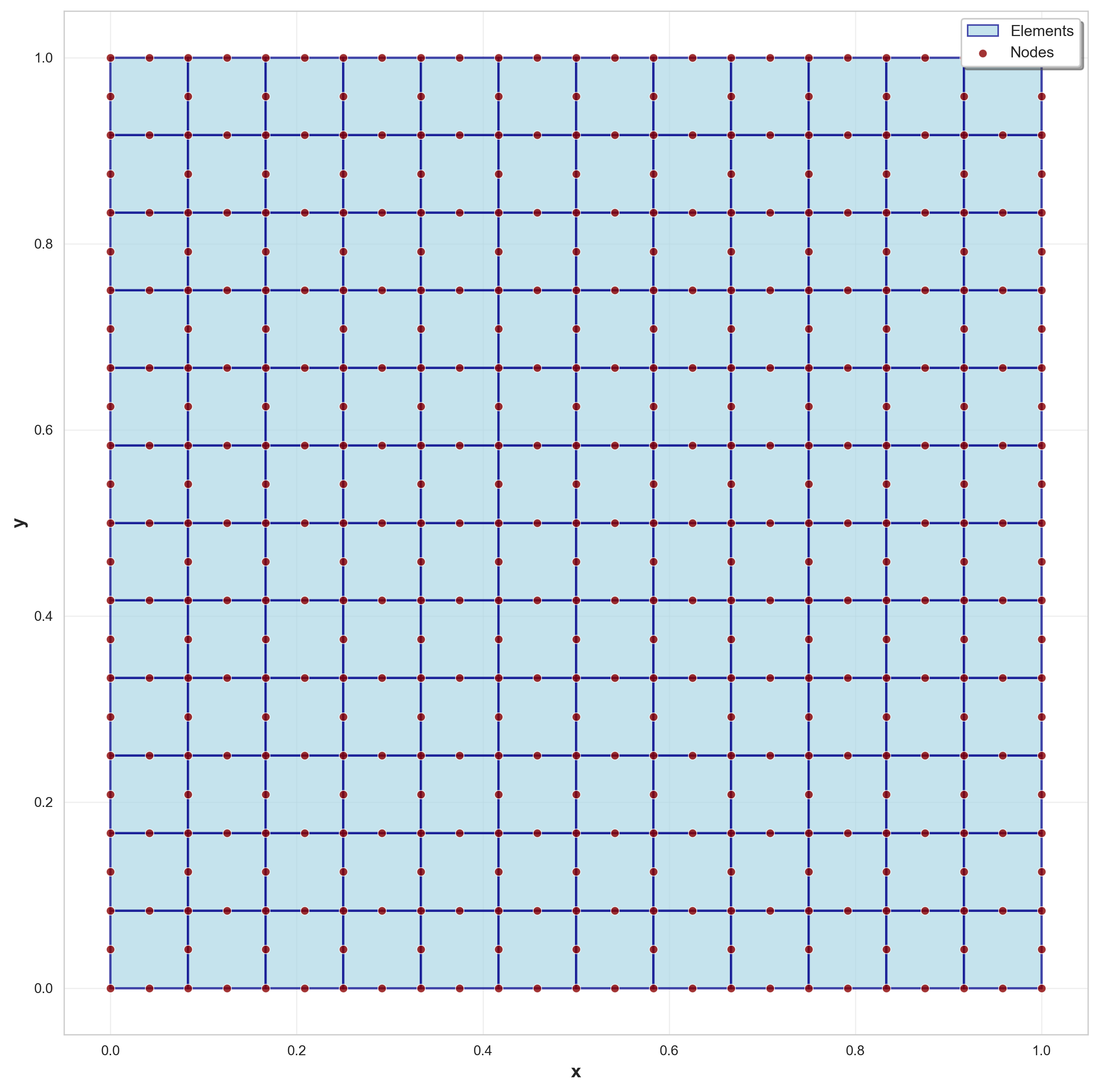}
    \caption{Serendipity Q8 mesh (12$\times$12), $h_{\max}=0.125$.}
    \label{fig:mesh-serendipity}
\end{figure}

\begin{figure}[H]
    \centering
    \includegraphics[width=0.85\textwidth]{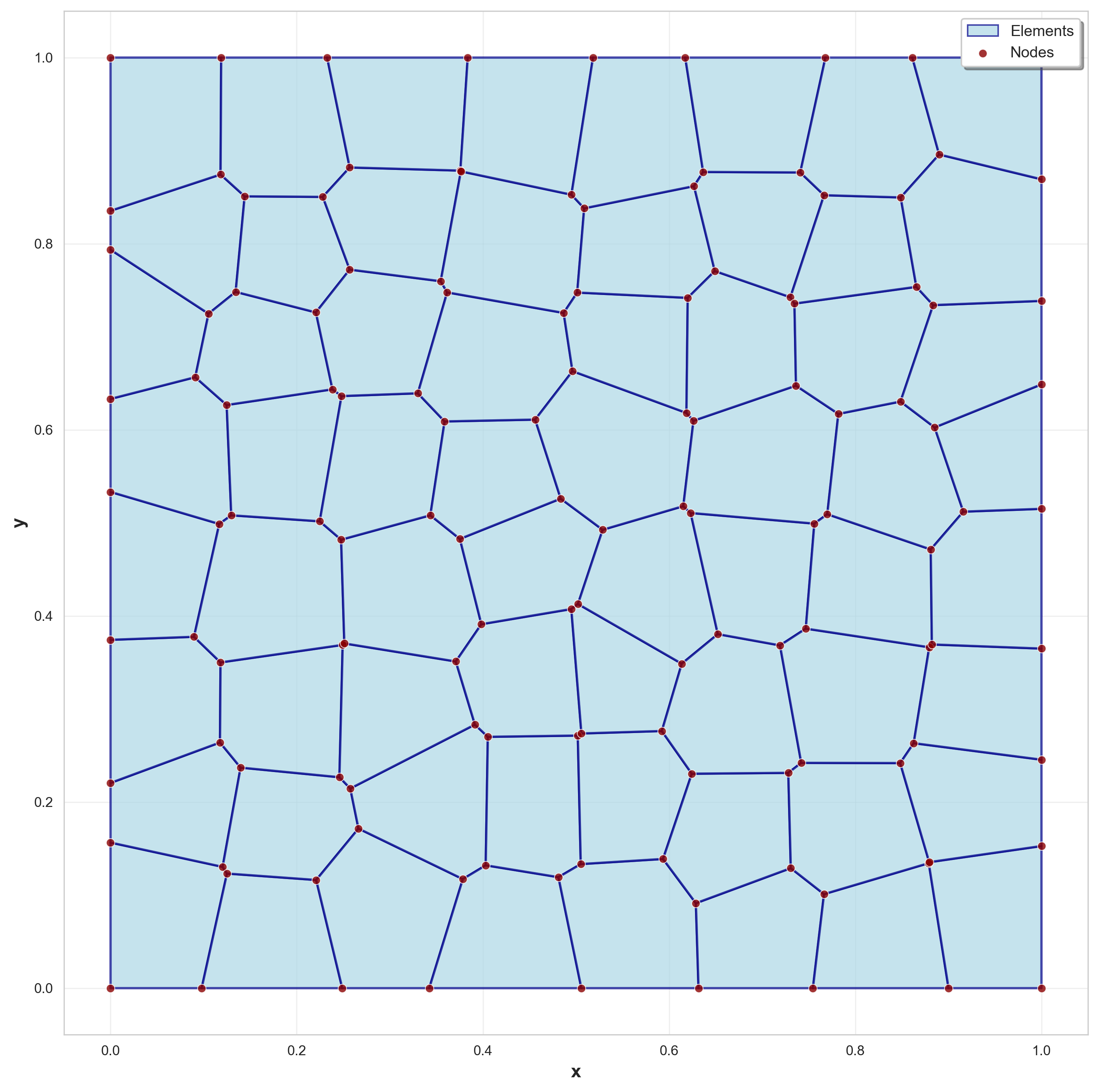}
    \caption{Voronoi polygonal mesh (avg. 5.5 vertices/element).}
    \label{fig:mesh-voronoi}
\end{figure}

\subsection{Convergence study for a manufactured parabolic solution}
\label{sec:manufactured}

The numerical experiments consider the parabolic problem $\partial_t u - \Delta u = f$ on the unit square $\Omega=(0,1)\times(0,1)$ for $t\in(0,T]$ with $T=1$, subject to homogeneous Dirichlet boundary conditions $u=0$ on $\partial\Omega$. The full model reads
\begin{equation}
    \nonumber
    \left\{
    \begin{aligned}
        \frac{\partial u}{\partial t} - \Delta u &= f(t,x,y) && \text{in } \Omega \times (0,T],\\
        u(t,x,y) &= 0 && \text{on } \partial\Omega \times (0,T],\\
        u(0,x,y) &= u_0(x,y) && \text{in } \Omega.
    \end{aligned}
    \right.
\end{equation}

An analytical solution is prescribed as
\begin{equation}
    \nonumber
    u(t,x,y) = e^{t}\,\sin(\pi x)\,\sin(\pi y),
\end{equation}
which implies the initial condition
\begin{equation}
    \nonumber
    u_0(x,y) = \sin(\pi x)\,\sin(\pi y),
\end{equation}
and automatically satisfies the boundary conditions since $\sin(\pi x)\sin(\pi y)=0$ on $\partial\Omega$.

The right-hand side follows from substitution in the PDE. Using
\begin{equation}
    \nonumber
    \partial_t u = e^{t}\sin(\pi x)\sin(\pi y)
\end{equation}
and
\begin{equation}
    \nonumber
    \Delta u = -2\pi^2 e^{t}\sin(\pi x)\sin(\pi y),
\end{equation}
one obtains
\begin{equation}
    \nonumber
    f(t,x,y) = e^{t}\,\sin(\pi x)\,\sin(\pi y)\,\big(1+2\pi^2\big).
\end{equation}
For $H^1$-error evaluation, the exact gradient is
\begin{equation}
    \nonumber
    \nabla u(t,x,y) = e^{t}\begin{pmatrix}
        \pi \cos(\pi x)\,\sin(\pi y)\\[2pt]
        \pi \sin(\pi x)\,\cos(\pi y)
    \end{pmatrix}.
\end{equation}

Spatial discretization uses the conforming scalar VEM of order $k$. Theoretical rates are $\|u-u_h\|_{L^2(\Omega)}=\mathcal{O}(h^{k+1})$ and $\|u-u_h\|_{H^1(\Omega)}=\mathcal{O}(h^{k})$; in particular, for $k=1$ one expects $\mathcal{O}(h^{2})$ in $L^2$ and $\mathcal{O}(h)$ in $H^1$.

The chosen time integration schemes are the SSP-RK of third and fourth order as explained in Section \ref{sec:runge_kutta}. Explicit schemes are combined with a lumped (diagonal) mass matrix to avoid global solves and to enforce strong-stability properties under a CFL restriction consistent with the spectral bound for $(\M)^{-1}\K$.

Verification follows the analytical-solution paradigm: the exact solution is known for all $(t,x,y)$, the boundary conditions are satisfied by construction, and the exponential dependence in time provides a nonstationary but smooth benchmark. Errors are measured in the $L^2$-norm and in the $H^1$-seminorm across mesh refinements and element families in order to confirm the predicted VEM convergence rates. For the distorted Q4 and serendipity Q8 meshes, both SSP--RK3 and SSP--RK4 are considered, while for the Voronoi family the results are reported for SSP--RK3. In all cases, the time step is chosen as $\Delta t=\theta h^2$, with $\theta$ below the SSP bound dictated by the spectral estimate for $(\M)^{-1}\K$; cf.~Section~\ref{sec:estimates}. This diffusion-type scaling ensures that temporal errors remain subordinate to spatial errors, so that the observed rates reflect the approximation properties of the lumped VEM scheme. diffusion-type scaling ensures temporal discretization errors remain negligible compared to spatial errors, allowing us to isolate the spatial convergence behavior at $T=1$.

For the distorted quadrilateral family, the convergence plots are shown in Figures~\ref{fig:distorted-h1-convergence} and \ref{fig:distorted-l2-convergence}. The $H^1$-seminorm errors decay at an experimental order of convergence (EOC) close to $1.00$, with the RK3 and RK4 curves nearly indistinguishable and tracking the reference first-order slope. This behavior agrees with the theoretical estimate $\|u-u_h\|_{H^1(\Omega)}=\mathcal{O}(h)$ for the $k=1$ conforming VEM. The $L^2$-norm exhibits the expected second-order decay, with RK4 providing slightly smaller error constants while preserving the same asymptotic rate. The close overlap of both time integrators indicates that, under the chosen $\mathcal{O}(h^2)$ time-step scaling, the total error is dominated by the spatial discretization.

\begin{figure}[H]
    \centering
    \includegraphics[width=0.85\textwidth]{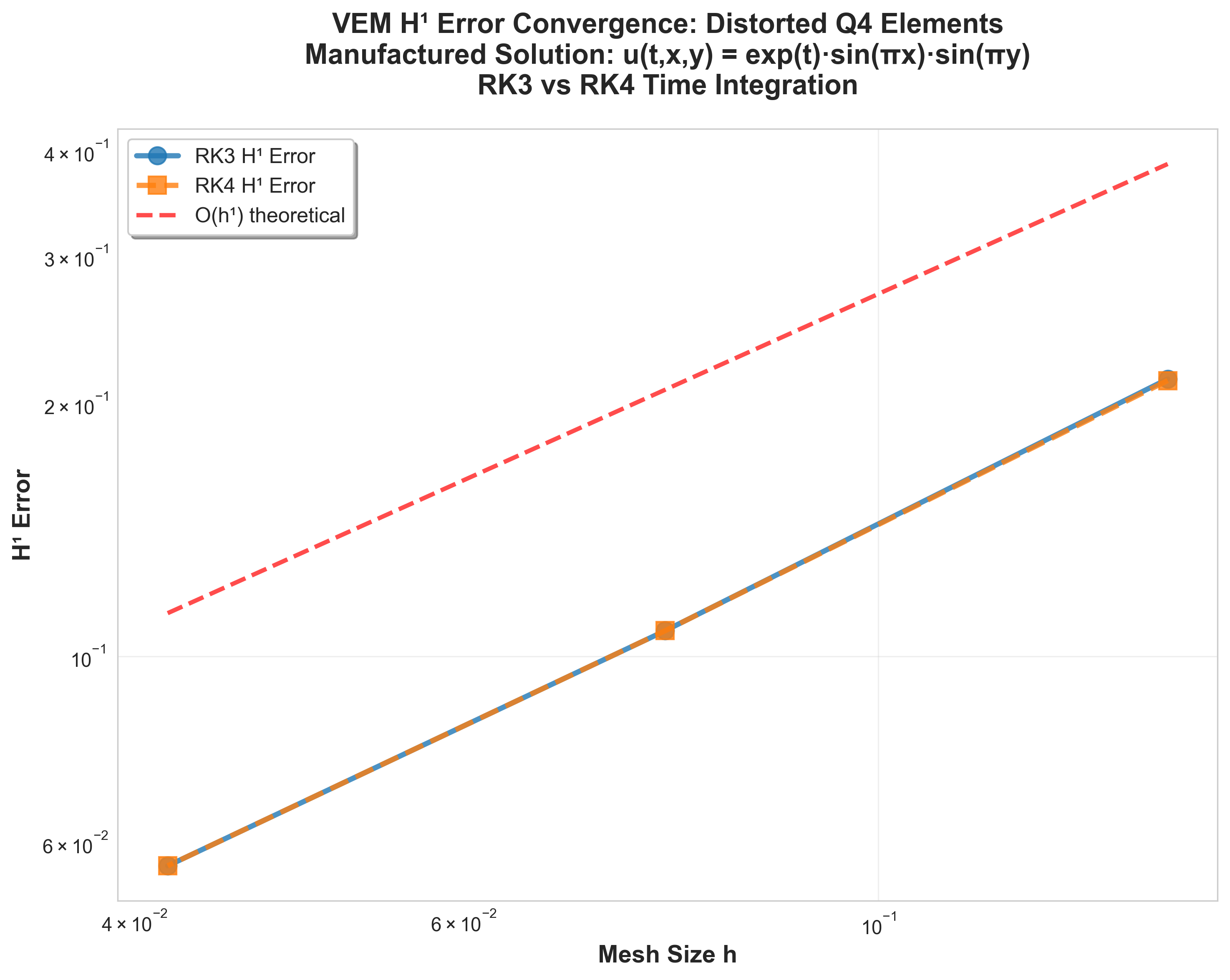}
    \caption{$H^1$-seminorm error convergence for distorted Q4 elements comparing RK3 and RK4 time integration.}
    \label{fig:distorted-h1-convergence}
\end{figure}

\begin{figure}[H]
    \centering
    \includegraphics[width=0.85\textwidth]{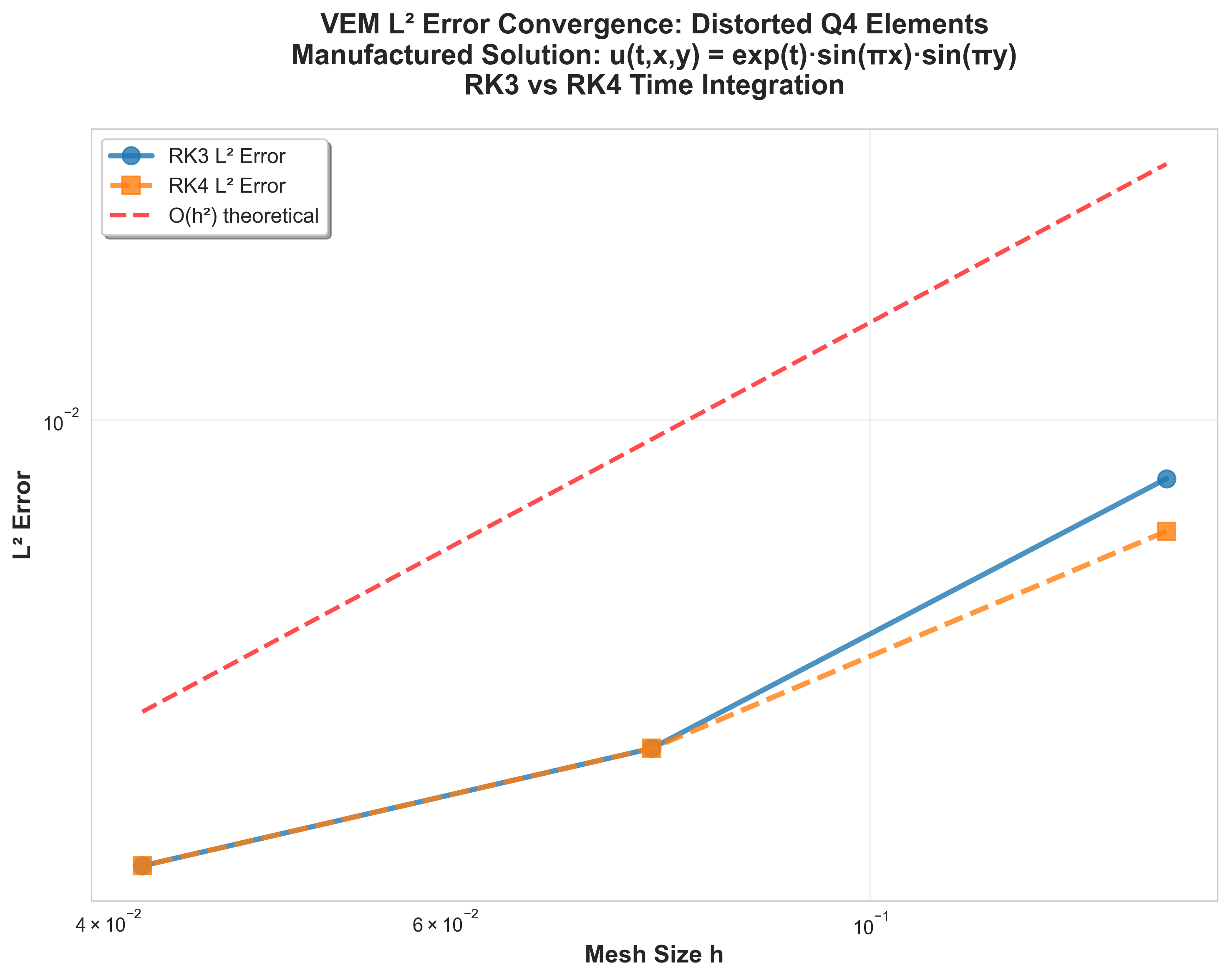}
    \caption{$L^2$-norm error convergence for distorted Q4 elements comparing RK3 and RK4 time integration.}
    \label{fig:distorted-l2-convergence}
\end{figure}

The serendipity family displays the same qualitative behavior, as shown in Figures~\ref{fig:serendipity-h1-convergence} and \ref{fig:serendipity-l2-convergence}. The measured convergence rates remain close to $1.00$ in the $H^1$-seminorm and $2.00$ in the $L^2$-norm, demonstrating that the reduced edge-based degrees of freedom do not compromise the approximation properties of the method. Once again, RK3 and RK4 produce nearly overlapping error curves, with only a mild reduction of the $L^2$ error constants for RK4. The consistency between distorted bilinear elements and serendipity-type elements confirms that the lumped formulation preserves optimal convergence independently of the underlying quadrilateral discretization strategy.

\begin{figure}[H]
    \centering
    \includegraphics[width=0.85\textwidth]{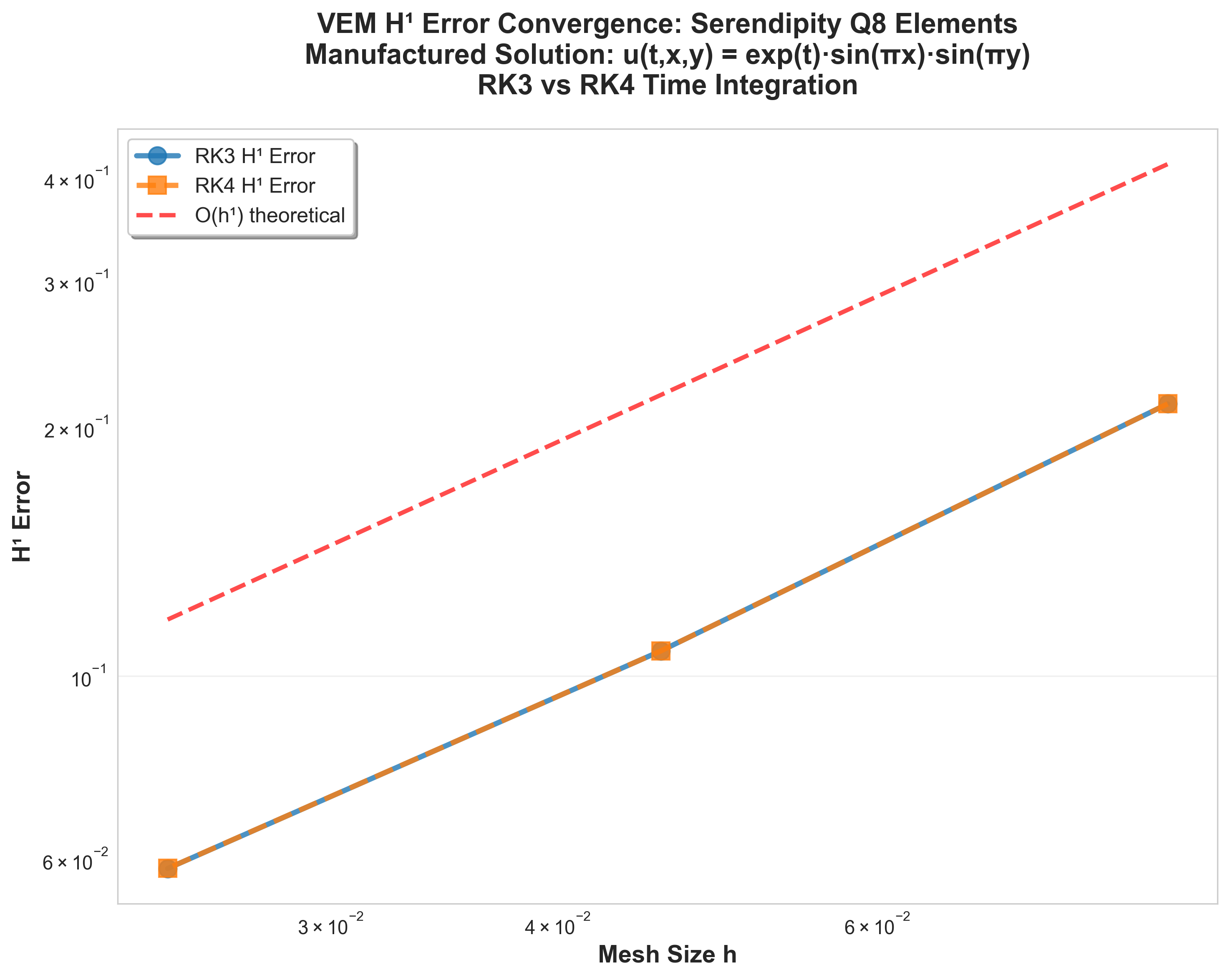}
    \caption{$H^1$-seminorm error convergence for serendipity Q8 elements comparing RK3 and RK4 time integration.}
    \label{fig:serendipity-h1-convergence}
\end{figure}

\begin{figure}[H]
    \centering
    \includegraphics[width=0.85\textwidth]{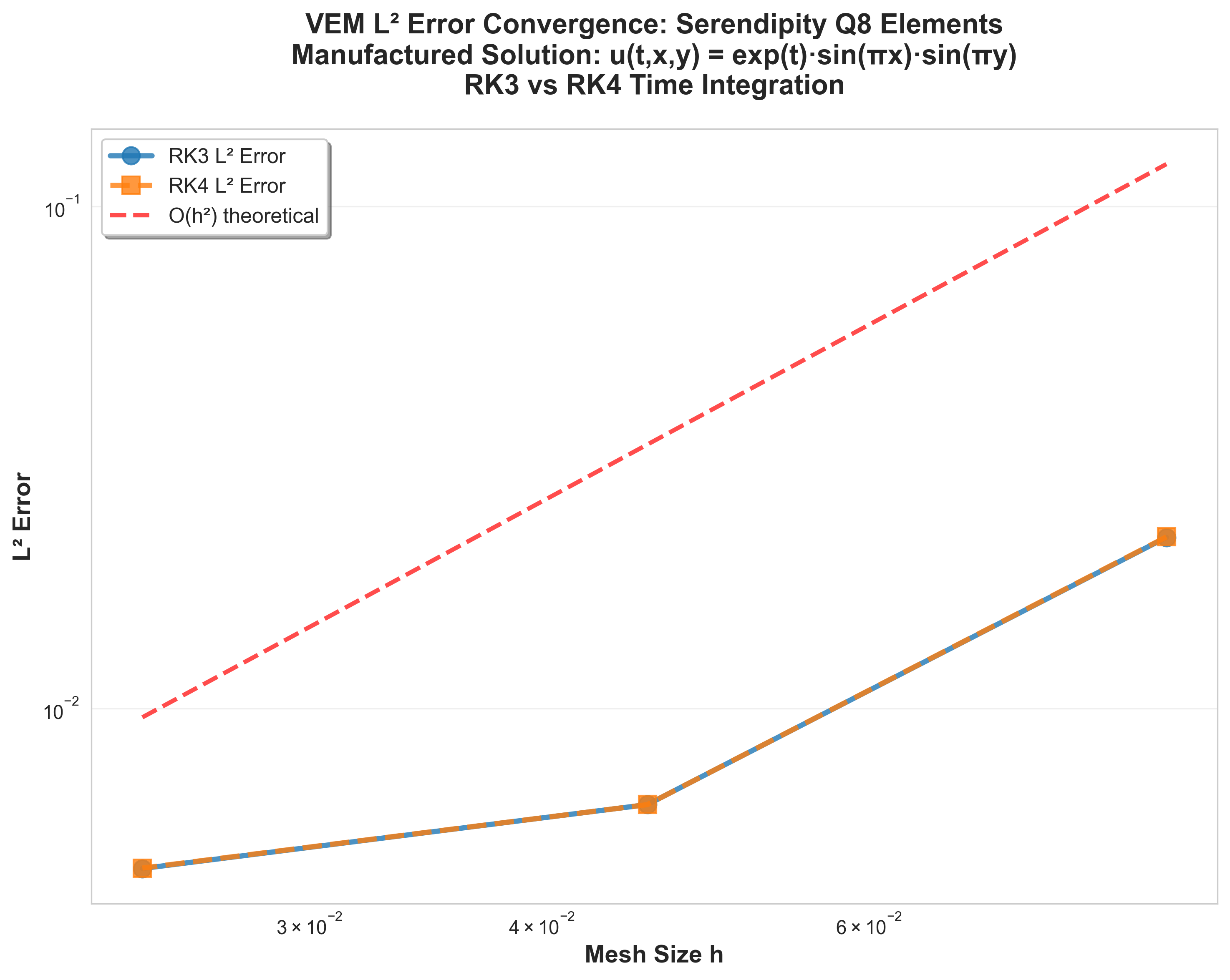}
    \caption{$L^2$-norm error convergence for serendipity Q8 elements comparing RK3 and RK4 time integration.}
    \label{fig:serendipity-l2-convergence}
\end{figure}

The genuinely polygonal Voronoi meshes further confirm the robustness of the method on general element shapes; see Figures~\ref{fig:voronoi-h1-convergence} and \ref{fig:voronoi-l2-convergence}. The $H^1$-seminorm errors decay at a rate close to first order, while the $L^2$-errors exhibit a trend compatible with the theoretical $\mathcal{O}(h^2)$ behavior. Mild pre-asymptotic effects are visible on the coarsest mesh, but the expected rate becomes clearer as the mesh is refined. These results show that mass lumping extends naturally to general polygonal meshes without loss of stability or spatial accuracy.

\begin{figure}[H]
    \centering
    \includegraphics[width=0.85\textwidth]{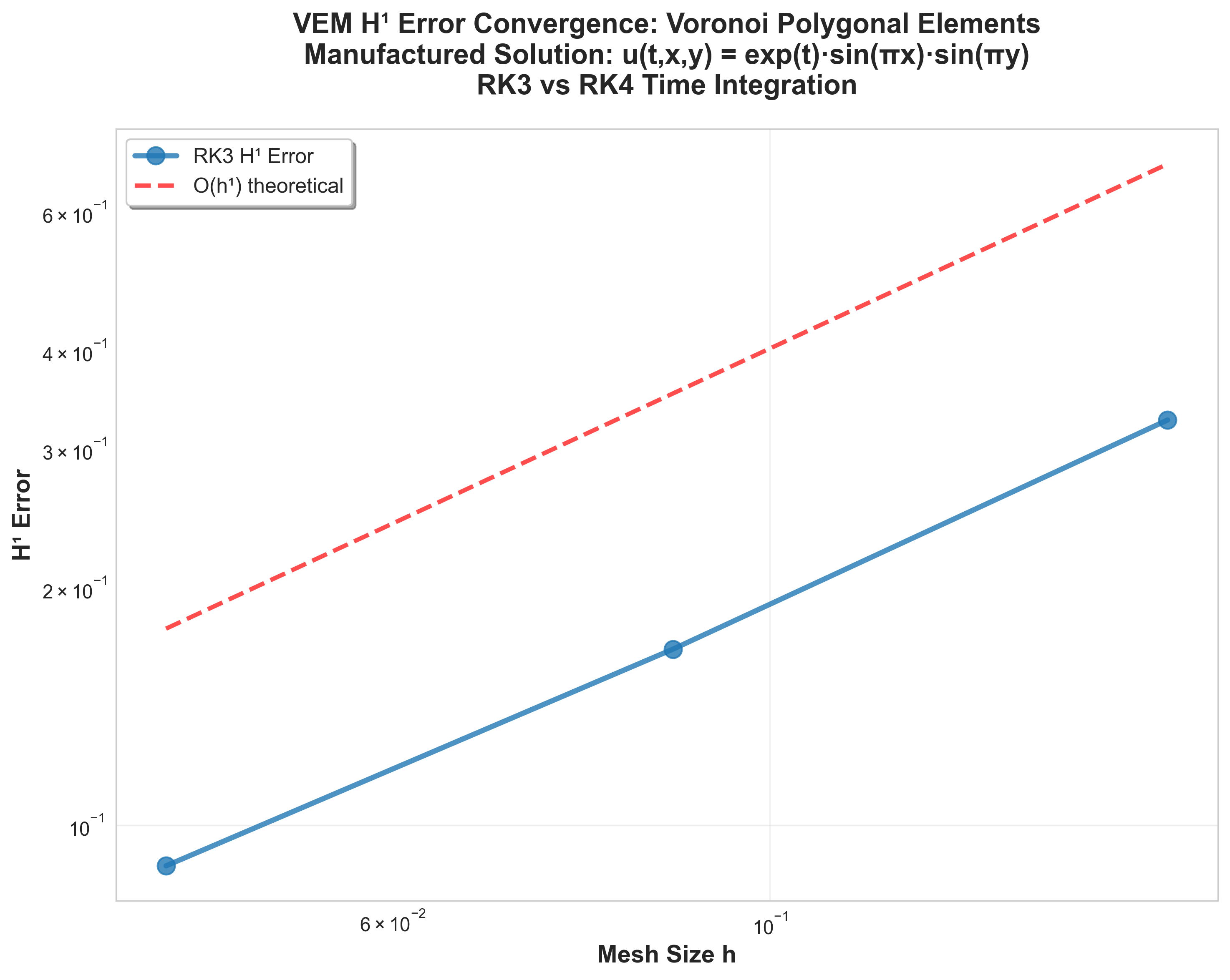}
    \caption{$H^1$-seminorm error convergence for Voronoi polygonal elements using RK3 time integration.}
    \label{fig:voronoi-h1-convergence}
\end{figure}

\begin{figure}[H]
    \centering
    \includegraphics[width=0.85\textwidth]{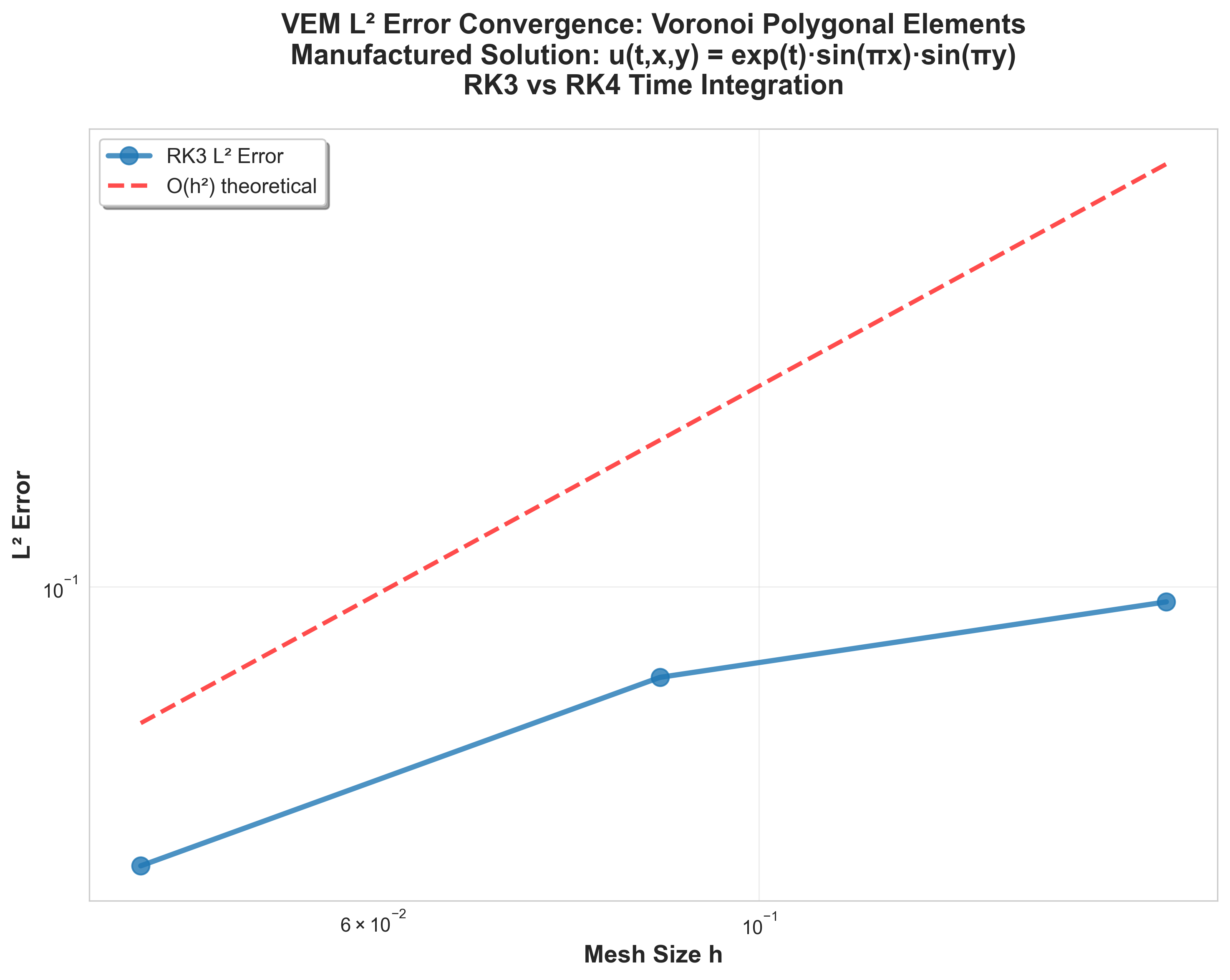}
    \caption{$L^2$-norm error convergence for Voronoi polygonal elements using RK3 time integration.}
    \label{fig:voronoi-l2-convergence}
\end{figure}

In summary, these numerical experiments validate the theoretical estimates across all three mesh families: the lumped VEM with $k=1$ exhibits first-order $H^1$ convergence and an $L^2$ behavior consistent with second-order accuracy on distorted quadrilaterals, serendipity elements, and general Voronoi polygons, with no visible degradation due to geometric distortion, polynomial enrichment strategy, or diagonal mass approximation. The SSP time integrators remain stable under a mesh-robust CFL condition $\Delta t\propto h^{2}$, and while upgrading from RK3 to RK4 improves prefactors in the $L^2$-norm, it does not alter the asymptotic spatial convergence trends.

\subsection{Accuracy and cost comparison with a consistent-mass implicit scheme}

In order to complement the convergence study of Section~\ref{sec:manufactured}, a second benchmark is considered in which the proposed lumped explicit formulation is compared with a standard implicit discretization based on the consistent mass matrix. The objective is not to reassess asymptotic convergence rates, which have already been verified for the manufactured solution of the previous subsection, but rather to examine the practical effect of mass lumping in a setting where the exact solution is available and both accuracy and computational cost can be measured directly. The model problem, the exact solution, and the forcing term are the same as those introduced in Section~\ref{sec:manufactured}. In particular, the benchmark is posed on $\Omega=(0,1)\times(0,1)$ up to the final time $T=1$, and provides an exact reference for the evaluation of the final $L^2$- and $H^1$-errors.

The same conforming scalar virtual element discretization of order $k=1$ is employed on the three mesh families introduced above, namely distorted quadrilateral meshes, serendipity-type quadrilateral meshes, and Voronoi polygonal meshes. The spatial approximation space and the discrete diffusion operator are therefore identical in the two methods under comparison, so that the only difference lies in the treatment of the mass matrix and in the associated time integrator.

Table~\ref{tab:accuracy} reports the final errors and the number of time steps for the two methods on all mesh families. As expected, the explicit SSP--RK3 discretization requires a substantially larger number of time steps, owing to the parabolic restriction $\Delta t_E=\mathcal{O}(h^2)$, whereas the implicit Euler method can be advanced with time steps of order $\Delta t_I=\mathcal{O}(h)$. Nevertheless, the two approaches deliver very similar final errors in both the $L^2$-norm and the $H^1$-seminorm over the whole refinement range. In particular, the $H^1$-errors are essentially indistinguishable for each pair of runs, while the $L^2$-errors remain of the same magnitude and show the same decay trend under mesh refinement. This indicates that, for the selected time-step choices, the temporal discretization error remains subordinate to the spatial error in both cases, and confirms that the lumped explicit formulation preserves the approximation properties already observed in Section~\ref{sec:manufactured}.

The corresponding computational times are summarized in Table~\ref{tab:cost}. For all three mesh families and for all refinement levels, the dominant contribution to the total runtime is the assembly stage, which includes the computation of the local virtual element projections. By contrast, the additional cost associated with the implicit method, namely the sparse factorization and the repeated linear solves, remains comparatively small at the tested problem sizes. On the explicit side, the absence of global solves makes each SSP--RK3 stage inexpensive, but the larger number of time steps leads to a visible, though still secondary, increase in the integration time. Overall, Table~\ref{tab:cost} shows that the practical cost of both approaches is governed primarily by matrix assembly rather than by the time-integration kernel itself.

This comparison is summarized in Table~\ref{tab:efficiency}, where the total runtimes are reported together with the corresponding errors and the measured speedup factor. The results indicate that, at matched accuracy, the overall performance of the two approaches is very similar on the present mesh sequences. For the distorted quadrilateral and serendipity families, the speedup remains essentially equal to one across the refinement range, while for the coarsest Voronoi mesh the implicit method is slightly faster, a difference that disappears on the finer meshes. These data therefore do not suggest a decisive end-to-end timing advantage for either strategy in the present implementation. Rather, they show that the algorithmic simplification introduced by mass lumping, namely the replacement of global linear solves by diagonal explicit updates, is largely offset here by the fact that the dominant cost is common to both methods and lies in the assembly of the virtual element operators. This conclusion is also illustrated in Figure~\ref{fig:error-runtime-comparison}, where the $H^1$-error is plotted against the total wall-clock time for each of the three mesh families. The curves remain very close throughout, showing that the two approaches achieve comparable accuracy at comparable overall computational cost.

\begin{figure}[H]
    \centering
    \includegraphics[width=\textwidth]{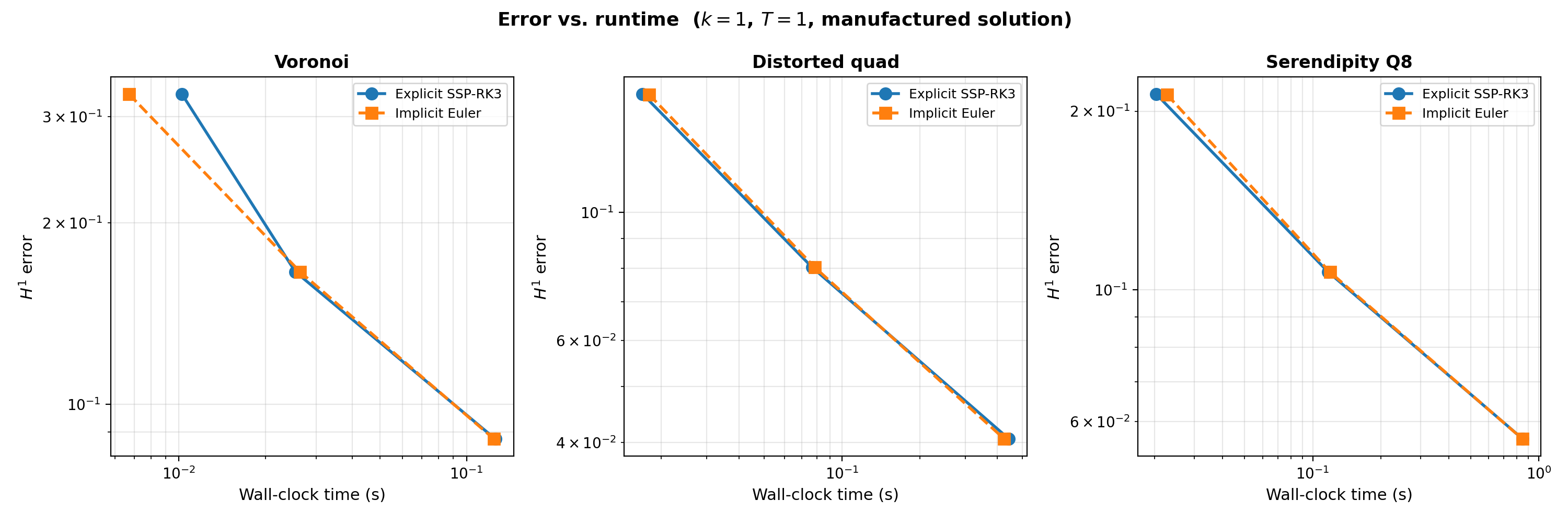}
    \caption{$H^1$-error versus total wall-clock time for the lumped explicit SSP--RK3 scheme and the consistent-mass implicit Euler method on the manufactured benchmark. For all three mesh families, the two methods exhibit very similar error--cost behavior, confirming that at the tested resolutions neither approach yields a decisive practical advantage in end-to-end runtime.}
    \label{fig:error-runtime-comparison}
\end{figure}

Taken together, Tables~\ref{tab:accuracy}--\ref{tab:efficiency} provide a useful practical complement to the convergence study. On the one hand, they confirm that the proposed lumped explicit scheme attains an accuracy comparable to that of a standard consistent-mass implicit discretization on the same benchmark problem. On the other hand, they show that, for the present prototype and for the moderate problem sizes considered here, the structural advantage of explicit diagonal updates is not yet reflected in a significant wall-clock speedup. This observation is consistent with the fact that the comparison is performed in a regime where assembly costs dominate. Even so, the experiment remains relevant, since it isolates the effect of mass lumping in a controlled setting and shows that the proposed formulation achieves the expected accuracy without relying on global linear solves during time integration.

\begin{table}[H]
    \centering
    \caption{Accuracy comparison: lumped explicit SSP-RK3 vs.\ consistent-mass implicit Euler on the manufactured benchmark ($k=1$, $T=1$).}\label{tab:accuracy}
    \begin{tabular}{l c r l r r c c}
    \toprule
    Mesh & $h$ & $N_{\mathrm{dof}}$ & Method & $\Delta t$ & $N_t$ & $\|e\|_{L^2}$ & $|e|_{H^1}$ \\
    \midrule
    Voronoi & 0.1715 & 34 & E (SSP-RK3) & 2.3e-02 & 43 & 5.32e-02 & 3.27e-01 \\
    Voronoi & 0.1715 & 34 & I (Euler) & 8.3e-02 & 12 & 6.28e-02 & 3.27e-01 \\
    Voronoi & 0.0877 & 130 & E (SSP-RK3) & 5.3e-03 & 189 & 3.30e-02 & 1.66e-01 \\
    Voronoi & 0.0877 & 130 & I (Euler) & 4.3e-02 & 23 & 3.40e-02 & 1.66e-01 \\
    Voronoi & 0.0441 & 514 & E (SSP-RK3) & 1.2e-03 & 834 & 1.39e-02 & 8.76e-02 \\
    Voronoi & 0.0441 & 514 & I (Euler) & 2.2e-02 & 46 & 1.35e-02 & 8.77e-02 \\
    \addlinespace
    Distorted quad & 0.1111 & 81 & E (SSP-RK3) & 8.6e-03 & 116 & 3.35e-03 & 1.60e-01 \\
    Distorted quad & 0.1111 & 81 & I (Euler) & 5.6e-02 & 18 & 8.80e-04 & 1.60e-01 \\
    Distorted quad & 0.0588 & 289 & E (SSP-RK3) & 2.1e-03 & 484 & 9.22e-04 & 8.03e-02 \\
    Distorted quad & 0.0588 & 289 & I (Euler) & 2.9e-02 & 34 & 6.52e-04 & 8.03e-02 \\
    Distorted quad & 0.0303 & 1089 & E (SSP-RK3) & 5.0e-04 & 1983 & 4.42e-04 & 4.06e-02 \\
    Distorted quad & 0.0303 & 1089 & I (Euler) & 1.5e-02 & 66 & 4.00e-04 & 4.06e-02 \\
    \addlinespace
    Serendipity Q8 & 0.0867 & 133 & E (SSP-RK3) & 1.0e-02 & 96 & 2.75e-02 & 2.14e-01 \\
    Serendipity Q8 & 0.0867 & 133 & I (Euler) & 4.2e-02 & 24 & 2.42e-02 & 2.13e-01 \\
    Serendipity Q8 & 0.0456 & 481 & E (SSP-RK3) & 2.5e-03 & 405 & 1.06e-02 & 1.07e-01 \\
    Serendipity Q8 & 0.0456 & 481 & I (Euler) & 2.3e-02 & 44 & 9.46e-03 & 1.07e-01 \\
    Serendipity Q8 & 0.0242 & 1703 & E (SSP-RK3) & 6.6e-04 & 1513 & 6.67e-03 & 5.62e-02 \\
    Serendipity Q8 & 0.0242 & 1703 & I (Euler) & 1.2e-02 & 83 & 6.21e-03 & 5.61e-02 \\
    \bottomrule
    \end{tabular}
\end{table}

\begin{table}[H]
    \centering
    \caption{Computational cost breakdown (seconds). Assembly includes VEM projection computation. Factorization applies only to the implicit method.}\label{tab:cost}
    \begin{tabular}{l c l r r r r}
    \toprule
    Mesh & $h$ & Method & Assembly & Factor. & Integration & Total \\
    \midrule
    Voronoi & 0.1715 & E (SSP-RK3) & 0.0095 & --- & 0.0007 & 0.0102 \\
    Voronoi & 0.1715 & I (Euler) & 0.0062 & 0.0004 & 0.0001 & 0.0067 \\
    Voronoi & 0.0877 & E (SSP-RK3) & 0.0216 & --- & 0.0038 & 0.0254 \\
    Voronoi & 0.0877 & I (Euler) & 0.0259 & 0.0002 & 0.0003 & 0.0264 \\
    Voronoi & 0.0441 & E (SSP-RK3) & 0.1024 & --- & 0.0235 & 0.1259 \\
    Voronoi & 0.0441 & I (Euler) & 0.1223 & 0.0009 & 0.0013 & 0.1245 \\
    \addlinespace
    Distorted quad & 0.1111 & E (SSP-RK3) & 0.0149 & --- & 0.0020 & 0.0169 \\
    Distorted quad & 0.1111 & I (Euler) & 0.0177 & 0.0001 & 0.0001 & 0.0179 \\
    Distorted quad & 0.0588 & E (SSP-RK3) & 0.0664 & --- & 0.0102 & 0.0767 \\
    Distorted quad & 0.0588 & I (Euler) & 0.0776 & 0.0004 & 0.0006 & 0.0786 \\
    Distorted quad & 0.0303 & E (SSP-RK3) & 0.3723 & --- & 0.0703 & 0.4426 \\
    Distorted quad & 0.0303 & I (Euler) & 0.4199 & 0.0016 & 0.0033 & 0.4248 \\
    \addlinespace
    Serendipity Q8 & 0.0867 & E (SSP-RK3) & 0.0185 & --- & 0.0018 & 0.0203 \\
    Serendipity Q8 & 0.0867 & I (Euler) & 0.0222 & 0.0001 & 0.0003 & 0.0227 \\
    Serendipity Q8 & 0.0456 & E (SSP-RK3) & 0.1051 & --- & 0.0119 & 0.1169 \\
    Serendipity Q8 & 0.0456 & I (Euler) & 0.1176 & 0.0005 & 0.0011 & 0.1192 \\
    Serendipity Q8 & 0.0242 & E (SSP-RK3) & 0.7592 & --- & 0.0866 & 0.8458 \\
    Serendipity Q8 & 0.0242 & I (Euler) & 0.8347 & 0.0029 & 0.0084 & 0.8459 \\
    \bottomrule
    \end{tabular}
\end{table}

\begin{table}[H]
    \centering
    \caption{Efficiency comparison at matched accuracy. Speedup = runtime(implicit Euler) / runtime(explicit SSP-RK3).}\label{tab:efficiency}
    \begin{tabular}{l c c c c c c c c}
    \toprule
    Mesh & $h$ & $\|e\|_{L^2}^E$ & $\|e\|_{L^2}^I$ & $|e|_{H^1}^E$ & $|e|_{H^1}^I$ & Time$^E$ & Time$^I$ & Speedup \\
    \midrule
    Distorted quad & 0.1111 & 3.4e-03 & 8.8e-04 & 1.6e-01 & 1.6e-01 & 0.017 & 0.018 & 1.1$\times$ \\
    Distorted quad & 0.0588 & 9.2e-04 & 6.5e-04 & 8.0e-02 & 8.0e-02 & 0.077 & 0.079 & 1.0$\times$ \\
    Distorted quad & 0.0303 & 4.4e-04 & 4.0e-04 & 4.1e-02 & 4.1e-02 & 0.443 & 0.425 & 1.0$\times$ \\
    \addlinespace
    Serendipity Q8 & 0.0867 & 2.7e-02 & 2.4e-02 & 2.1e-01 & 2.1e-01 & 0.020 & 0.023 & 1.1$\times$ \\
    Serendipity Q8 & 0.0456 & 1.1e-02 & 9.5e-03 & 1.1e-01 & 1.1e-01 & 0.117 & 0.119 & 1.0$\times$ \\
    Serendipity Q8 & 0.0242 & 6.7e-03 & 6.2e-03 & 5.6e-02 & 5.6e-02 & 0.846 & 0.846 & 1.0$\times$ \\
    \addlinespace
    Voronoi & 0.1715 & 5.3e-02 & 6.3e-02 & 3.3e-01 & 3.3e-01 & 0.010 & 0.007 & 0.7$\times$ \\
    Voronoi & 0.0877 & 3.3e-02 & 3.4e-02 & 1.7e-01 & 1.7e-01 & 0.025 & 0.026 & 1.0$\times$ \\
    Voronoi & 0.0441 & 1.4e-02 & 1.4e-02 & 8.8e-02 & 8.8e-02 & 0.126 & 0.124 & 1.0$\times$ \\
    \bottomrule
    \end{tabular}
\end{table}

\subsection{Robustness study for heterogeneous and anisotropic diffusion}

As a second application, we consider a transient diffusion problem with heterogeneous and anisotropic conductivity. In contrast to the manufactured test of the previous subsection, the objective here is not to assess asymptotic error rates against a known exact solution, but rather to examine the robustness of the lumped explicit VEM discretization in the presence of discontinuous material coefficients, directional diffusion, and nontrivial tensor fields on polygonal meshes. This setting is representative of diffusion processes in composite media, where conductivity may vary abruptly across interfaces and may exhibit preferred directions within each subregion.

More precisely, let $\Omega=(0,1)\times(0,1)$ and $T>0$. We consider the parabolic problem
\begin{equation}
    \nonumber
    \left\{
    \begin{aligned}
        \frac{\partial u}{\partial t} - \nabla\cdot\bigl(K(x)\nabla u\bigr) &= f(t,x,y) && \text{in } \Omega \times (0,T],\\
        u(t,x,y) &= 0 && \text{on } \partial\Omega \times (0,T],\\
        u(0,x,y) &= 0 && \text{in } \Omega,
    \end{aligned}
    \right.
\end{equation}
where the forcing term is chosen as
\begin{equation}
    \nonumber
    f(t,x,y)=e^t\sin(\pi x)\sin(\pi y).
\end{equation}
The diffusion tensor $K(x)$ is assumed symmetric positive definite and piecewise constant, with a discontinuity across the vertical interface
\begin{equation}
    \nonumber
    \Gamma = \{(x,y)\in\Omega : x = 0.5\}.
\end{equation}
Accordingly, the tensor field is defined by
\begin{equation}
    \nonumber
    K(x)=
    \begin{cases}
        K^{-}, & x<0.5,\\[2pt]
        K^{+}, & x\geq 0.5,
    \end{cases}
\end{equation}
with different choices of $K^{-}$ and $K^{+}$ used to generate several benchmark configurations.

The selected test cases are designed to isolate distinct effects of coefficient heterogeneity and anisotropy. The isotropic reference case corresponds to $K^{-}=K^{+}=I$. A first heterogeneous configuration introduces a scalar jump across the interface, with $K^{-}=I$ and $K^{+}=\kappa I$, where $\kappa>1$ denotes the contrast ratio. This case isolates the influence of coefficient discontinuity while preserving isotropy on both sides of the interface. A second configuration considers anisotropy without increasing the maximal eigenvalue, namely
\begin{equation}
    \nonumber
    K^{-}=\mathrm{diag}(1,10^{-2}),
    \qquad
    K^{+}=\mathrm{diag}(10^{-2},1),
\end{equation}
so that the principal diffusion direction changes abruptly across $\Gamma$. Finally, a more demanding mixed case combines rotation, anisotropy, and high contrast by prescribing rotated tensors of the form
\begin{equation}
    \nonumber
    K^{-}=R(\pi/6)\,\mathrm{diag}(1,10^{-2})\,R(\pi/6)^T,
    \qquad
    K^{+}=100\,R(-\pi/4)\,\mathrm{diag}(1,10^{-2})\,R(-\pi/4)^T,
\end{equation}
where $R(\theta)$ denotes the planar rotation matrix of angle $\theta$.

The spatial discretization remains the conforming scalar VEM of order $k=1$, with the same lumped mass construction introduced in the previous sections. No modification of the discrete space is required, since the heterogeneous tensor only enters through the diffusion bilinear form. Time integration is again performed by an explicit SSP Runge--Kutta method, so that the fully discrete scheme retains the advantages of diagonal mass matrices and avoids global linear solves. In this variable-coefficient setting, however, the admissible time step is expected to depend not only on the mesh size but also on the strongest local diffusion scale. In particular, one anticipates a restriction of the form
\begin{equation}
    \nonumber
    \Delta t \lesssim C \,\frac{h^2}{\lambda_{\max}(K)},
\end{equation}
where $\lambda_{\max}(K)$ denotes the largest eigenvalue attained by the diffusion tensor.

Since no closed-form exact solution is available for the discontinuous tensor configurations considered here, this second application is used primarily as a robustness and stability study. The numerical investigation focuses on three complementary aspects: qualitative solution profiles under heterogeneous and anisotropic diffusion, the dependence of the largest stable time step on the coefficient contrast, and the behavior of the normalized stability quantity
\begin{equation}
    \nonumber
    \frac{\Delta t_{\max}\,\lambda_{\max}(K)}{h^2}
\end{equation}
across mesh refinements and mesh families. In this way, the experiment complements the convergence study of the previous subsection by testing the explicit lumped VEM formulation in a more realistic variable-coefficient regime, where discontinuities and directional effects play a central role.

For ease of presentation, the benchmark configurations are identified by short labels in the tables and figures. We denote by \textbf{ISO} the isotropic reference case with $K^{-}=K^{+}=I$, by \textbf{SJ10} and \textbf{SJ100} the scalar-jump cases with $K^{-}=I$ and $K^{+}=10I$ or $K^{+}=100I$, respectively, by \textbf{ANISO} the anisotropy-swap configuration with $K^{-}=\mathrm{diag}(1,10^{-2})$ and $K^{+}=\mathrm{diag}(10^{-2},1)$, and by \textbf{RAJ} the rotated anisotropy with jump, where the diffusion tensors are rotated and scaled as defined above. This notation is used below to improve readability and to facilitate comparison between related test cases.

\begin{table}[htbp]
    \centering
    \caption{Empirical maximum stable time step $\Delta t_{\max}$ for the lumped-mass SSP--RK3 scheme on Voronoi meshes. The case codes are: ISO (isotropic), SJ10 and SJ100 (scalar jump with $\kappa=10$ and $\kappa=100$), ANISO (anisotropy swap), and RAJ (rotated anisotropy with jump).}
    \label{tab:stability_voronoi}
    \begin{tabular}{l c c c c c}
    \toprule
    Case & $h$ & $\lambda_{\max}(K)$ & $\Delta t_{\max}$ & $\Delta t_{\max}/h^2$ & $\Delta t_{\max}\lambda_{\max}/h^2$ \\
    \midrule
    ISO   & 0.1715 & 1.0   & 2.94e-02 & 1.0000 & 1.0000 \\
    ISO   & 0.0877 & 1.0   & 6.61e-03 & 0.8594 & 0.8594 \\
    ISO   & 0.0441 & 1.0   & 1.49e-03 & 0.7656 & 0.7656 \\
    \addlinespace

    SJ10  & 0.1715 & 10.0  & 4.18e-03 & 0.1422 & 1.4219 \\
    SJ10  & 0.0877 & 10.0  & 7.33e-04 & 0.0953 & 0.9531 \\
    SJ10  & 0.0441 & 10.0  & 1.52e-04 & 0.0781 & 0.7813 \\
    \addlinespace

    SJ100 & 0.1715 & 100.0 & 4.23e-04 & 0.0144 & 1.4375 \\
    SJ100 & 0.0877 & 100.0 & 7.33e-05 & 0.0095 & 0.9531 \\
    SJ100 & 0.0441 & 100.0 & 1.52e-05 & 0.0078 & 0.7812 \\
    \addlinespace

    ANISO & 0.1715 & 1.0   & 4.69e-02 & 1.5938 & 1.5938 \\
    ANISO & 0.0877 & 1.0   & 9.38e-03 & 1.2188 & 1.2188 \\
    ANISO & 0.0441 & 1.0   & 2.01e-03 & 1.0312 & 1.0312 \\
    \addlinespace

    RAJ   & 0.1715 & 100.0 & 6.25e-04 & 0.0213 & 2.1250 \\
    RAJ   & 0.0877 & 100.0 & 7.69e-05 & 0.0100 & 1.0000 \\
    RAJ   & 0.0441 & 100.0 & 1.67e-05 & 0.0086 & 0.8594 \\
    \bottomrule
    \end{tabular}
\end{table}

\begin{table}[htbp]
    \centering
    \caption{Cross-mesh stability comparison at comparable $h$. The normalized ratio $\Delta t_{\max}\lambda_{\max}(K)/h^2$ measures the effective CFL number.}
    \label{tab:cross_mesh}
    \begin{tabular}{l l c c c c}
    \toprule
    Case & Mesh & $h$ & $\Delta t_{\max}$ & $\Delta t_{\max}/h^2$ & $\Delta t_{\max}\lambda_{\max}/h^2$ \\
    \midrule
    ISO   & Voronoi          & 0.0877 & 6.61e-03 & 0.8594 & 0.8594 \\
    ISO   & Distorted quad   & 0.1111 & 1.08e-02 & 0.8750 & 0.8750 \\
    ISO   & Serendipity (Q8) & 0.0867 & 1.36e-02 & 1.8125 & 1.8125 \\
    \addlinespace

    SJ10  & Voronoi          & 0.0877 & 7.33e-04 & 0.0953 & 0.9531 \\
    SJ10  & Distorted quad   & 0.1111 & 1.20e-03 & 0.0969 & 0.9687 \\
    SJ10  & Serendipity (Q8) & 0.0867 & 1.43e-03 & 0.1906 & 1.9062 \\
    \addlinespace

    ANISO & Voronoi          & 0.0877 & 9.38e-03 & 1.2188 & 1.2188 \\
    ANISO & Distorted quad   & 0.1111 & 1.20e-02 & 0.9688 & 0.9688 \\
    ANISO & Serendipity (Q8) & 0.0867 & 1.43e-02 & 1.9062 & 1.9062 \\
    \addlinespace

    RAJ   & Voronoi          & 0.0877 & 7.69e-05 & 0.0100 & 1.0000 \\
    RAJ   & Distorted quad   & 0.1111 & 2.24e-04 & 0.0181 & 1.8125 \\
    RAJ   & Serendipity (Q8) & 0.0867 & 1.95e-04 & 0.0259 & 2.5938 \\
    \bottomrule
    \end{tabular}
\end{table}

The qualitative impact of the heterogeneous and anisotropic tensor fields is illustrated in Figure~\ref{fig:heterogeneous-snapshots}, which reports solution snapshots at the final time for the four benchmark configurations. The isotropic reference case exhibits the expected symmetric diffusion pattern, whereas the scalar-jump configuration produces a clear asymmetry across the interface, with stronger diffusion in the high-conductivity region. In the anisotropy-swap case, the solution profile becomes directionally stretched in a manner consistent with the principal axes of the tensor field on each side of the interface. The rotated anisotropy with jump generates the most pronounced distortion, combining directional transport and strong contrast effects. These plots confirm that the proposed formulation captures the qualitative features induced by discontinuous and anisotropic diffusion tensors on general polygonal meshes.

Tables~\ref{tab:stability_voronoi} and \ref{tab:cross_mesh} summarize the empirical stability thresholds of the lumped-mass SSP--RK3 discretization for the benchmark cases ISO, SJ10, SJ100, ANISO, and RAJ. The Voronoi results in Table~\ref{tab:stability_voronoi} show that the largest admissible time step decreases consistently under mesh refinement according to a diffusion-type law of order $h^2$, while its dependence on the material parameters is governed by the largest eigenvalue of the tensor field. In particular, for the scalar-jump cases SJ10 and SJ100, the quantity $\Delta t_{\max}/h^2$ decreases proportionally to $1/\lambda_{\max}(K)$, whereas the normalized indicator $\Delta t_{\max}\lambda_{\max}(K)/h^2$ remains bounded and nearly constant across the three mesh levels. This behavior is fully consistent with the expected restriction $\Delta t \lesssim C h^2/\lambda_{\max}(K)$ and confirms that the lumped explicit formulation preserves the same parabolic CFL structure in the presence of discontinuous coefficients. The anisotropic cases ANISO and RAJ exhibit the same qualitative trend, although with different prefactors reflecting the directional character of the diffusion tensor and the associated mesh-dependent stability constants.

Table~\ref{tab:cross_mesh} complements this analysis by comparing the effective CFL number across three mesh families at comparable values of $h$. For all four benchmark configurations, the quantity $\Delta t_{\max}\lambda_{\max}(K)/h^2$ remains of order one, showing that the stability behavior is robust with respect to the underlying element geometry. The Voronoi and distorted quadrilateral meshes produce closely comparable values, while the serendipity family admits slightly larger stable steps, indicating a more permissive stability constant without altering the overall scaling law. These results support the conclusion that the mass-lumped VEM discretization retains its expected explicit stability behavior not only on general polygonal meshes, but also across different structured and unstructured mesh constructions.

\begin{figure}[H]
    \centering
    \includegraphics[width=\textwidth]{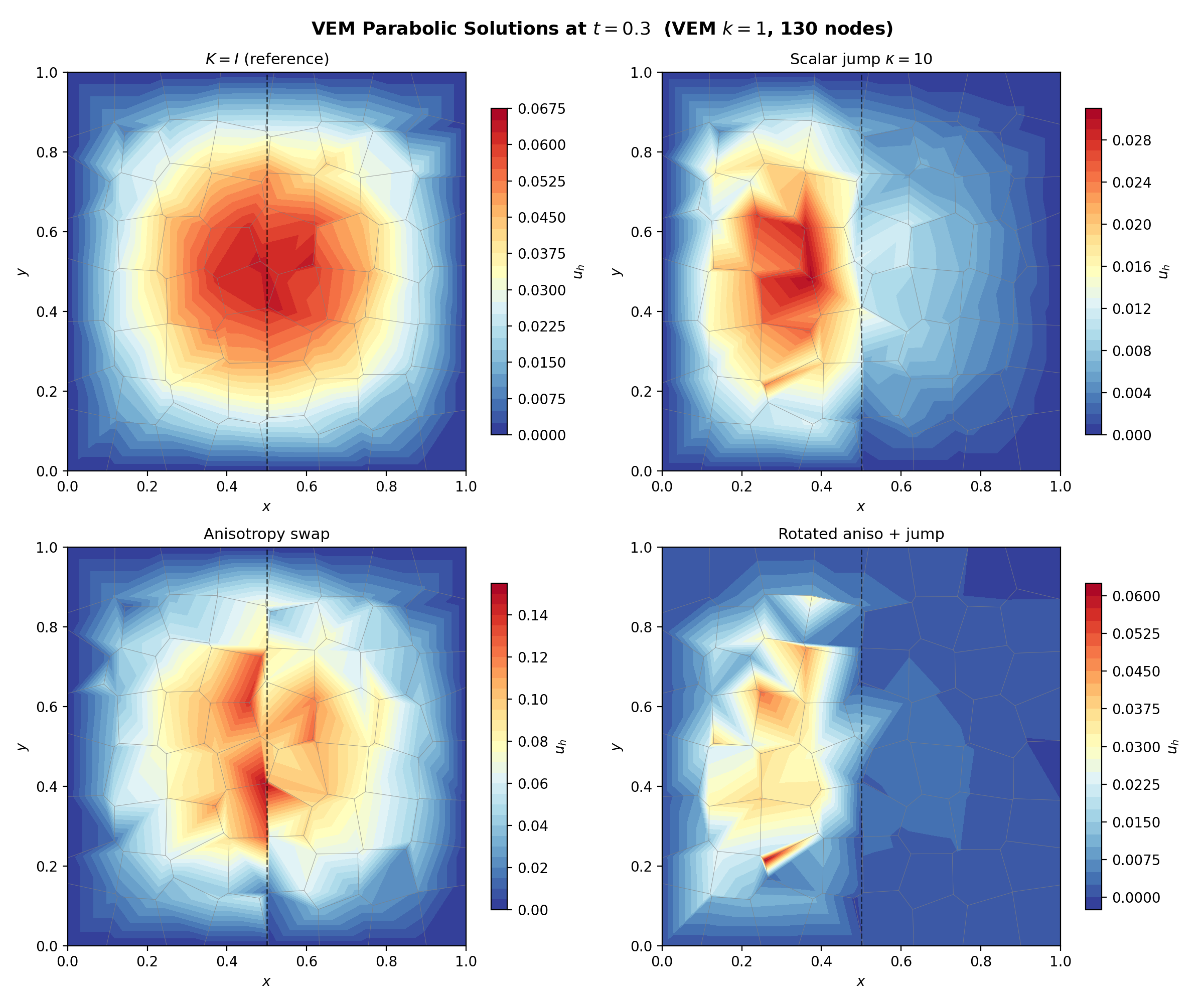}
    \caption{Numerical solutions at final time for the heterogeneous and anisotropic benchmark configurations on a Voronoi mesh. The dashed line indicates the material interface $x=0.5$.}
    \label{fig:heterogeneous-snapshots}
\end{figure}

\begin{figure}[H]
    \centering
    \includegraphics[width=0.8\textwidth]{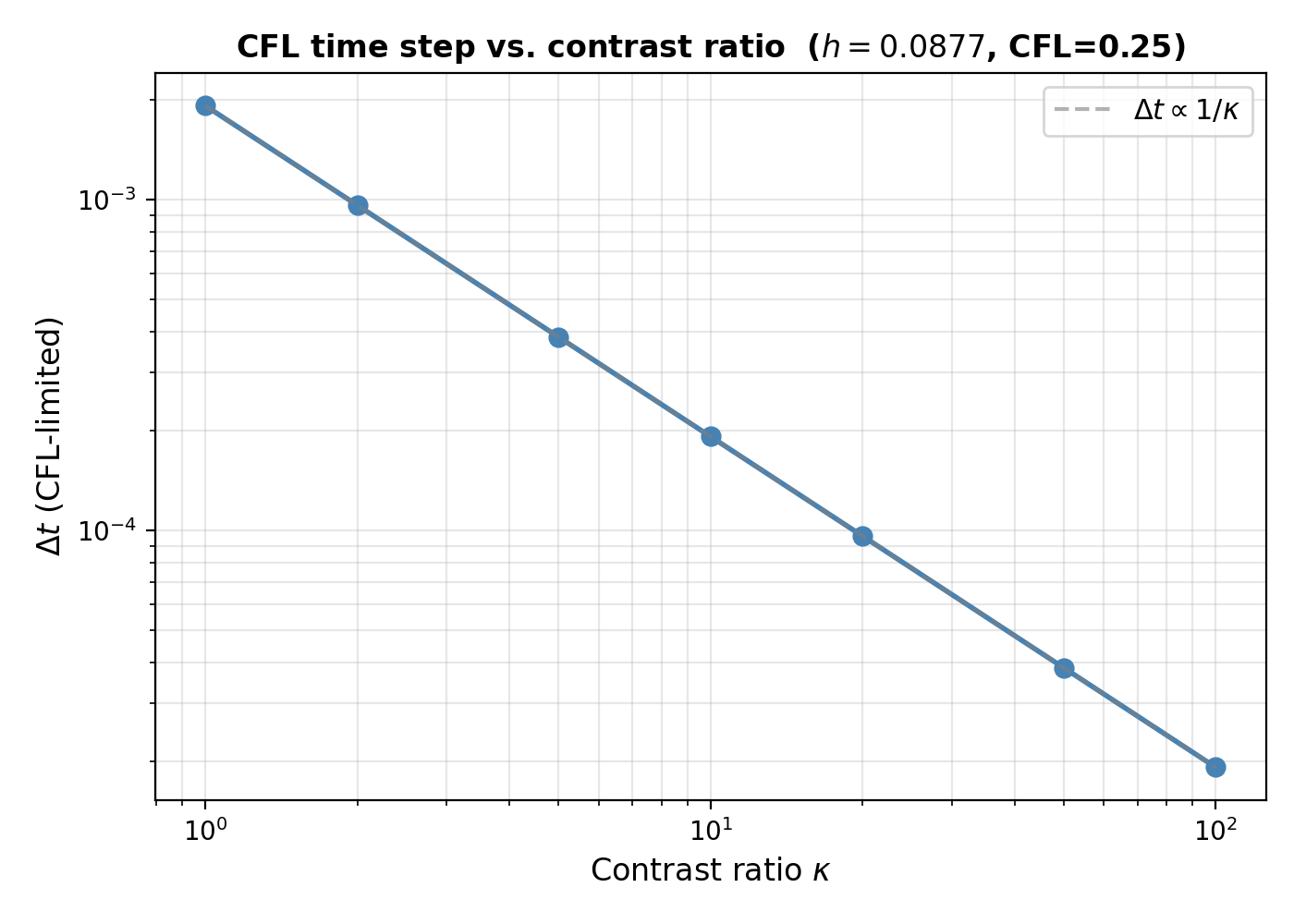}
    \caption{Dependence of the empirical stability threshold on the scalar contrast ratio $\kappa$. The observed behavior is consistent with the scaling $\Delta t_{\max}\propto 1/\kappa$.}
    \label{fig:cfl-sensitivity}
\end{figure}

\section{Conclusion}
\label{sec:conclusion}

This paper has developed and analyzed a mass-lumped Virtual Element Method with explicit Strong Stability-Preserving Runge--Kutta time integration for two-dimensional parabolic problems on general polygonal meshes. The proposed formulation combines VEM spatial discretization, diagonal mass lumping, and SSP-RK temporal schemes into a stable and computationally efficient framework for explicit time integration on arbitrary polygonal grids.

The mass-lumping procedure constructs a diagonal mass matrix by row summation of the consistent VEM mass matrix, complemented by a flooring mechanism ensuring uniform positivity of the lumped weights. A central structural property is that stabilization contributions vanish identically under row summation, so that the lumped weights depend only on the polynomial projection and are obtained through the local system $\mathbf{G}^E\mathbf{w}^E=\mathbf{c}^E$, with computational cost $\mathcal{O}(N_k^3)$ per element and $N_k \ll N_{\mathrm{dof}}^E$. The resulting lumped bilinear form is uniformly equivalent to the consistent $L^2$ inner product, with constants independent of the number of element edges, and therefore provides a symmetric positive definite discrete inner product suitable for explicit time stepping without global mass-matrix inversion.

The analysis yields a mesh-robust spectral estimate of the form
\begin{equation}\nonumber
\lambda_{\max}\!\big((\M)^{-1}\K\big)\ \le\ \frac{C_{\mathrm{inv}}^2}{\hat{\beta}_*}\, h^{-2},
\end{equation}
where the constants depend only on the polynomial degree, the space dimension, and the mesh regularity parameters, but not on the number of element faces or edges. This estimate implies the classical parabolic CFL restriction $\Delta t=\mathcal{O}(h^2)$ for forward Euler stability and extends, through the SSP theory, to higher-order explicit Runge--Kutta methods. In particular, stability properties established at the forward Euler level, such as energy dissipation or monotonicity, are inherited by the multi-stage scheme under the timestep restriction $0<\Delta t \le C_{\mathrm{SSP}}\,\Delta t_{\mathrm{FE}}$, with explicit SSP coefficients for the third-order and fourth-order methods considered.

The numerical investigation confirms the theoretical findings and illustrates the practical behavior of the method on several representative classes of polygonal meshes. The convergence tests on distorted quadrilateral, serendipity-type, and Voronoi families show that the lumped VEM with $k=1$ attains the expected first-order convergence in the $H^1$-seminorm and second-order convergence in the $L^2$-norm, with no visible deterioration due to mesh distortion, element shape, or the diagonal mass approximation. The experiments with higher-order SSP schemes further indicate that the admissible timestep follows the predicted $\Delta t \propto h^2$ scaling, while higher-order time discretizations improve the error constants without altering the spatial asymptotic rates.

The additional numerical examples strengthen the practical assessment of the method in two relevant directions. First, the comparison between the explicit lumped formulation and the corresponding consistent-mass implicit discretization provides an accuracy/efficiency perspective that complements the asymptotic error analysis. It shows that the proposed explicit method remains competitive while retaining the implementation advantages associated with diagonal mass matrices and stage-wise locality. Second, the heterogeneous anisotropic diffusion test confirms that the method preserves its robustness in the presence of spatially varying and strongly directional diffusion tensors, thereby supporting the relevance of the formulation beyond smooth isotropic benchmark problems.

From a computational viewpoint, the method offers a particularly simple and structured explicit solver. The diagonal mass matrix reduces each time update to scaled vector operations and local stiffness applications, avoiding global linear solves and making the procedure naturally parallelizable. This feature is especially attractive on general polygonal meshes, where matrix assembly is already more involved than in standard simplicial settings and where explicit solvers can benefit directly from data locality and hardware acceleration.

Several extensions remain of interest. Variable-coefficient and heterogeneous diffusion operators fit naturally within the projector-based VEM construction, and the present anisotropic tests indicate that this direction is both mathematically and computationally relevant. Reaction--diffusion systems and nonlinear parabolic problems can also benefit from the lumped explicit framework, since the diagonal mass matrix simplifies the evaluation of nonlinear terms and preserves the applicability of SSP time discretizations. Further developments may include operator-splitting strategies for coupled problems, vector-valued formulations, and adaptive polygonal refinement combined with explicit local time-stepping techniques.

Overall, the results show that mass-lumped VEM combined with SSP Runge--Kutta time integration provides a rigorous and effective approach for the explicit approximation of parabolic problems on general polygonal meshes. The combination of optimal spatial convergence, mesh-robust timestep restrictions, computational simplicity, and satisfactory performance in efficiency and anisotropic-diffusion tests makes the method a promising tool for large-scale simulations on complex geometries and nonstandard meshes.

\newpage

\bibliographystyle{unsrt}  


\newpage

\appendix

\section{Results in Mathematics}
\label{ap:results_maths}

\subsection{Rank-nullity}
\label{ap:rank_nullity}

\begin{theo}[Rank-nullity]
    Let $V$ and $W$ be vector spaces over a field $\mathbb{F} = \mathbb{R} \; \text{or} \; \mathbb{C}$ with $\dim(V) < \infty$, and let $T: V \longrightarrow W$ be linear. Then, it holds that:
    \begin{equation}
    \nonumber
        \dim(V) = \dim ( \ker T) + \dim (Im T).
    \end{equation}
    Equivalently, for a matrix $\mathbf{A}\in \mathbb{F}^{m\times n}$, acting as $\mathbf{A}: \mathbb{F}^n \longrightarrow \mathbb{F}^m$,
    \begin{equation}
    \nonumber
        n = nullity(\mathbf{A}) + rank(\mathbf{A}),
    \end{equation} 
    where 
    \begin{equation}
    \nonumber
        nullity (\mathbf{A}) = \dim (\{ \mathbf{x}\in \mathbb{F}^n:\; \mathbf{A} \mathbf{x} = \mathbf{0} \}),
    \end{equation}
    and
    \begin{equation}
    \nonumber
        rank(\mathbf{A}) = \dim (col(\mathbf{A})).
    \end{equation}
\end{theo}

\begin{cor} Under the same conditions as the theorem:
    \begin{enumerate}
        \item $T$ is injective if and only if $\dim (\ker T) = 0$,
        \item $T$ is surjective if and only if $\dim (ImT) = \dim W$.
    \end{enumerate}
\end{cor}

\subsection{Rayleigh quotient}
\label{ap:rayleigh_quotient}

The Rayleigh quotient is used to estimate eigenvalues and conditioning. For a symmetric matrix $\mathbf{A} \in \mathbb{R}^{n \times n}$ and any nonzero $\mathbf{x} \in \mathbb{R}^n$, the Rayleigh quotient is given by:
\begin{equation}\label{eq:rayleigh_quotient}
    R_{\mathbf{A}} (\mathbf{x}) = \frac{\mathbf{x}^T\mathbf{A}\mathbf{x}}{\mathbf{x}^T \mathbf{x}}.
\end{equation}
The key properties from (\ref{eq:rayleigh_quotient}) are:
\begin{enumerate}
    \item $\lambda_{\min} \leq R_{\mathbf{A}}(\mathbf{x}) \leq \lambda_{\max}$;
    \item if $\mathbf{A} \mathbf{x} = \lambda \mathbf{x} \Rightarrow R_{\mathbf{A}}(\mathbf{x}) = \lambda$;
    \item Courant-Fischer characterization:
    \begin{equation}
        \nonumber
        \lambda_{\min} = \min \limits_{\mathbf{x} \neq \mathbf{0}} R_{\mathbf{A}}(\mathbf{x}), \; \lambda_{\max}  = \max \limits_{\mathbf{x} \neq \mathbf{0}} R_{\mathbf{A}}(\mathbf{x})    
    \end{equation}
\end{enumerate}

For a symmetric positive definite $\mathbf{M}$ and a symmetric $\mathbf{K}$, the generalized version of the quotient is:
\begin{equation}
    \nonumber
    R_{\mathbf{K},\mathbf{M}} (\mathbf{x}) = \frac{\mathbf{x}^T \mathbf{K}\mathbf{x}}{\mathbf{x}^T \mathbf{M}\mathbf{x}}.
\end{equation}
Then, $R_{\mathbf{K},\mathbf{M}}$ is bounded by the generalized eigenvalues of $\mathbf{K}\mathbf{\phi} = \lambda \mathbf{M}\mathbf{\phi}$, and
\begin{equation}
    \nonumber
    \lambda_{\min} = \min \limits_{\mathbf{x} \neq \mathbf{0}} R_{\mathbf{K},\mathbf{M}} (\mathbf{x}), \; \lambda_{\max} = \max \limits_{\mathbf{x} \neq \mathbf{0}} R_{\mathbf{K},\mathbf{M}} (\mathbf{x}).
\end{equation}

\subsection{Maximum principle}
\label{ap:maximum_principle}
Let $\Omega$ be an open real set.

\begin{theo}[Maximum principle] \label{theo:maximum_principle}
    Suppose $u \in C^2(\Omega) \cap C(\overline{\Omega})$ is harmonic within $\Omega$. Then,
    \begin{equation}
        \nonumber
        \sup \limits_{\Omega} = \sup \limits_{\partial \Omega} u, \; \inf \limits_{\Omega} u = \inf \limits_{\partial \Omega} u.
    \end{equation}
    In particular, if 
    \begin{equation}
        \nonumber
        m \leq u|_{\partial \Omega} \leq M \; \text{on} \; \partial \Omega,
    \end{equation}
    then
    \begin{equation}
        \nonumber
        m \leq u \leq M \; \text{in} \; \Omega.
    \end{equation}
\end{theo}

\begin{theo}[Strong maximum principle]\label{theo:strong_maximum_principle}
    If $u$ attains its maximum (or minimum) at an interior point of $\overline{\Omega}$, then $u$ is constant in $\Omega$. Equivalently, if $u$ is nonconstant and $u \geq 0$ on $\partial \Omega$, then $u > 0$ in $\Omega$.
\end{theo}

\section{Proof of the theorems}
\label{ap:proofs}

\subsection{VEM classical results}
In this section, some classical results in the Virtual Element Method theory that are necessary to prove theorems presented in this work. The first theorem can be found in \cite{chen2018error}. It establishes an equivalence between the $L^2$ norm on a star-shaped element and a properly scaled degree-of-freedom (DOF) norm for the local VEM space. The constants involved depend only on the dimension, polynomial degree, and mesh chunkiness (and on mild edge-control assumptions), and are independent of the element size and the actual number of edges. This equivalence is a key ingredient for stability estimates and for deriving spectral bounds used in the CFL analysis.
\begin{theo}[DOF-$L^2$ equivalence]\label{theo:l2_equivalence}
    Let $E\in\mathcal{T}_h$ be star-shaped with chunkiness $\gamma>0$ and diameter $h_E$. Let $V_{h,k}(E)$ be the local scalar VEM space of order $k\ge 1$ defined in Section~\ref{sec:virtual_element_space}, with local degrees of freedom $\{\chi_i^E\}_{i=1}^{N^E_{\mathrm{dof}}}$ (vertex values for $k\ge 1$, edge moments up to degree $k-2$ for $k\ge 2$, and interior moments up to degree $k-2$ for $k\ge 2$). Assume a standard edge-control hypothesis (e.g., a uniform bound on the number of edges $N^E_{\mathrm{dof}}\le N_{\max}$, or a minimum edge length $|e|\ge c_e\,h_E$ for every $e\subset\partial E$). Then there exist constants $C_{*}(d,k,\gamma)>0$ and $C^{*}(d,k,\gamma)>0$, depending only on $d$, $k$, $\gamma$ (and on the chosen edge-control hypothesis) but independent of $h_E$ and of the actual number of edges of $E$, such that for all $v_h\in V_{h,k}(E)$
    \begin{equation}
        \nonumber
        C_{*}(d,k,\gamma)\, h_E^{-d}\, \|v_h\|^2_{L^2(E)}\;\le\; \sum_{i=1}^{N^E_{\mathrm{dof}}} \big(\chi_i^E(v_h)\big)^2\;\le\; C^{*}(d,k,\gamma)\, h_E^{-d}\, \|v_h\|^2_{L^2(E)}.
    \end{equation}
    Equivalently, defining the local discrete DOF-norm
    \begin{equation}
        \nonumber
        \|v_h\|^2_{\mathrm{dof},E} := h_E^{-d}\sum_{i=1}^{N^E_{\mathrm{dof}}} \big(\chi_i^E(v_h)\big)^2,
    \end{equation}
    one has
    \begin{equation}
        \nonumber
        (C_{*}(d,k,\gamma))^{1/2}\,\|v_h\|_{L^2(E)}\ \le\ \|v_h\|_{\mathrm{dof},E}\ \le\ (C^{*}(d,k,\gamma))^{1/2}\,\|v_h\|_{L^2(E)}.
    \end{equation}
\end{theo}

Inverse estimates are a classical tool in Galerkin discretizations. Under the VEM shape-regularity assumptions (star-shaped elements with chunkiness $\gamma$ and mild edge-control), the $H^1$-seminorm of a discrete function can be controlled by $h_E^{-1}$ times its $L^2$-norm, with constants depending only on $d$, $k$, and $\gamma$, and independent of the element size and the number of faces/edges. The next theorem states the local and broken/global forms of this inequality for $V_{h,k}$, and is used to bound generalized Rayleigh quotients and to derive mesh-robust CFL conditions. This result is originally presented in \cite{beirao2013vem}.
\begin{theo}[Inverse estimate]\label{theo:inverse_ineq}
    Let $\mathcal{T}_h$ be a VEM shape-regular mesh: every $E\in\mathcal{T}_h$ is star-shaped with chunkiness $\gamma>0$, and either $N_E\le N_{\max}$ uniformly or $|e|\ge c_e\,h_E$ on each edge $e\subset\partial E$. Let $V_{h,k}\subset H^1_0(\Omega)$ be the conforming VEM space of order $k\ge1$. Then there exists a constant $C_{\mathrm{inv}}=C_{\mathrm{inv}}(d,k,\gamma)>0$ (independent of $h$ and of the actual number of faces/edges of each element) such that:
    
    for all $v_h\in V_{h,k}(E)$,
    \[
      |v_h|_{1,E}\ \le\ C_{\mathrm{inv}}\, h_E^{-1}\, \|v_h\|_{0,E},
    \]
    
    and for all $v_h\in V_{h,k}$,
    \[
      |v_h|_{1,\Omega}^2 \;=\; \sum_{E\in\mathcal{T}_h} |v_h|_{1,E}^2
      \;\le\; C_{\mathrm{inv}}^{\,2} \sum_{E\in\mathcal{T}_h} h_E^{-2}\,\|v_h\|_{0,E}^2
      \;\le\; C_{\mathrm{inv}}^{\,2}\, h_{\min}^{-2}\, \|v_h\|_{0,\Omega}^2.
    \]
\end{theo}

\subsection{Proof of theorem \ref{theo:homogeneous_case}}
\label{sec:proof_homogeneous}

Substitute the update into the energy step $n+1$:
\begin{equation}
    \nonumber
    \| \u^{n+1}\|_{\M}^2= (\u^{n+1})^T \M \u^{n+1} = \left( \u^n-\Delta t \M^{-1}\K \u^n \right)^T\M \left( \u^n-\Delta t \M^{-1}\K \u^n \right).
\end{equation}
Since $\M$ is symmetric positive definite, and $\K$ is symmetric positive semi-definite, it is possible to expand the quadratic form such that:
\begin{equation}
    \nonumber
    \| \u^{n+1} \|_{\M} = (\u^n)^T \M \u^n - \Delta t (\u^n)^T \K \u + \Delta t^2 \K \M^{-1} \K \u^n,
\end{equation}
where the cross terms are equal due to the symmetry:
\begin{equation}
    \nonumber
    (\u^n)^T \left( \M \Delta t \M^{-1} \K \u^n \right) = \Delta t (\u^n)^T \K \u^n,
\end{equation}
and similarly for the transpose. Thus, the energy difference is:
\begin{equation}
    \nonumber
    \|\u^{n+1}\|^2_{\M} - \| \u^n \|^2_{\M} = -2 \Delta t (\u^n)^T\K \u^n + \Delta t^2 (\u^n)^T \K \M^{-1}\K \u^n.
\end{equation}
For the difference to be non-positive:
\begin{equation}
    \nonumber
    \Delta t^2 (\u^n)^T \K \M^{-1}\K \u^n \leq 2 \Delta t (\u^n)^T \K \u^n.
\end{equation}
If $(\u^n)^T \K \u^n = 0$ the scheme is stable. For $(\u^n)^T \K \u^n > 0$:
\begin{equation}
    \nonumber
    \Delta t \frac{\u^n \K \M^{-1}\K \u^n}{(\u^n)^T \K \u} \leq 2.
\end{equation}
Note that $(\u^n)^T \K \u^n > 0$ is true for nonconstant modes, as $\K \geq 0$, and it is positive definite on the space orthogonal to constants by the Poincaré's inequality. The fraction is the Rayleigh quotient for symmetric matrix $\K \M^{-1} \K$ in the quadratic form induced by $\K$. Since $\K$ is symmetric positive semi-definite, let 
\begin{equation}
    \nonumber
    \K = \mathbf{Q} \mathbf{\Lambda} \mathbf{Q}^T
\end{equation}
with 
\begin{equation}
    \nonumber
    \mathbf{\Lambda} = \text{diag} (\mu_1, ..., \mu_{N_{\text{dof}}}), \; \mu_j \geq 0, \; \forall j = 1,..., N_{\text{dof}}.
\end{equation}
Then, for $\mathbf{w} = \mathbf{Q}^T \u^n$, the quotient is:
\begin{equation}
    \nonumber
    \frac{\mathbf{w}^T \mathbf{\Lambda}^{1/2} \mathbf{\Lambda}^{1/2} \M^{-1} \mathbf{\Lambda}^{1/2} \mathbf{\Lambda}^{1/2} \mathbf{w}}{\mathbf{w}^T \mathbf{\Lambda} \mathbf{w}} \leq \lambda_{\max} \left( \M^{-1} \K \right)
\end{equation}
as the maximum Rayleigh quotient is $\lambda_{\max}$. Thus, the condition holds if $\Delta t  \leq 2/\lambda_{\max}$. To avoid any notation misunderstanding, $\mu_j$ are the eigenvalues of $\K$, and $\lambda_j$ are the eigenvalues associated to $\M^{-1} \K$. This is uniform over all $\u^n$, ensuring global stability.

\subsection{Proof of theorem \ref{theo:non_homogeneous_case}}
\label{sec:proof_nonhomogeneous}

To shorten notation, define
\begin{equation}
    \nonumber
    E^n = \| \u^n\|^2_{\M} = (\u^n)^T \M \u^n.
\end{equation}
The proof strategy is to show that the energy growth per step is controlled by the forcing term, leading to a bounded situation over time via discrete Grönwall inequality. This ensures stability under the same CFL condition as the homogeneous case. 

The forward Euler step for the nonhomogeneous system is:
\begin{equation}
    \nonumber
    \u^{n+1} = \u^n - \Delta t \M^{-1} \K \u^n + \Delta t \M^{-1}\mathbf{f}_h^n.
\end{equation}
This is an affine update in $\Delta t$, i.e., the homogeneous part is linear in $\Delta t$, and the forcing term adds another linear term. Since the energy $\| \cdot \|_{\M}^2$ is quadratic, the difference $E^{n+1}-E^n$ will be quadratic in $\Delta t$.

Define the following vector:
\begin{equation}
    \nonumber
    \mathbf{g}^n = -\M^{-1}\K \u^n + \M^{-1}\mathbf{f}_h^n,
\end{equation}
so,
\begin{equation}
    \nonumber
    \u^{n+1} = \u^n + \Delta t \mathbf{g}^n.
\end{equation}
The energy difference is
\begin{equation}
    \nonumber
    E^{n+1}-E^n = \| \u^{n+1}\|^2_{\M} - \| \u^{n}\|^2_{\M} = (\u^n+\Delta t \mathbf{g})^T \M (\u^n + \Delta t \mathbf{g}) - (\u^n)^T\M \u^n.
\end{equation}
Since $\M$, by Lemma \ref{lemma:spd}, is symmetric positive definite:
\begin{equation}\label{eq:aux_spd}
    \begin{split}
        E^{n+1}-E^n &= (\u^n)^T\M \u^n + 2 \Delta t (\u^n)^T \M \mathbf{g}^n + \Delta t^2 (\mathbf{g}^n)^T \M \mathbf{g}^n - (\u^n)^T\M \u^n \\
        &= 2 \Delta t (\u^n)^T\M \mathbf{g} + \Delta t^2 (\mathbf{g}^n)^T \M \mathbf{g}^n.
    \end{split}
\end{equation}
It holds that:
\begin{equation}\label{eq:aux_lump_g}
    \M \mathbf{g}^n = - \K \u^n + \mathbf{f}^n_h.
\end{equation}
Substituting (\ref{eq:aux_lump_g}) in (\ref{eq:aux_spd}):
\begin{equation}
    \nonumber
    2 \Delta t (\u^n)^T \M \mathbf{g}^n = 2 \Delta t (\u^n)^T (-\K + \mathbf{f}^n_h) = -2 \Delta t (\u^n)^T \K \u^n + 2 \Delta t (\u^n)^T \mathbf{f}_h^n,
\end{equation}
and
\begin{equation}
    \nonumber
    \Delta t^2 (\mathbf{g}^n)^T \M \mathbf{g}^n = \Delta t^2 (\mathbf{f}_h^n - \K \u^n)^T \M^{-1} (\mathbf{f}_h^n - \K \u^n) = \Delta t^2 \| \mathbf{f}^n_h - \K \u^n \|^2_{\M^{-1}},
\end{equation}
where
\begin{equation}
    \nonumber
    \| \mathbf{v} \|^2_{\M^{-1}} = \mathbf{v}^T \M^{-1} \mathbf{v}.
\end{equation}
Thus,
\begin{equation}\label{eq:aux_energy_difference}
    \nonumber
    E^{n+1} - E^n = -2 \Delta t (\u^n)^T \K \u^n + 2 \Delta t (\u^n)^T \mathbf{f}_h^n + \Delta t^2 \| \mathbf{f}^n_h - \K \u^n \|^2_{\M^{-1}}.
\end{equation}

The terms involving $\K$ are $-2 \Delta t (\u^n)^T \K \u^n$ and $\Delta t^2 (\u^n)^T \K \M^{-1} \K \u^n$. These are non-positive under the CFL condition:
\begin{equation}
    \nonumber
    \Delta t \leq \frac{2}{\lambda_{\max}\left( \M^{-1} \K \right)}.
\end{equation}
This can be seen once $\K$ and $\M$ are symmetric with $\M$ symmetric positive definite (see Lemma \ref{lemma:spd}). Then, there is a $\M$-orthonormal basis $\{ \mathbf{v}_k \}_{k=1}^{N_{\text{dof}}}$, and real eigenvalues $\lambda_k \geq 0$ such that:
\begin{equation}
    \nonumber
    \K \mathbf{v}_k = \lambda_k \M \mathbf{v}_k, \: \mathbf{v}_i \M \mathbf{v}_j = \delta_{ij}.
\end{equation}
Expanding
\begin{equation}
    \nonumber
    \u^n = \sum \limits_k c_k \mathbf{v}_k,
\end{equation}
with $c_k > 0$, it holds true that:
\begin{equation}
    \nonumber
    (\u^n)^T \K \u^n = \sum \limits_k \lambda_k |c_k|^2, \; (\u^n)^T \K \M^{-1} \K \u^n = \sum \limits_k \lambda_k^2 |c_k|^2.
\end{equation}
Hence, the homogeneous part of the energy increment is:
\begin{equation}
    \nonumber
    -2 \Delta t (\u^n)^T \K \u^n + \Delta t^2 (\u^n)^T \K \M^{-1} \K \u^n = \sum \limits_k (-2 \Delta t \lambda_k + \Delta t^2 \lambda_k^2)|c_k|^2 = \sum \limits_k \lambda_k  (-2 \Delta t  + \Delta t^2 \lambda_k)|c_k|^2.
\end{equation}
If the CFL holds, $\Delta t \leq 2/\lambda_{\max}$, then for every $k$:
\begin{equation}
    \nonumber
    -2 \Delta t + \Delta t^2 \lambda_k \leq -2 \Delta t + \Delta t^2 \lambda_{\max} \leq 0,
\end{equation}
so each summand is less or equal to zero. Therefore, the entire sum is less or equal to zero, and it is possible to conclude that:
\begin{equation}
    \nonumber
    -2 \Delta t (\u^n)^T \K \u^n + \Delta t^2 (\u^n)^T \K \M^{-1} \K \u^n \leq 0.
\end{equation}
The energy difference in (\ref{eq:aux_energy_difference}) simplifies to:
\begin{equation}\label{eq:aux_energy_difference_ineq}
    E^{n+1} - E^n \leq 2 \Delta t (\u^n)^T \mathbf{f}^n_h + \Delta t^2 \| \mathbf{f}^n_h \|^2_{\M^{-1}},
\end{equation}
as the mixed term $-2\Delta t^2 (\u^n)^T \K \M^{-1} \mathbf{f}^n_h$ is absorbed into the non-positive part. 

The loading force term is $2\Delta t (\u^n)^T \mathbf{f}^n_h$. Using the Cauchy-Schwarz inequality:
\begin{equation}
    \nonumber
    \left|(\u^n)^T \mathbf{f}^n_h\right| = \left| (\u^n)^T\M^{1/2} \M^{1/2} \mathbf{f}^n_h \right| \leq \| \u^n \|_{\M} \|\mathbf{f}^n_h \|_{\M}.
\end{equation}
This is sharp and mesh independent. Applying the Young's inequality:
\begin{equation}
    \nonumber
    2 \| \u^n \|_{\M} \| \mathbf{f}^n_h \|_{\M^{-1}} \leq \| \u^n \|_{\M} + \| \mathbf{f}^n_h\|_{\M^{-1}}.
\end{equation}
So,
\begin{equation}\label{eq:aux_young_ineq}
    2 \Delta t (\u^n)^T \mathbf{f}_h^n \leq \Delta t \| \u^n \|_{\M} + \Delta t \| \mathbf{f}^n_h \|_{\M^{-1}}.
\end{equation}
Combining (\ref{eq:aux_energy_difference_ineq}) and (\ref{eq:aux_young_ineq}):
\begin{equation}
    \nonumber
    \begin{split}
        \| \u^{n+1} \|_{\M} &\leq \| \u^n \|^2_{\M} + \Delta t \| \u^n \|^2_{\M} + \Delta t \| \mathbf{f}^n_h \|_{\M^{-1}} + \Delta t^2 \| \mathbf{f}^n_h \|_{\M^{-1}} =\\
        &= (1+\Delta t) | \u^n \|^2_{\M}  + (\Delta t + \Delta t^2) \| \mathbf{f}^n_h \|_{\M^{-1}}.
    \end{split}
\end{equation}
This is the per step inequality with growth $\mathcal{O}(\Delta t)$.

\subsection{Proof of corollary \ref{cor:gronwall_consequence}}
\label{sec:proof_gronwall}

Choose
\begin{equation}
    \nonumber
    C_F = \sup \limits_{0 \leq n \leq n_{\max}} \| \mathbf{f}^n_h\|_{\M^{-1}}, 
\end{equation}
where $n_{\max}$ is the maximum number of iterations. The recurrence is given by:
\begin{equation}
    \nonumber
    E^{n+1} \leq (1+\Delta t)E^n + (\Delta t + \Delta t^2)C_F^2.
\end{equation}
Iterating:
\begin{equation}\label{eq:aux_cor_ineq}
    E^n \leq (1 + \Delta t)^n E^0 + (\Delta t + \Delta t^2) C_F^2 \sum \limits^{n-1}_{j=0}(1 + \Delta t)^j.
\end{equation}
The sum is:
\begin{equation}\label{eq:aux_cor_sum}
    \sum \limits^{n-1}_{j=0}(1 + \Delta t)^j = \frac{(1+\Delta t)^n-1}{\Delta t}.
\end{equation}
From (\ref{eq:aux_cor_sum}) in (\ref{eq:aux_cor_ineq}):
\begin{equation}
    \nonumber
    E^n \leq (1+\Delta t)^n E^0 + (1+\Delta t)[(1+\Delta t)^n -1]C_F^2.
\end{equation}
For small $\Delta t$, 
\begin{equation}
    \nonumber
    (1+\Delta t)^n \approx \exp (t_n),
\end{equation}
yielding:
\begin{equation}
    \nonumber
    E^n \leq \exp (t_n) E^0 + (1+ \Delta t)(\exp (t_n) - 1)C_F^2.
\end{equation}

\end{document}